\pdfoutput=1
\documentclass{amsart}

\usepackage{lmodern}
\usepackage[utf8]{inputenc}
\usepackage[mathscr]{eucal}% So-called Euler fonts, allows \mathscr
\usepackage{amssymb}% allows special arrows like >-> e.g.
\usepackage{stmaryrd}
\usepackage[usenames,dvipsnames]{xcolor}
\usepackage{xspace}%allows smart spacing after period
\usepackage{amsthm}% allows theorem* , e.g.
\usepackage{bbold}% allows \unit \mathbb{1}
\usepackage[normalem]{ulem}% allows \sout to cross-out text.
\usepackage{float}
\usepackage{adjustbox}
\usepackage{enumitem}% allows easy change of label
\setlist{leftmargin=*, wide, labelindent=0pt}
\setlist[enumerate]{label*=(\alph*),ref=\alph*}

\usepackage{array}% allows array
\usepackage{amsmath}% allows pmatrix
\usepackage{wasysym}% allows LEFTCIRCLE
\usepackage{etoolbox}% allows for \ifblank
\usepackage{mathdots}% allows diagonal dots
\usepackage{attrib}
\usepackage[style=alphabetic,maxbibnames=99,backend=biber]{biblatex}
\usepackage[disable]{todonotes}

\numberwithin{equation}{section}% makes equat numb contain the section
\usepackage[all]{xy}
\SelectTips{cm}{}
\newdir{ >}{{}*!/-10pt/\dir{>}}
\usepackage{xr-hyper}
\usepackage[colorlinks=true,linkcolor={Black},citecolor={Black},urlcolor={Black}]{hyperref}
\usepackage[nameinlink]{cleveref}% allows references via \cref and \Cref which automatically repeat the name (Thm, Prop, etc) of the ref.

\newtheorem{thmx}{Theorem}

\crefname{Thm}{Theorem}{Theorems}
\crefname{Rem}{Remark}{Remarks}
\crefname{Prop}{Proposition}{Propositions}
\crefname{Cor}{Corollary}{Corollaries}
\crefname{Cons}{Construction}{Constructions}
\crefname{Exa}{Example}{Examples}
\crefname{Lem}{Lemma}{Lemmas}
\crefname{Rec}{Recollection}{Recollections}

\def\makeautorefname#1#2{\expandafter\def\csname#1autorefname\endcsname{#2}}
\def\equationautorefname~#1\null{(#1)\null}
\makeautorefname{footnote}{footnote}%
\makeautorefname{item}{item}%
\makeautorefname{figure}{Figure}%
\makeautorefname{table}{Table}%
\makeautorefname{part}{Part}%
\makeautorefname{appendix}{Appendix}%
\makeautorefname{chapter}{Chapter}%
\makeautorefname{section}{Section}%
\makeautorefname{subsection}{Section}%
\makeautorefname{subsubsection}{Section}%
\makeautorefname{theorem}{Theorem}%
\makeautorefname{corollary}{Corollary}%
\makeautorefname{lemma}{Lemma}%
\makeautorefname{proposition}{Proposition}%
\makeautorefname{pro}{Property}
\makeautorefname{conj}{Conjecture}%
\makeautorefname{definition}{Definition}%
\makeautorefname{notn}{Notation}
\makeautorefname{notns}{Notations}
\makeautorefname{rem}{Remark}%
\makeautorefname{quest}{Question}%
\makeautorefname{example}{Example}%
\theoremstyle{plain}  % default
\newtheorem{theorem}{Theorem}[section]
\newtheorem{proposition}{Proposition}[section]
\newtheorem{lemma}{Lemma}[section]
\newtheorem{corollary}{Corollary}[section]

\newtheorem*{Conj*}{Conjecture}

\theoremstyle{definition}
\newtheorem{definition}{Definition}[section]
\newtheorem{notation}{Notation}[section]
\newtheorem{example}{Example}[section]
\newtheorem{remark}{Remark}[section]

\newtheorem{construction}{Construction}[section]
\makeatletter
\let\c@corollary=\c@theorem
\let\c@proposition=\c@theorem
\let\c@strategy=\c@theorem
\let\c@warning=\c@theorem
\let\c@construction=\c@theorem
\let\c@notation=\c@theorem
\let\c@lemma=\c@theorem
\let\c@definition=\c@theorem
\let\c@example=\c@theorem
\let\c@remark=\c@theorem
\numberwithin{equation}{section}

\newcommand{\nc}{\newcommand}
\nc{\dmo}{\DeclareMathOperator}

\dmo{\coker}{coker}
\dmo{\cone}{cone}
\dmo{\Der}{D}% ground notation for derived categories
\dmo{\DAM}{\DAMbig^{\geom}}%
\dmo{\DAMbig}{DAM}%
\dmo{\Ext}{Ext}
\dmo{\Gal}{Gal}
\dmo{\Hm}{H}% (co)homology
\dmo{\Hom}{Hom}
\dmo{\Id}{Id}
\dmo{\Ind}{Ind}
\dmo{\Infl}{Infl}
\dmo{\Ker}{Ker}
\dmo{\Mod}{Mod}% modules
\dmo{\opname}{op}
\dmo{\perm}{perm}% finite permutation modules
\dmo{\Perm}{Perm}% permutation modules
\dmo{\SH}{SH}% ground name for cat of compact spectra
\dmo{\SHmot}{SH^{\mathrm{c}}_{\bbA^{\!1}}}% ground name for cat of compact motivic spectra
\dmo{\Spc}{Spc}
\dmo{\Spec}{Spec}
\dmo{\Spech}{\Spec^{h}}
\dmo{\stab}{stab}% stable category of fin. gen. mod.
\dmo{\Stab}{Stab}% stable category of non-fin. gen. mod.
\dmo{\supp}{Supp}
\dmo{\supph}{\supp^{h}}
\dmo{\thick}{thick}
\dmo{\Locname}{Loc}% for localizing subcat (w/o tensor)

\nc{\Inj}{\mathrm{Inj}}% injective complexes
\nc{\cat}[1]{\mathscr{#1}}%or: \nc{\cat}[1]{\mathcal{#1}}
\nc{\cK}{\cat{K}}
\nc{\cL}{\cat{L}}
\nc{\colim}{\mathop{\mathrm{colim}}}
\nc{\cofib}{\mathop{\mathrm{cofib}}}
\nc{\cP}{\cat{P}}
\nc{\cQ}{\cat{Q}}
\nc{\cS}{\cat{S}}
\nc{\cT}{\cat{T}}
\nc{\CAlg}{\mathrm{CAlg}}
\nc{\Algn}{\mathrm{Alg}_{\bb{E}_n}}
\nc{\eg}{{\sl e.g.}\@\xspace}
\nc{\gp}{\mathfrak{p}}% prime p
\nc{\gq}{\mathfrak{q}}% prime q
\nc{\hook}{\hookrightarrow}
\nc{\ie}{{\sl i.e.}\@\xspace}
\nc{\into}{\mathop{\rightarrowtail}}
\nc{\inv}{^{-1}}
\nc{\kk}{k}%{\Bbbk}
\nc{\kkG}{\kk G}% common misprint
\nc{\Loc}[1]{\Locname(#1)}
\nc{\loccit}{{\sl loc.\ cit.}\xspace}
\nc{\Mid}{\,\big|\,}
\nc{\onto}{\mathop{\twoheadrightarrow}}
\nc{\op}{^{\opname}}
\nc{\sminus}{\smallsetminus}
\nc{\potimes}[1]{^{\otimes #1}}% tensor power
\nc{\sbull}{{\scriptscriptstyle\bullet}}
\nc{\SET}[2]{\big\{\,#1\Mid#2\,\big\}}
\nc{\unit}{\mathbb{1}}% unit for \otimes
\newcommand{\bb}[1]{\mathbb{#1}}

\newcommand{\mc}[1]{\mathcal{#1}}
\newcommand{\mf}[1]{\mathfrak{#1}}
\newcommand{\mo}[1]{\operatorname{#1}}
\newcommand{\bdot}{\text{\textbullet}}
\nc{\W}{\mathbb{W}}
\nc{\ho}{\mathrm{ho}}
\dmo{\sta}{sta}% for the Rickard functor and the quotient-to-stab functors
\nc{\rsd}[1]{\mathrm{rsd}_{#1}}% for residue field functors

\dmo{\End}{End}
\nc{\isoto}{\overset{\sim}{\,\to\,}}
\nc{\isofrom}{\overset{\sim}{\,\leftarrow\,}}
\let\le=\leqslant% to have "nicer" \le

\nc{\sto}{\rightsquigarrow}
\nc{\xisoto}[1]{\xrightarrow[\sim]{#1}}
\nc{\xto}[1]{\xrightarrow{#1}}
\nc{\xfrom}[1]{\xleftarrow{#1}}
\nc{\xinto}[1]{\overset{#1}{\,\into\,}}
\nc{\xonto}[1]{\overset{#1}{\,\onto\,}}
\nc{\lto}{\leftarrow}
\nc{\normaleq}{\trianglelefteqslant}

\nc{\normal}{\lhd}
\nc{\normaleop}{\mathop{\mathring{\trianglelefteqslant}}}%{\trianglelefteqslant^{\textup{o}}}
\dmo{\chara}{char}%
\dmo{\CoInd}{CoInd}
\dmo{\DMbig}{DM}% Voevodsky's triangulated category of motives
\dmo{\id}{id}
\dmo{\Img}{Im}
\dmo{\im}{im}
\dmo{\Komp}{K}% ground notation for htpy categories
\dmo{\proj}{proj}% f.g. projective modules
\dmo{\rmH}{H}
\dmo{\Res}{Res}
\dmo{\smallb}{b}% ground exponent for ``bounded''
\dmo{\geom}{gm}% ground exponent for ``compact''
\dmo{\stabname}{stab}% stable category of fin. gen. mod.
\dmo{\comp}{comp}
\dmo{\Supp}{Supp}
\dmo{\kosname}{kos}
\dmo{\subname}{Sub}

\nc{\Sub}[1]{\subname_{#1}}
\nc{\Weyl}[2]{{#1}/\!\!/{#2}}%{\overline{#1}_{#2}}%{W_{\!#1}{#2}}% sometimes Weyl is N_G(H)/C_G(H) not mod H.
\nc{\WGH}{\Weyl{G}{H}}
\nc{\tInd}{{}^{\otimes\!}\Ind}% tensor-induction
\nc{\inn}{;}% could be replaced by {\le} or {,} or {;}
\nc{\Vee}[1]{V_{#1}}%
\nc{\Loctens}[1]{\Locname_{\otimes}(#1)}
\nc{\SpcKG}{\Spc(\cK(G))}% most used
\nc{\SpcKGk}{\Spc(\cK(G;\kk))}% most used
\nc{\SpcKE}{\Spc(\cK(E))}% most used
\nc{\Rall}{\rmH^{\sbull\sbull}}%{\rmH_{\scriptscriptstyle\blacktriangle}^{\sbull}}
\nc{\EA}[2]{\mathcal{E}_{#1}(#2)}
\nc{\EApp}[1]{\EA{p}{#1}}

\nc{\kos}[2][]{\kosname_{#1}(#2)}% syntax \kos[ambient group]{subgroup}
\nc{\sKG}{\kos[G]{K}}% most used

\nc{\Zp}{\hat{\bbZ}_p}% p-adic integers

\nc{\adh}[1]{\overline{#1}}% adherence
\nc{\adj}{\dashv}
\nc{\apriori}{{\sl a priori}\xspace}
\nc{\bs}{\backslash}
\nc{\cf}{{\sl cf.}\ }
\nc{\Db}{\Der_{\smallb}}% derived bounded category%\scriptscriptstyle
\nc{\D}{\Der}% derived category
\nc{\FFsep}{\overline{\FF}}
\nc{\Fp}{\bbF_{\!p}}% finite field with p elements
\nc{\gm}{\mathfrak{m}}% prime m
\mathchardef\mhyphen="2D
\nc{\ideal}[1]{\langle #1\rangle}
\nc{\Kb}{\Komp_{\smallb}}% htpy bounded category%\scriptscriptstyle
\nc{\K}{\Komp}% htpy category
\nc{\leop}{\mathop{\mathring{\le}}}%{\le^{\textup{o}}}
\nc{\Lotimes}{\otimes^{\rmL}}
\nc{\To}{\Rightarrow}
\nc{\cSpec}{\mathsf{Spec}}
\nc{\Top}{\mathsf{Top}}
\nc{\Sp}{\mathsf{Sp}^{\omega}}
\nc{\ttCat}{2-\mo{Ring}^{\text{rig}}}
\nc{\FrEinfz}{\mo{Free}_{\bb{E}_{\infty}/\bb{E}_0}}
\nc{\X}[2]{X^{#1,#2}} %macro for whatever notation means X^{otimes i} otimes X^{vee otimes j}
\nc{\A}{\bb{A}^1} %macro for the free thing, possibly to be modified later
\nc{\Apoi}{\bb{A}^{1,+}} 
\nc{\hCob}{\bb{Q}[\mo{hCob}^{1d,or}]_{\eta}}
\nc{\hCobpoi}{\bb{Q}[\mo{hCob}^{1d,or}_{+}]_{\eta}}
\nc{\Speccons}{\mo{Spec}^{\mo{cons}}}
\newcommand{\Seg}{\mathrm{Seg}}
\newcommand{\Fin}{\mathrm{Fin}}
\newcommand{\inj}{\mathrm{inj}}
\newcommand{\tr}{\mathrm{tr}}
\let\theoldbibliography\thebibliography
\renewcommand{\thebibliography}[1]{%
	\theoldbibliography{#1}%
	\setlength{\parskip}{0ex}
	\setlength{\itemsep}{0.5ex plus 0.2ex minus 0.2ex}
	\small
}
\apptocmd{\thebibliography}{\raggedright}{}{}
\font\maljapanese=dmjhira at 2.5ex
\newcommand{\yo}{\textrm{\!\maljapanese\char"48}}

\usepackage{tikz}

\usetikzlibrary{patterns}
\usepackage{tikz-cd}
\usepackage{mathtools}
\usepackage{xspace}
\makeatletter
\let\ea\expandafter
\newcount\foreachcount

\def\foreachLetter#1#2#3{\foreachcount=#1
	\ea\loop\ea\ea\ea#3\@Alph\foreachcount
	\advance\foreachcount by 1
	\ifnum\foreachcount<#2\repeat}

\def\definebb#1{\ea\gdef\csname #1#1\endcsname{\ensuremath{\mathbb{#1}}\xspace}}
\foreachLetter{1}{27}{\definebb}

\makeatother

\date{\today}
\author{Tobias Barthel, Logan Hyslop, Maxime Ramzi}

\address{Tobias Barthel, Max Planck Institute for Mathematics, Vivatsgasse 7, 53111 Bonn, Germany}
\email{tbarthel@mpim-bonn.mpg.de}
\urladdr{https://sites.google.com/view/tobiasbarthel/}

\address{Logan Hyslop, Harvard University Department of Mathematics, 1 Oxford St, Cambridge, MA 02138, United States}
\email{loganrhyslop@gmail.com}
\urladdr{https://loganhyslop.github.io/}

\address{Maxime Ramzi, FB Mathematik und Informatik, Universität Münster, Einsteinstrasse 62}
\email{maxime.ramzi@uni-muenster.de}
\urladdr{https://sites.google.com/view/maxime-ramzi-en/home}

\hypersetup{pdfauthor={Tobias Barthel, Logan Hyslop, Maxime Ramzi},pdftitle={Geometric Points in Higher Zariski Geometry}}

\bibliography{Biblio.bib}
\begin{document}
	%------------------------------------------------------------------------------

    \title{Geometric Points in tensor triangular geometry}
    
	\begin{abstract}
    In this paper, we study geometric points in tensor triangular geometry. In doing so, we construct a counter-example to Balmer's Nerves of Steel conjecture using free constructions in higher Zariski geometry. We then go on to introduce and discuss constructible spectra in the context of tensor triangular geometry. For tensor triangulated categories satisfying a mild enhancement condition, we use these spectra to construct geometric incarnations of (homological or triangular) primes via maps to ``pointlike'' tensor triangulated categories.
	\end{abstract}

	\maketitle
 \vspace{-5pt}
 \begin{figure}[h]
   \includegraphics[scale=1.8]{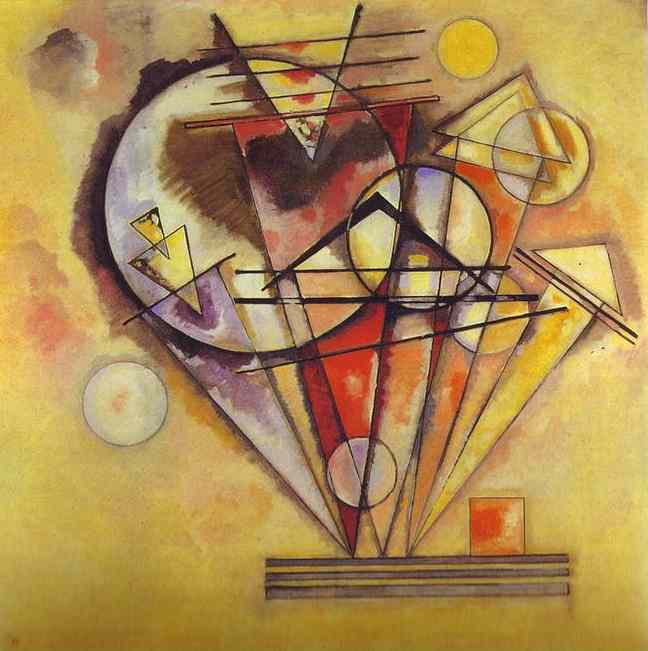}
   \caption[Cover Caption]{\emph{On the points,} Wassily Kandinsky, 1928\footnotemark}
 \end{figure}
\footnotetext{\url{https://www.wikiart.org/en/wassily-kandinsky/on-the-points-1928} (visited on 03/13/2026)}  
\newpage
\tableofcontents

\vspace{-1.4cm}
	\section{Introduction}

\subsection{Context and main results}

Hilbert's Nullstellensatz led to the idea that the prime ideals of a commutative ring $R$ should correspond to the points of the affine scheme $\Spec(R)$, thus forming the origin of modern algebraic geometry. One categorical level up, tensor triangular geometry aims to develop an analogous geometric theory by replacing commutative rings by tensor triangulated categories (henceforth, tt-categories), see \cite{Balmer2005,aoki2025higherzariskigeometry}. The role of the Zariski spectrum for a tt-category $\cat C$ is played by the Balmer spectrum $\Spc(\cat C)$, which is built from prime tt-ideals in $\cat C$ together with a suitable topology akin to the Zariski topology.  
This motivates the question, first raised in \cite{BalmerICM2010}:
    \vspace{0.5em}
    \begin{center}
        \emph{So, what are tt-fields?}
    \vspace{.5em}
    \end{center}

Specifically, this calls for a definition of a theory of fields in tensor triangular geometry which allows us to realize the points of the spectrum of any tensor triangulated category $\cat C$ geometrically, i.e., through residue functors $\cat C \to \cat F$. The minimal desiderata therefore are to give a definition of \emph{tt-fields} such that:
\begin{enumerate}
    \item[($\dagger$)] If $\cat F$ is a tt-field, then the spectrum $\Spc(\cat F) = \ast$ is a single point.
    \item[($\ddagger$)] For every prime $\cat P \in \cat C$, there exists a tt-functor $\cat C \to \cat F_{\cat P}$ whose image on spectra is precisely $\{\cat P\}$.
\end{enumerate}
These are the two conditions singled out by Balmer in \cite[Section 4.3]{BalmerICM2010}, but as noted there, we emphasize that both of them require sharpening: For instance, the first one does not rule out nil-extensions of fields, while the second one should be strengthened to a uniqueness statement up to a suitable equivalence relation.

Unfortunately, such a general theory of tt-fields remains elusive and various attempts at partial solutions have been considered. Guided by classical examples from commutative algebra, homotopy theory, and modular representation theory, two closely related definitions of tt-field were proposed in \cite{BalmerICM2010} and \cite{BKS2019}. While both of them satisfy Desideratum $(\dagger)$, it is much harder to verify $(\ddagger)$ for these.  In particular, we construct an example \Cref{prop:ultrattfields} that shows that the two proposed definitions of tt-fields given in \cite{BalmerICM2010} and \cite{BKS2019} do not coincide, and for which $(\ddagger)$ fails, at least if we are using the definition proposed in \cite{BKS2019}.

Changing perspective and restricting attention to rigid tt-categories, Balmer, Krause, and Stevenson \cite{BKS2019} relaxed the condition that the residue fields be tt-categories to instead study homological functors from $\cat{C}$ to \emph{abelian} residue fields. In a series of papers \cite{balmernilpotence,balmerhomological, BalmerCameron2021}, this idea has been developed into a theory of homological spectrum $\Spc^h(\cat C)$, homological primes, and homological support, paralleling the triangular theory from \cite{Balmer2005}. In particular, Balmer shows that there is a surjective map
    \[
        \phi\colon \Spc^h(\cat C)  \twoheadrightarrow \Spc(\cat C),
    \]
and that every homological prime $\cat{B}\in \Spc^h(\cat C)$ is detected by an abelian residue field. In addition, he proves an abstract nilpotence theorem for $\cat{C}$ based on $\Spc^h{\cat{C}}$. Since $\Spc(\cat C)$ parametrized thick tensor ideals of $\cat C$, from an abstract point of view the map $\phi$ expresses a formal relationship between a nilpotence theorem and a thick ideal theorem for $\cat C$. Inspired by a broad variety of examples, the next conjecture reconciles the two points of view: 

\begin{Conj*}[Balmer's Nerves of Steel Conjecture] 
For any rigid $tt$-category $\cat C$ the comparison map $\phi\colon \Spc^h(\cat C)  \twoheadrightarrow \Spc(\cat C)$ is bijective. 
\end{Conj*}

Besides it significance in the aforementioned program to construct residue fields in tt-geometry, this conjecture would also have other important applications, as reviewed in \Cref{sec:prelim}. The conjecture is known in all examples in which the spectrum has been computed (see e.g., \cite[Section 5]{balmernilpotence}) and enjoys a number of strong permanence (such as descent) properties \cite{BHS2023a,BHSZ2024pp}. Despite this evidence, our first main result (\Cref{thm:ENCfails}) disproves the Nerves of Steel conjecture:

\begin{thmx}\label{thmx:nos}
    The Nerves of Steel Conjecture is false. More precisely, the comparison map $\phi$ fails to be injective for the free rigid commutative 2-ring $\Apoi$ on a pointed object. 
\end{thmx}

The key ingredients in the proof of this result are a generalized version of the 1-dimensional cobordism hypothesis and Deligne's semisimplicity theorem for his categories $\mathrm{Rep}(GL_t)$. 

In light of the disproof of the Nerves of Steel conjecture, our second main objective in this paper is a different approach to the constructing of residue objects in tt-geometry. Based on the notion of Nullstellensatzian object introduced in \cite{2022arXiv220709929B}, we construct a hierarchy of \emph{constructible spectra} $\Speccons_{\bb{E}_n}({\cat C})$ for $1\leq n \leq \infty$ along with suitable comparison maps to the homological spectrum of $\cat C$, provided the tt-category $\cat C$ admits a suitable enhancement as a rigid $\bb{E}_m$-2-ring for $1\leq n < m$. 

As a first proof of concept, for rational rigid commutative 2-rings $\cat C$, we show that Nullstellensatzian categories supply a satisfactory theory of residue fields. Generalizing prior work of Mathew \cite{mathewresidue} in the Noetherian case, \Cref{cor:Qttresidues} states:

\begin{thmx}\label{thmx:Q}
    If $R$ is a rational $\bb{E}_{\infty}$-ring, then $\mo{Perf}(R)$ has enough tt-fields, and points in the homological spectrum of $\mo{Perf}(R)$ are all witnessed by maps from $R$ into rational 2-periodic fields.  
\end{thmx}

As one application, we deduce that for any module over a rational $\bb{E}_{\infty}$-ring $R$, its naive homological support agrees with the genuine one, answering a question from \cite{BHSZ2024pp} in this case. The perspective employed in the proof of \Cref{thmx:Q} will also lead to a verification of the rational (in fact: finite height) monogenic Nerves of Steel conjecture in work in progress by  Burklund and separately Chedalavada.

Returning to the general case, the theory of $\bb{E}_n$-constructible spectra we develop along with an elaboration on Burklund's multiplicative structure theorem (\cite{burklund2022multiplicativestructuresmoorespectra}) leads to the next result. It says that, as long as the given tt-category has a suitably structured enhancement, every point in its homological spectrum is ``geometric'', that is, detected by a tt-category whose homological spectrum is a single point:

\begin{thmx}\label{thmx:pts}  
    Let $\cat C$ be a rigid tt-category which admits an enhancement and let $\mf{m}\in \Spc^h(\cat C)$ be a homological prime.  Then there exists a tt-functor $\cat C\to \cat K$ to a rigid tt-category $\cat K$ such that 
        \begin{enumerate}
            \item $\Spc^h(\cat K)=*$ a single point,
            \item and such that the map $\Spc^h(\cat K)\to \Spc^h(\cat C)$ has image exactly $\{\mf{m}\}$.
        \end{enumerate}
\end{thmx}

As a consequence, we deduce that the residue tt-functors constructed in \Cref{thmx:pts} satisfy the desiderata ($\dagger$) and ($\ddagger$). However, we caution the reader that, as we will explain, the constructions that go into the proof of this theorem are not yet sufficient to supply a fully satisfying theory of fields in tt-geometry.

\subsection{Methodology}\label{ssec:methodology}

\subsubsection*{Free constructions and the counterexample to the Nerves of Steel conjecture}

The starting point for our disproof of the Nerves of Steel conjecture is the \emph{exact nilpotence condition}, due to Balmer \cite[Theorem A.1]{balmerhomological} and the second-named author \cite{logan}. This condition holding for all localizations of a given rigid tt-category is equivalent to the Nerves of Steel conjecture for said category but, since it does not make reference to the homological spectrum, it also makes sense outside the rigid context. The main theorem of \cite{logan} shows that it fails for the free (non-rigid) commutative 2-ring on a pointed object. Our strategy is to adapt this approach to the rigid context, which poses substantial additional difficulties.

There is a free rigid commutative 2-ring $\Apoi$ on a pointed object, as well as a free rigid commutative 2-ring $\A$ on an object.  The category $\A$ plays the role of the affine line in higher Zariski geometry \cite{aoki2025higherzariskigeometry}, in that it represents the global sections functor. To make our lives easier, we will implicitly assume that $\Apoi$ (resp. $\A$) is rational, working with the free rigid commutative 2-ring on a pointed object (resp.~on an object) over the derived category of the rationals $\cat{D}^b(\bb{Q})$.  

In contrast to the non-rigid case, $\A$ and $\Apoi$ are not local and have very large Balmer spectra, see \Cref{cor:BalmerSpcBig}  and \Cref{cor:A1pointsdescribe} for a description of many points of $\A$. In particular, if we wish to study the exact-nilpotence condition, we must content ourselves with attempting to study some of their localizations. This applies, in particular, to a certain generic point $\eta$ of $\A$ and likewise for $\Apoi$. 

Now taking the free pointed object to the pointed object $\bb{1}\xrightarrow{0} X$ induces a tt-functor $\mathrm{can}\colon \Apoi\to \A$. In order to analyze this functor as well as the categories it relates, we will make use of the 1-dimensional cobordism hypothesis due to Lurie (\cite{luriecob}, see also \cite{harpaz}) and its (forthcoming) generalization by Barkan--Steinebrunner \cite{barkansteinebrunner2}. It provides identifications
    \[
        \A = \mo{Fun}((\mo{Cob})^{op},\cat{D}(\bb{Q}))^{\omega} 
            \quad \text{ and } \quad
        \Apoi = \mo{Fun}((\mo{Cob}^+)^{op},\cat{D}(\bb{Q}))^{\omega},
    \]
resulting in a `generators and relations descriptions' of these categories. In particular, in \Cref{section:affineline} we are able to relate $\A$ to Deligne's category $\mathrm{Rep}(GL_t)$ (see \cite{deligne} and recalled in \Cref{sec:appendix}) and use his semisimplicity theorem. The key structural features we establish are summarized in the following result, collecting \Cref{thm:A1etasemisimple} and \Cref{proposition:conservative}: 

\begin{thmx}\label{thmx:nos2}
The functor $\mathrm{can}$ extends to a tt-functor $\mathrm{can}_{\eta}\colon\Apoi_{\eta}\to \A_{\eta}$ on generic points. Moreover, we have:
    \begin{enumerate}
        \item The functor $\mathrm{can}_{\eta}\colon\Apoi_{\eta}\to \A_{\eta}$ is conservative.
        \item The tt-category $\A_{\eta}$ is a semisimple tt-field.
    \end{enumerate}
\end{thmx}

As explained at the end of \Cref{sec:ComparisonFunctor}, this theorem suffices to conclude that the exact-nilpotence condition fails for $\Apoi_\eta$, for the (image of the) universal fiber sequence $Y\to \bb{1}\to X$, thereby proving \Cref{thmx:nos}.

\subsubsection*{$\bb{E}_\infty$-constructible spectra and the rational case}

The drawbacks of the known notions of residue fields in tt-geometry---as exhibited above---lead us to a novel approach, which is formulated in the setting of higher Zariski geometry \cite{aoki2025higherzariskigeometry}. This requires that the tt-category under consideration admits a suitable enhancement, i.e., arises as the homotopy category of a commutative 2-ring\footnote{by which we mean an (essentially small) stably symmetric monoidal idempotent-complete $\infty$-category}. Since all tt-categories in nature are of this form and because the salient operations preserve the enhancement, we view this is as a very mild assumption.

In \cite{2022arXiv220709929B} and motivated by applications to chromatic homotopy theory, Burklund, Schlank and Yuan introduced a general notion of Nullstellensatzian object which captures the essential features of algebraically closed field abstractly. Applying this concept to the category of commutative 2-rings leads to our definition of the \emph{constructible spectrum} $\Speccons(\cat{C})$ of $\cat C$, see \Cref{def:consspec}. This spectrum is a compact $T_1$-space whose points are given by Nullstellensatzian commutative 2-rings. In \Cref{cor:conscomparisonmap} we construct a natural comparison map    
    \[
        \psi\colon\Speccons(\cat{C}) \to \Spc(\cat C).
    \]
The significance of this map is then explained in \Cref{prop:ConsDetectFieldPts}: A point in $\Spc(\cat C)$ is in the image of $\psi$ if and only if it is can be detected by a map of commutative 2-rings $\cat C \to \cat D$ with $\Spc(\cat D) = \ast$. In other words, the constructible spectrum provides an abstract setting which captures Desiderata $(\dagger)$ and $(\ddagger)$. In general, however, the comparison map $\psi$ is not surjective, so not every prime ideal in the Balmer spectrum can be realized geometrically by a map of commutative 2-rings. We remark that the failure of realizability holds both in the finite positive height situation (\Cref{cor:nonsurj}) as well as at height $\infty$ (\Cref{ex:nosurj}), that is in characteristic $p$.

This is in sharp contrast to what happens in characteristic $0$. Working with a rational commutative 2-ring $\cat C$ for the remainder of this section, we first observe that Nullstellensatzian $\cat C$-algebras correspond under decategorification to Nullstellensatzian $\bb{1}_{\cat{C}}$-algebra: Indeed, \Cref{prop:consequalscons} establishes a homeomorphism 
    \[
        \Speccons(\cat{C})\simeq \Speccons_{\mo{CAlg}\left(\mo{Ind}\left(\cat{C}\right)\right)}(\bb{1}_{\cat{C}}).
    \]
This affords the construction of a comparison map $\psi^h\colon\Speccons(\cat{C}) \to \Spc^h(\cat C)$ for rational $\cat C$, which we show in \Cref{thm:consequalshom} to be a bijection: 

\begin{thmx}\label{thmx:Qnos}
     Let $\cat{C}$ be a rational rigid commutative 2-ring.  Then there is a natural isomorphism of sets $\psi^h\colon\Speccons(\cat{C})\simeq \Spc^{h}(\cat{C})$ between the constructible spectrum and the homological spectrum of $\cat{C}$.
\end{thmx}

With this result in hand, it is then not too difficult to deduce our applications to rational monogenic commutative 2-rings in \Cref{thmx:Q}.  At the same time, this construction gives a new topology on the homological spectrum of a rational rigid commutative 2-rings which is always compact $T_1$.  The question as to when it is compact Hausdorff can be checked by understanding Nullstellensatzian algebras in a given category, and is closely related to when the homological spectrum is a sheafification of the Balmer spectrum for a certain ``canonical topology'' (see \Cref{thm:NoSandENCNSian}).  It is possible that the topology on the homological spectrum induced by \Cref{thmx:Qnos} is always compact Hausdorff, a question which is equivalent to a uniform bound on the exact-nilpotence condition for \textit{Nullstellensatzian} rational rigid commutative 2-rings, and which reduces to separating two specified points in the constructible spectrum of our counter-example to the nerves of steel conjecture (see \Cref{prop:CHausmaybe} for a precise statement).

\subsubsection*{$\bb{E}_n$-constructible spectrum and geometric points}

As we have seen, outside the rational setting, points of the Balmer spectrum of a tt-category can in general not be realized geometrically through maps of commutative 2-rings. Surprisingly, a minor modification of the constructible spectrum resolves this issue entirely.

Generalizing the previous setup, suppose for the remainder of this section that $\cat C$ is a rigid $\bb{E}_m$-2-ring, i.e., an (essentially small) rigid stably $\bb{E}_m$-monoidal idempotent-complete $\infty$-category for some $m \geq 3$. The condition on $m$ guarantees that the homotopy category of $\cat C$ is a rigid tt-category. For any $1 \leq n<m$, we then define the $\bb{E}_n$-constructible spectrum of $\cat C$ as $\Speccons_{\bb{E}_n}(\bb{1}_{\cat C})$. This comes with a comparison map 
    \[
        \psi^h_n\colon \Speccons_{\bb{E}_n}(\bb{1}_{\cat C}) \to \Spc^h(\cat C)
    \]
induced by sending a Nullstellensatzian $\bb{E}_n$-algebra to its homological support. Our main theorem (\Cref{thm:nconsequalshom}) about the $\bb{E}_n$-constructible spectrum shows that $\psi^h_n$ is bijective: 

\begin{thmx}\label{thmx:pts2}
    Let $\mo{Ind}(\cat C)$ be the Ind category of a rigid $\bb{E}_m$-2-ring $\cat C$.  Then for all $1\leq n < m$, there is a natural bijection
        \[
            \psi^h_n\colon\Speccons_{\bb{E}_n}(\bb{1}_{\cat C}) \simeq \Spc^h(\cat C).
        \]
\end{thmx}

In fact, \Cref{thm:nconsequalshomoptimal} provides an example that shows that the condition $n<m$ in the statement above is optimal. The key ingredient in the proof of \Cref{thmx:pts2} is the construction, for each homological prime $\mf{m}$, of $\bb{E}_n$-algebra variants $E_{\mf{m}}^{n}$ of the weak rings $E_{\mf{m}}$ introduced in \cite[Construction 2.11]{balmerhomological}. The latter play a distinguished role in the study of the homological spectrum, and likewise their structured analogues are essential in the analysis of $\Speccons_{\bb{E}_n}(\bb{1}_{\cat C})$. The construction of $E_{\mf{m}}^{n}$ relies essentially on Burklund's work \cite{burklund2022multiplicativestructuresmoorespectra} on multiplicative structures on quotient objects in higher algebra. 

In order to deduce \Cref{thmx:pts} from \Cref{thmx:pts2}  we need one additional step. While we cannot directly conclude that the homological spectrum of $E_{\mf{m}}^n$ is a point, we will prove in \Cref{sec:geompts} that $\Spc^h(\mo{Perf}_{\cat{C}}(E_{\mf{m}}^k))$ has the desired property for $k$ sufficiently large. This establishes a sufficient supplies of residue rigid $\bb{E}_n$-2-rings, which upon passage to homotopy categories gives \Cref{thmx:pts}. Finally, we remark that the overall structure of $\mo{Perf}_{\cat{C}}(E_{\mf{m}}^k)$ is not yet well-understood, so the extent to which these residue tt-categories should be considered fields is still under investigation. 

\subsection{Structure of the document}\label{ssec:structure}

There are two parts, consisting in the disproof of the Nerves of Steel conjecture and our analysis of the constructible spectrum, respectively. 

In the first part, we begin with some preliminaries on tt-geometry collected here for the convenience of the reader. Most of this material in \Cref{sec:prelim} is purely expositional, with the exception of the discussion of tt-fields. \Cref{sec:freeconstructions} deals with free constructions in higher Zariski geometry and the generalized 1-dimensional cobordism hypothesis, which is then applied in \Cref{section:affineline} to analyse the affine line and its pointed variant. This part culminates in \Cref{sec:ComparisonFunctor}, where the proof of \Cref{thmx:nos} is assembled.

The second part is concerned with the constructible spectrum and its applications to the construction of geometric points in higher Zariski geometry. After a brief recollection on the theory of Nullstellensatzian objects from \cite{2022arXiv220709929B}, \Cref{sec:SpectraSpectra} introduces the $\bb{E}_{\infty}$-constructible spectrum of rigid commutative 2-rings, establishes its basic properties, and discusses its relevance in tt-geometry. The end of this section also contains our counterexamples to the most optimistic guesses about its utility, motivating the in-depth study of the rational case in \Cref{sec:rational} and the proof of \Cref{thmx:Q}. Finally, \Cref{sec:enconsspec} and \Cref{sec:geompts} introduce the $\bb{E}_{n}$-constructible spectrum of rigid 2-rings and, using a modification of Burklund's approach to multiplicative structures, verify \Cref{thmx:pts}.

This paper concludes with an outline of a special case of Deligne's proof of the semi-simplicity of $\mo{Rep}(GL_t)$ from \cite{deligne} in \Cref{sec:appendix}.

\subsection{Conventions and notation}\label{ssec:conventions}

Throughout the main body of this work, we will work in the framework of higher Zariski geometry as developed in \cite{aoki2025higherzariskigeometry}, formulated in the language of $\infty$-categories (\cite{HTT,HA}) and enhancing tensor triangular geometry as developed by Balmer \cite{Balmer2005}. Our notations and terminology will conform to these references unless otherwise noted. In particular, \emph{tt-category} refers to the classical 1-categorical notion of a tensor triangulated category, while we use the term \emph{$\bb{E}_m$-2-ring} with $0 \leq m \leq \infty$ for an essentially small stably $\bb{E}_m$-monoidal $\infty$-category, usually assumed to be idempotent-complete. When $m \geq 3$, the homotopy category $\cat K = \ho(\cat C)$ of an $\bb{E}_m$-2-ring $\cat C$ is a tt-category; in this case, $\cat C$ is said to be an \emph{$\bb{E}_m$-enhancement} of $\cat K$.  For the remainder of this paper, tt-category will implicitly mean rigid tt-category, and the term 2-ring without any modifiers will implicitly mean commutative 2-ring.

\subsection*{Acknowledgements} 

We thank Paul Balmer, Anish Chedalavada, Akhil Mathew, Thomas Nikolaus, Phil Pützstück, Tomer Schlank, Robin Sroka, and Jan Steinebrunner for many helpful discussions. TB and LH were supported by the European Research Council (ERC) under Horizon Europe (grant No.~101042990) and are grateful to the Max Planck Institute for Mathematics in Bonn for its hospitality and financial support. MR is funded by the  Deutsche Forschungsgemeinschaft (DFG, German Research Foundation) – Project-ID 427320536 – SFB 1442, as well as under Germany's Excellence Strategy EXC 2044/2 –390685587, Mathematics Münster: Dynamics–Geometry–Structure.

\newpage
\part{A Counterexample to the Nerves of Steel Conjecture}
\section{Preliminaries on Tensor Triangular Geometry}\label{sec:prelim}

\subsection{Basic notions of tensor-triangular geometry}  Motivated by computations due to Devinatz--Hopkins--Smith \cite{DevHopSmith,hopkins1998nilpotence}, Hopkins--Neeman \cite{Hopkins1987,Neeman1992}, and Thomason \cite{Thomason1997}, Balmer \cite{Balmer2005} made the following definition, marking the beginning of the field now known as tensor triangular geometry.

\begin{definition}[\cite{Balmer2005}]
Let $\cat{C}$ be an essentially small tt-category.  Then the \textit{Balmer spectrum} of $\cat{C}$, $\Spc(\cat{C})$ is the set of prime $\otimes$-ideals in $\cat{C}$, topologized by defining distinguished closed subsets 
    \[
        \supp(x)\coloneqq \left\{\mc{P}\in\Spc(\cat{C})\colon x\notin \mc{P}\right\},
    \]
for objects $x\in\cat{C}$.
\end{definition}

This definition parallels the construction of the Zariski spectrum in algebraic geometry, and thus leads to a geometric perspective on tt-categories; for a recent account, see \cite{aoki2025higherzariskigeometry}. The Balmer spectrum comes equipped with natural maps to other spaces, such as the homogeneous spectrum of graded prime ideals in the graded endomorphism ring of the unit.

\begin{proposition}[{\cite[Theorem 5.3]{BalmerSpSpSp}}]\label{prop:comparisonUnit}  There exists a natural continuous map 
    \[
        \rho\colon\Spc(\cat{C})\to \Spech(\pi_*(\mo{End}_{\cat{C}}(\bb{1})))
    \] 
from the Balmer spectrum of a tt-category $\cat{C}$ to the spectrum of homogeneous prime ideals in the graded-commutative ring $\pi_*(\mo{End}_{\cat{C}}(\bb{1}))$. The map $\rho$ is order-reversing, in the sense that $\mc{P} \subseteq \mc{Q}$ if and only if $\rho(\mc{P}) \supseteq \rho(\mc{Q})$.
\end{proposition}

Equipped with a good notion of geometry for tt-categories, the familiar words and concepts of algebraic geometry, or at least some of them, can be adapted to the world of tensor triangular geometry.  A particularly important notion for the present paper is what it means to be local, that is, have the property that the spectrum has a unique closed point.  Pondering the definitions (\cite{BalmerSpSpSp}), closed points of the space $\Spc(\cat{C})$ correspond to minimal prime ideals in $\cat{C}$, and one learns that $\mo{Spc}(\cat{C})$ has a unique closed point if and only if the ideal $(0)$ is prime.  Collecting this into a definition, and rephrasing what it means for $(0)$ to be prime, we recall:

\begin{definition}  A tt-category $\cat{C}$ is said to be $\textit{local}$ is for all objects $X,Y\in\cat{C}$, $X\otimes Y\simeq 0$ implies that $X\simeq 0$ or $Y\simeq 0$.
\end{definition}

Another important player in this paper is the notion of $\otimes$-nilpotent maps.

\begin{definition}  A map $f\colon x\to y$ in a tt-category is said to be \textit{$\otimes$-nilpotent} if there exists some $n\geq 0$ with $f^{\otimes n}\simeq 0$.
\end{definition}

\begin{lemma}\label{lemma:omnibusnil}
Let $\cat{C}$ be a tt-category. 
The class of $\otimes$-nilpotent morphisms in $\cat{C}$ is closed under:
\begin{itemize}
\item Tensoring with any object; 
\item Retracts; 
\item Dualization.
\end{itemize}
Furthermore, if $F\colon \cat{C}\to \cat{D}$ is a tt-functor, it sends $\otimes$-nilpotent morphisms to $\otimes$-nilpotent morphisms. 
\end{lemma}
\begin{proof}
The first item follows from the fact that $(f\otimes z)^{\otimes n}\simeq f^{\otimes n}\otimes z^{\otimes n}$. The second item follows from the fact that if $f_0$ is a retract of $f_1$, then $f_0^{\otimes n}$ is a retract of $f_1^{\otimes n}$, and retracts of the $0$ morphism are $0$. The third item follows from the fact that dualization commutes with tensor products and the dual of the $0$ morphism is the $0$ morphism. The last fact is also clear from the equivalence $F(f)^{\otimes n}\simeq F(f^{\otimes n})$. 
\end{proof}

\begin{definition}\label{def:tensorfaithful}
An object $x\in\cat{C}$ in a tt-category $\cat{C}$ is said to be \textit{$\otimes$-faithful} if $x\otimes-$ is conservative on morphisms, that is: if $f\colon y\to z$ is a morphism in $\cat{C}$ with $x\otimes f\simeq 0$, then $f\simeq 0$.
\end{definition}

We will make use of the following elementary observation:

\begin{lemma}\label{lem:equivtensorfaithful}
An object $x$ in a tt-category $\cat{C}$ is $\otimes$-faithful if and only if the co-evaluation morphism \[\mo{coev}_x\colon\bb{1}\to x\otimes x^{\vee}\] is split injective.  In particular, if every non-zero object of $\cat{C}$ is $\otimes$-faithful, then the Balmer spectrum of $\cat{C}$ is a single point.
\end{lemma}
\begin{proof}
If $\bb{1}$ splits off of $x\otimes x^{\vee}$, then $x\otimes x^{\vee}\otimes -$ is conservative on morphisms, and the same must then be true of the functor $x\otimes -$.  Conversely, if $x\otimes -$ is $\otimes$-faithful, then since $x\otimes \mo{coev}_x$ is split injective, the map $\mo{fib}(\mo{coev}_x)\to \bb{1}$ tensors with $x$ to zero, and is thus itself zero, which implies that $\mo{coev}_x$ is split injective.

The final claim follows from the observation that every $\otimes$-faithful object generates the unit ideal, which then must be true of every nonzero object.
\end{proof}

\subsection{The homological spectrum}  The first Balmer spectrum computation, as with many more to follow, was carried out by first proving an abstract nilpotence theorem, then using this to produce a classification of the thick $\otimes$-ideals.  Following the idea that there should be a general interplay between tensor triangular geometry and abstract nilpotence theorems and building on earlier work with Krause and Stevenson \cite{BKS2019}, Balmer in \cite[Remark~3.4]{balmernilpotence} found a new way to attach a topological space to a tt-category.  We recall the construction now.

\begin{definition}
Let $\cat{C}$ be a tt-category.  The category $\mo{mod}(\cat{C})$ is the full subcategory of additive presheaves on $\cat{C}$ valued in abelian groups 
    \[
        \mo{Fun}^{\oplus}(\cat{C}^{op},\mo{Ab})
    \]
generated by cokernels of maps $\yo(x)\to \yo(y)$ for $x,y\in\cat{C}$.
\end{definition}

The category $\mo{mod}(\cat{C})$ is an abelian category, which inherits a symmetric monoidal structure through Day convolution in such a way that the Yoneda embedding 
    \[
        \yo\colon\cat{C}\to \mo{mod}(\cat{C})
    \]
is symmetric monoidal.  In the absence of a general theory of ``residue fields'' for tt-categories, we turn our attention to nice ``abelian residue fields'' instead. Explicitly, Balmer defines a homological residue field of a tt-category $\cat{C}$ as a quotient of $\mo{mod}(\cat{C})$ by a maximal Serre $\otimes$-ideal.  This leads us to:

\begin{definition}[{\cite[Remark 3.4]{balmernilpotence}}]  The \textit{homological spectrum} of a tt-category $\cat{C}$ is defined as a set by  
    \[
        \Spc^{h}(\cat{C})\coloneqq\{\text{maximal Serre }\otimes\text{-ideals in }\mo{mod}(\cat{C})\}.
    \]
This becomes a topological space by defining closed subsets to be generated by those of the form 
    \[
        \supp^h(x)\coloneqq\{\mf{m}\in\Spc^h(\cat{C})\colon\yo(x)\notin\mf{m}\}.
    \]
At any given point $\mf{m}$ in the homological spectrum, the \emph{homological residue field} at $\mf{m}$ is defined to be the composite functor $\cat{C}\to \mo{mod}(\cat{C})/\mf{m}$ to the quotient by the maximal Serre $\otimes$-ideal $\mf{m}$.
\end{definition}
\subsection{Weak rings}
\begin{definition}
Let $\cat{C}$ be a tt-category, and consider its big category $\mo{Ind}(\cat{C})$.  A \textit{weak ring} in $\mo{Ind}(\cat{C})$ is a pointed object $f\colon\bb{1}\to R$ which becomes split after tensoring with $R$, that it, such that
\[R\otimes f\colon R\to R\otimes R\] is split injective.
\end{definition}

\begin{remark}\label{rem:weakringprimes}  Each point of the homological spectrum $\mf{m}$ gives rise to a unique weak ring object $E_{\mf{m}}$ in the big category $\mo{Ind}(\cat{C})$ corresponding to it, such that the kernel of tensoring with $E_{\mf{m}}$ is exactly the kernel of the composite $\mo{Ind}(\cat{C})\to \mo{Ind}\left(\mo{mod}(\cat{C})\right)\to \mo{Ind}\left(\mo{mod}(\cat{C})/\mf{m}\right)$ \cite[Proposition 3.9]{BKS2019}.  In \cite{BalmerCameron2021}, Balmer--Cameron demonstrate techniques to compute a great number of these weak rings in practice.
\end{remark}

The following property of the weak rings attached to homological primes will be useful later.

\begin{proposition}[{\cite[Proposition 5.3]{balmernilpotence}}]\label{prop:weakRingtenszero}
Let $\cat{C}$ be a tt-category, and consider  homological primes $\mf{m}_1,\mf{m_2}\in \Spc^{h}(\cat{C})$ with associated weak rings $E_{\mf{m}_1},E_{\mf{m}_2}$.  Then $E_{\mf{m}_1}\otimes E_{\mf{m}_2}\simeq 0$ if and only if $\mf{m}_1 \neq \mf{m}_2$.
\end{proposition}
There are two natural notions of homological support attached to big objects in a given tt-category.
\begin{definition}
Let $\cat C$ be a tt-category, which has associated big category ``$\mo{Ind}(\cat C)$'' (say, arising as the homotopy category the Ind-category of a rigid $\bb{E}_n$-2-ring with homotopy category $C$), and let $t\in \mo{Ind}(\cat C)$  a big object.  Then we define
\begin{itemize}
	\item  The \textit{naive homological support} of $t$ is 
	\[\supp^{n}(t)\coloneqq \{\mf{m}\in\Spc^h(\cat C)\colon t\otimes E_{\mf{m}}\neq 0\}.\]
	\item  The \textit{genuine homological support} of $t$ is 
	\[\supp^{h}(t)\coloneqq \{\mf{m}\in\Spc^h(\cat C)\colon \mo{hom}(t,E_{\mf{m}})\neq 0\}.\]
\end{itemize}
\end{definition}
We record the following properties of these supports, which will be useful later.
\begin{proposition}[{\cite[Theorem 4.5]{balmerhomological}}]\label{prop:tensorproductprop}
The genuine homological support satisfies the so-called \textit{tensor product property}.  Namely, 
\[\supph(t_1\otimes t_2)=\supph(t_1)\cap \supph(t_2).\]
\end{proposition}

\begin{proposition}[{\cite[Theorem 4.7]{balmerhomological}}]\label{prop:naiveequalsgenuine}
If $t$ is a weak ring, then there is an equality
\[\supp^n(t)=\supph(t).\]
\end{proposition}
We will frequently identify $\supph(t)$ with $\supp^{n}(t)$ implicitly when working with weak rings later on.  Finally, we record also the following detection property of the homological support for weak rings.
\begin{proposition}[{\cite[Theorem 4.7]{balmerhomological}}]\label{prop:suppdetection}
If $w\in \mo{Ind}(\cat C)$ is a weak ring with homological support $\supph(w)=\emptyset$, then $w=0$. 
\end{proposition}

\subsection{Nilpotence detection}
Balmer investigated the connection of this homological spectrum to nilpotence theorems, proving that in some sense, proving abstract nilpotence theorems is equivalent to computing the homological spectrum of a category.

\begin{proposition}[{\cite[Corollary 4.7, Theorem 5.4]{balmernilpotence}}] Given a map $f\colon x\to Y$ for $x\in\cat{C}$ and $Y\in\mo{Ind}(\cat{C})$, if $f$ vanishes in every homological residue field of $\cat{C}$, then $f$ is $\otimes$-nilpotent.  Furthermore, given a collection $I=\{\mf{m}\}\subseteq \Spc^h(\cat{C})$ of points in the homological spectrum with the property that $f$ vanishing in the residue field at $\mf{m}$ for all $\mf{m}\in I$ implies $f$ is $\otimes$-nilpotent, then $I=\Spc^h(\cat{C})$ is the entire space.
\end{proposition}

An important related notion is that of nil-conservativity; we recall the definition:

\begin{definition}
Let $\{f_i\colon\cat{C}\to \cat{D}_i\}_{i\in I}$ be a family of functors from a fixed tt-category $\cat{C}$ to various tt-categories $\cat{D}_i$.  We say that the family is \textit{jointly nil-conservative} if the induced functors on big categories (abusively also denoted by $f_i$) detect weak rings, that is: if $R$ is a weak ring in $\mo{Ind}(\cat{C})$ such that $f_i(R)\simeq 0$ for all $i\in I$, then $R\simeq 0$.
\end{definition}
The connection between jointly nil-conservative families and the homological spectrum is summarized by the following theorem.
\begin{theorem}[{\cite[Theorem 1.9]{BCHS2024}}]\label{thm:surjspech} Let $\{f_i\colon\cat{C}\to \cat{D}_i\}_{i\in I}$ be a family of tt-functors.  Then the family is jointly nil-conservative if and only if the induced map \[\coprod_{i\in I}\Spc^{h}(\cat{D}_i)\to\Spc^h(\cat{C})\] is surjective.
\end{theorem}

We will later require a basechange property for jointly nil-conservative families in the highly structured case, which in turn will require the following lemma.

\begin{lemma}\label{lem:FFBasechange}
Let $\cat{C}$ and $\cat{D}$ be rigid 2-rings, and let $f^*\colon\cat{C}\to \cat{D}$ be a 2-ring map which is fully faithful (so that the right adjoint $f_*$ on Ind-categories has $f_*f^*\simeq id_{\mo{Ind}(\cat{C})}$.  Then for any map of rigid 2-rings $g^*\colon\cat{C}\to \cat{E}$, the basechange $\cat{E}\to \cat{E}\otimes_{\cat{C}}\cat{D}$ is also fully faithful.
\end{lemma}
\begin{proof}
By \cite[Proposition 4.18]{LocRig} (which applies since $\cat{C}$ is rigid), the adjunction $f^*\dashv f_*$ is $\mo{Ind}(\cat{C})$-linear and hence, by $2$-functoriality of $\cat{E}\otimes_{\mo{Ind}(\cat{C})} -$, it induces an adjunction upon tensoring with $\cat{E}$. Since fully faithfulness can be expressed in $2$-categorical terms (that the unit map is an equivalence), the result follows. 

\end{proof}
\begin{proposition}\label{prop:nilconsbasechange}
Let $\cat{C}$ be a rigid 2-ring.  Suppose we are given a family of functors $F_{i}\colon\cat{C}\to \cat{D}_i$, $i\in I$ of rigid 2-rings which is jointly nil-conservative.  Then, for any 2-ring map $G\colon\cat{C}\to \cat{E}$ to a rigid 2-ring $\cat{E}$, the family
\begin{equation*}
\{\cat{E}\to \cat{E}\otimes_{\cat{C}}\cat{D}_i\}_{i\in I}
\end{equation*} is jointly nil-conservative as well.
\end{proposition}

\begin{proof}
Let $S\in \mo{Ind}(\cat{E})$ be a weak ring, and assume that $F_i(S)\in \mo{Ind}(\cat{E})\otimes_{\mo{Ind}(\cat{C})} \mo{Ind}(\cat{D}_i)$ is $0$ for each $i$. 

Again, by \cite[Proposition 4.18]{LocRig}, the adjunction $G\dashv G_*$ between $\Ind(\cat{C})$ and $\Ind(\cat{E})$ induces an adjunction $G\otimes \Ind(\cat{D}_i) \dashv G_* \otimes \Ind(\cat{D}_i)$ between  $\Ind(\cat{D}_i)$ and $\Ind(\cat{E})\otimes_{\Ind(\cat{C})}\Ind(\cat{D}_i)$, which we abusively still denote $G,G_*$. 

We therefore find that $G_*F_i(S) = 0$. We claim that the following square is vertically right adjointable, i.e., that the canonical map $F_iG_*\to G_*F_i$ is an equivalence: 

\[\begin{tikzcd}
	{\Ind(\cat{C})} & {\Ind(\cat{D}_i)} \\
	{\Ind(\cat{E})} & {\Ind(\cat{E})\otimes_{\Ind(\cat{C})} \Ind(\cat{D}_i)}
	\arrow["{F_i}", from=1-1, to=1-2]
	\arrow["G"', from=1-1, to=2-1]
	\arrow["G"', from=1-2, to=2-2]
	\arrow[shift right=3, dashed, from=2-1, to=1-1]
	\arrow["{F_i}"', from=2-1, to=2-2]
	\arrow[shift right=3, dashed, from=2-2, to=1-2]
\end{tikzcd}\]

Once this is proved, it will follow that $F_i G_*(S)=0$ for all $i$, and hence that $G_*(S)= 0$ by nil-conservativity. But there is a map of weak rings $GG_*(S)\to S$ by design, so this implies that $S=0$. 

To prove this adjointability, we again use $2$-functoriality of $-\otimes_{\Ind(\cat{C})}-$: the arrow $\Ind(\cat{C})\to \Ind(\cat{D}_i)$ sends the adjunction $G\dashv G_*$ to an adjunction in the arrow category, and those are precisely adjointable squares (this follows from the more general \cite[Theorem 4.6]{haugsengmonad}). 

\end{proof}

\subsection{The nerves of steel conjecture}  For any tt-category $\cat{C}$, there exists a natural continuouos surjection 
    \[
        \phi\colon\Spc^{h}(\cat{C})\twoheadrightarrow \Spc(\cat{C})
    \]
by \cite[Corollary 3.9]{balmernilpotence}, which connects us back to the setting of tensor triangular geometry.  The manner in which the very first computation of a Balmer spectrum was carried out in \cite{hopkins1998nilpotence} was essentially by computing the homological spectrum of $\Sp$, and then proving that the map $\Spc^h(\Sp)\to \Spc(\Sp)$ is an isomorphism.  In modern terms, Hopkins--Smith proved the first case of Balmer's nerves of steel conjecture, which we can now state.

\begin{definition}\label{def:steelcondition}  
    We say that the \textit{nerves of steel condition} holds for a tt-category $\cat{C}$ (written NoS($\cat{C})$) if the map $\phi\colon \Spc^h(\cat{C})\to \Spc(\cat{C})$ is a bijection. The \textit{nerves of steel conjecture} asserts that every tt-category satisfies the nerves of steel condition.
\end{definition}
The nerves of steel conjecture has been a center of attention for tt-geometers in recent years, and many equivalent formulations have been discovered.  For instance:
\begin{remark}\label{rem:eqsteelcondition}
    For a tt-category $\cat C$, the comparison map $\phi$ exhibits $\Spc(\cat C)$ as the Kolomogorov quotient of $\Spc^h(\cat C)$, i.e., the universal $T_0$-space under it; see \cite[Lemma 4.2]{BHS2023a}. This implies the equivalence of the following statements:
        \begin{enumerate}
            \item $\phi\colon \Spc^h(\cat{C})\to \Spc(\cat{C})$ is a bijective;
            \item $\phi\colon \Spc^h(\cat{C})\to \Spc(\cat{C})$ is a homeomorphism; 
            \item $\Spc^h(\cat{C})$ is $T_0$ (and hence spectral).
        \end{enumerate}
\end{remark}

\begin{remark}
There is an alternative description of the homological spectrum constructed in \cite{BirdWilliamson2025} via a space constructed from the Ziegler spectrum, relating homological primes to so-called ``definable'' $\otimes$-subcategories.
\end{remark}

\begin{remark} Work of \cite{BHSZ2024pp} shows that if the nerves of steel conjecture holds, then the triangular and homological notions of stratification coincide for all rigidly-compactly generated tt-categories. In particular, stratification would satisfy descent along weakly descendable families of geometric tt-functors.
\end{remark}

\subsection{The exact-nilpotence condition}

The nerves of steel conjecture was reformulated in \cite[Theorem A.1]{balmerhomological}, and again in \cite[Theorem 1.6]{logan} in terms of the so-called exact-nilpotence condition, which we now recall.
\begin{definition}[{\cite[Definition 1.5]{logan}}]\label{def:ENC}
Let $\cat{C}$ be a local tt-category.  We say that the \textit{exact-nilpotence condition (resp.~exact-nilpotence condition to order $n$)} (ENC($n$)) holds for $\cat{C}$ if for all fiber sequences of the form
    \[
        y\xrightarrow{g}\mathbb{1}\xrightarrow{f}x,
    \]
there exists a nonzero object $z\in\cat{C}$ with either $z\otimes g$ or $z\otimes f$ being $\otimes$-nilpotent (resp. either $z\otimes g^{\otimes n}\simeq 0$ or $z\otimes f^{\otimes n}\simeq 0$).
\end{definition}
As mentioned in \Cref{ssec:methodology}, the starting point of this paper is the following theorem.
\begin{theorem}[{\cite[Theorem A.1]{balmerhomological}}; {\cite[Theorem 1.6]{logan}}] The nerves of steel conjecture holds if and only if the exact-nilpotence condition holds for every local tt-category $\cat{C}$.
\end{theorem}

\subsection{Tensor triangular fields}\label{ssec:ttfields}

We recall the definition of tt-fields proposed by Balmer, Krause, and Stevenson in \cite{BKS2019} and discuss some of its basic properties. In particular, we show that it differs from the earlier one given by Balmer in \cite{BalmerICM2010}.

\begin{definition}[{\cite[Definition 1.1]{BKS2019}}]\label{def:ttfield}
    A rigidly-compactly generated tt-category $\cat{F}$ is a \emph{tt-field} if it satisfies the following two conditions:
        \begin{enumerate}
            \item every nonzero compact object $x \in \cat{F}^c$ is $\otimes$-faithful, i.e., $x \otimes f =0$ implies $f=0$ for morphisms $f$ in $\cat{F}$;
            \item every object in $\cat{F}$ is a coproduct of compacts.
        \end{enumerate}
    Occasionally, we will also refer to the essentially small tt-category $\cat F^c$ as a tt-field.
\end{definition}

Simple examples of tt-fields are given by:

\begin{lemma}\label{lem:ssimpleimpllocal}    
    A semi-simple tt-category $\cat{K}$ with simple unit is a tt-field, in that for any nonzero object $z\in\cat{K}$, $\mo{coev}\colon\bb{1}\to z\otimes z^{\vee}$ is split injective.
\end{lemma}
\begin{proof}
    In a semi-simple category, any nonzero map from an object with simple endomorphism ring is split-injective. 
\end{proof}

\begin{remark}\label{rem:KrullSchmidt}
    An important consequence of the definition of tt-field $\cat{F}$ is that $\cat{F}^c$ is a Krull--Schmidt category, see for example \cite[Theorem 5.7]{BKS2019}.
\end{remark}

In \cite[Section 4.3]{BalmerICM2010}, Balmer proposed a variant, only requiring condition (a) in \Cref{def:ttfield}. The next proposition provides an example demonstrating that the two definitions do not coincide, thereby resolving a question which was raised in \cite[Remark 1.2]{BKS2019}.

\begin{proposition}\label{prop:ultrattfields}
    Let $(\kappa_n)_{n \in \bb{N}}$ be an $\bb{N}$-indexed collection of fields and let $\mc{U}$ be a non-principal ultrafilter on $\bb{N}$. Consider the ultraproduct $\prod_{\mc{U}}\mo{Perf}(\kappa_n)$ taken in the category of 2-rings.
        \begin{enumerate}
            \item Every nonzero compact object of $\prod_{\mc{U}}\mo{Perf}(\kappa_n)$ is $\otimes$-faithful, i.e., the ultraproduct satisfies condition (a) of \Cref{def:ttfield}. 
            \item There does not exist any tt-functor $\prod_{\mc{U}}\mo{Perf}(\kappa_n) \to \cat{F}$ to a tt-field $\cat{F}$. In particular, $\prod_{\mc{U}}\mo{Perf}(\kappa_n)$ does not satisfy condition (b) of \Cref{def:ttfield}
        \end{enumerate}
\end{proposition}
\begin{proof}
    Every non-zero compact object in $\prod_{\mc{U}}\mo{Perf}(\kappa_n)$ is $\otimes$-faithful since the co-evaluation map of every non-zero object splits; this can be checked coordinate-wise.

    In order to show that $\prod_{\mc{U}}\mo{Perf}(\kappa_n)$ does not admit any maps into a tt-field, consider the image $X$ of the object $(\bigoplus_{i=1}^{n} \kappa_n)_{n\in\bb{N}}$ in the ultraproduct. Since $\mc{U}$ is non-principal, for any $m \in \bb{N}$ there exists some $V_m \in \mc{U}$ which does not meet the interval $[0,m-1]$. In $\prod_{n \in V_m}\mo{Perf}(\kappa_n)$, the object $\bigoplus_{i=1}^{m}\bb{1}$ splits off $(\bigoplus_{i=1}^{n} \kappa_n)_{n\in V_m}$, hence the same is true for $X$ in the ultraproduct. Its image under any tt-functor cannot satisfy the Krull--Schmidt property, so we conclude by \Cref{rem:KrullSchmidt}. 
\end{proof}

As an immediate consequence, we obtain:

\begin{corollary}\label{cor:ultrattfields}
    The collection of tt-fields is not closed under ultraproducts. 
\end{corollary}

\begin{remark}\label{rem:nilextension}
    In light of the above, one might be tempted to try to define a tt-field as a rigidly-compactly generated tt-category where every non-zero compact object is $\otimes$-faithful, following \cite{BalmerICM2010}. This notion would be closed under ultraproducts but, as we will show in \Cref{rem:hcob} below, there are categories with this property which act much more akin to a ``nil-extension'' of a field.  

    Furthermore, recent work of Riedel (in particular \cite[Remark~4.24]{riedel2025reducedpointsmathbbeinftyringspositive}) can be used to show that, Desideratum $(\ddagger)$ from the introduction also fails for the definition proposed in \cite{BalmerICM2010} at least if one required the map to admit an enhancement as an $\bb{E}_1$-2-ring map.  We note that this does not completely rule out the possibility of $(\ddagger)$ holding for this definition at the purely tensor triangulated level, just of there being some enhanced version of it.
\end{remark}

This concludes our review of the required background material on tensor triangular geometry.

\newpage
\section{Free Constructions}\label{sec:freeconstructions}
The goal of this section is to define and analyze the free categories we will need.  Recall from \Cref{ssec:methodology} that we will need to analyze the free rational rigid 2-ring on a dualizable object and the free rational rigid 2-ring on a pointed dualizable object. A first step for this is to analyze the non-rational, non-stable versions of these categories. In a first sub-section, we discuss the free symmetric monoidal category on a dualizable object 
as well as its $\bb{Q}$-linearization, and in a second subsection we compare it to the free symmetric monoidal category on a pointed dualizable object, respectively to its $\bb{Q}$-linearization. 
\subsection{The free symmetric monoidal category on a dualizable object} 
The free symmetric monoidal category on a dualizable object has the universal dualizable object, $X$, its tensor powers $X^{\otimes i}$, and similarly for its dual, $X^\vee$, as well as their tensors. Besides the symmetric groups $\Sigma_i\times\Sigma_j$ acting on $X^{\otimes i}\otimes X^{\vee,\otimes j}$, there are extra morphisms coming from the evaluation pairing $X\otimes X^\vee\to \bb{1}$ and the evaluation co-pairing $\bb{1}\to X^\vee\otimes X$.
\[\begin{tikzpicture}
	\draw[black, very thick] (-4,0) rectangle (1,4);
	\draw[black, very thick] (3,0) rectangle (8,4);
	\path	(-3.25,2)	node	[black]	{$\emptyset$}
			(0,3.25)	node	[black]	{$+$}
			(0,0.75)	node	[black]	{$-$};
	\draw[thick, ->] (-.25,0.75) arc (270:90:1.25);
	\path	(7.25,2)	node	[black]	{$\emptyset$}
	(4,3.25)	node	[black]	{$+$}
	(4,0.75)	node	[black]	{$-$};
	\draw[thick, <-] (4.25,0.75) arc (90:270:-1.25);
	\path 	(-1.5,-0.5)	node 	[black]	{co-evaluation on $X=+$}
		(5.5,-0.5)	node	[black]	{evaluation pairing for $X=+$};
\end{tikzpicture}\]

The graphical calculus that arises from the triangle identities suggested early on a connection to manifolds, as formalized by Baez and Dolan's famous Cobordism Hypothesis \cite{BaezDolan1995}, which describes in general the free symmetric monoidal $(\infty,n)$-category on a dualizable object in terms of cobordism categories. The following is an informal description of the answer in dimension $1$. We will single out below what we need more precisely: 
\begin{definition}
The $1$-dimensional oriented cobordism category $\mo{Cob}^{1d,or}$ is the $(\infty,1)$-category whose objects are oriented closed $0$-dimensional manifolds, i.e., finite sets with a sign $\pm$ on each element, and whose morphisms from $M$ to $N$ are oriented $1$-dimensional compact manifolds with boundary $W$ with an oriented identification $\partial W\cong \overline{M}\coprod N$.
\end{definition}
Here, $\overline{M}$ is the manifold $M$ with the reverse orientation. 
\newpage %Fixed line spacing issue, thanks Anish ;)
\[\begin{tikzpicture}
	\draw[black, very thick] (-2,0) rectangle (2,2);
	\path (-1.5,1) node [black] {$+$}
		(1.5,1)	node	[black]	{$+$};
	\draw[thick, ->] (-1.25,1) -- (1.25,1);
	\path (0,-0.5)	node	[black]	{identity of $+=X$.};
\end{tikzpicture}
\]
A proof in general (for all $n$) has been sketched by Lurie in \cite{luriecob}, though there does not seem to be a consensus about its completeness. On the other hand, in dimension $1$, a complete proof was given by Harpaz, so we can state: 
\begin{theorem}[\cite{harpaz},\cite{luriecob}]\label{thm:cobhyp}
The free rigid symmetric monoidal $\infty$-category on an object is the category $\mo{Cob}^{1d,or}$ of 1-dimensional oriented cobordisms between oriented 0-manifolds.
\end{theorem}
Since no higher dimensional, or unoriented cobordisms will appear in this paper, we simplify notation:
\begin{notation}
We let $\mo{Cob}\coloneqq\mo{Cob}^{1d,or}$.

For $X\in \mo{Cob}$ the universal dualizable object (i.e., the singleton, as a positively oriented $0$-dimensional manifold), we let $\X{i}{j}$ denote $X^{\otimes i}\otimes X^{\vee,\otimes j}$. 
\end{notation}
The precise information we need about $\mo{Cob}$ can be gathered in the following:
\begin{proposition}\label{proposition:omnibusCob}
Let $i,j,r,s$ be natural numbers. 
\begin{enumerate}
\item The symmetric monoidal dimension $T$ of $X$ in $\mo{Cob}$ is given by the circle as a cobordism from the empty manifold to itself, and $\Omega(\Hom_{\mo{Cob}}(\emptyset,\emptyset), T) \simeq S^1$. Furthermore, the induced map $BS^1\to \Hom_{\mo{Cob}}(\emptyset,\emptyset)$ witnesses the target as free commutative monoid over the source; 
\item The mapping space $\Hom_{\mo{Cob}}(\X{i}{j},\X{r}{s})$ is empty unless $i-j=r-s$; 
\item The mapping space $\Hom_{\mo{Cob}}(\X{i}{j},\X{r}{s})$ is free on a discrete set as a module over $\Hom_{\mo{Cob}}(\bb{1},\bb{1})$.
\end{enumerate}
In particular, all mapping spaces in $\mo{Cob}$ are finite disjoint unions of spaces of the form $\coprod_n ((BS^1)^{\times n})_{h\Sigma_n}$.
\end{proposition}
\begin{proof}
(a) This follows from the definition of $\mo{Cob}$, together with the homotopy equivalence of topological groups $S^1\to \mathrm{Diff}^+(S^1)$. 

(b) One can give a topological proof, or a proof based on the universal property of $\mo{Cob}$. 

\textit{Topological proof:} 
Fix an oriented compact $1$-dimensional manifold with boundary  $M$, viewed as a cobordism from $\partial_1 M$ to $\partial_2 M$ with associated decomposition of the boundary of $M$ as $\partial M=\overline{\partial_1 M}\coprod \partial_2 M$, and write $\partial_2 M=\coprod_{1}^{r}X\coprod \coprod_{1}^{s}X^\vee$, and $\partial_1 M=\coprod_{1}^{i} X\coprod \coprod_{1}^{j}X^\vee$.  

The oriented count of the size of the boundary of $M$ is zero by the classification of compact 1-manifolds.  For our given $M$, this oriented count is nothing but $r-s+j-i=(r-i)-(s-j)$, which is required to be zero.  But $(r-i)-(s-j)$ is zero if and only if $r-i=s-j$. 

\textit{Universal property proof:} Consider the wide subcategory $C$ of $\mo{Cob}$ spanned by morphisms $\X{i}{j}\to \X{r}{s}$ with $i-j=r-s$. It is clear that these morphisms are closed under composition, and under tensor products, and all symmetry/associator equivalences, so that $C\to \mo{Cob}$ is a symmetric monoidal functor. Furthermore, $C$ contains the evaluation and coevaluation, $X\otimes X^\vee \to \bb{1}$ and $\bb{1} \to X^\vee \otimes X$, so $X\in C$ is dualizable. It follows that the identity of $\mo{Cob}$ factors through $C$, which proves the result. 

(c) Every $1$-dimensional cobordism 
is (uniquely) a disjoint union of a simply-connected part and a disjoint union of circles. Consider then the full subspace $\Hom_{\mo{Cob}}(\X{i}{j},\X{r}{s})^{\mathrm{simple}}$ of $\Hom_{\mo{Cob}}(\X{i}{j},\X{r}{s})$ spanned by the simply-connected cobordisms. It is a space of intervals with fixed endpoints, and is therefore discrete. The above observation shows that, as a $\Hom_{\mo{Cob}}(\bb{1},\bb{1})$-module, $\Hom_{\mo{Cob}}(\X{i}{j},\X{r}{s})$ is free on $\Hom_{\mo{Cob}}(\X{i}{j},\X{r}{s})^{\mathrm{simple}}$.
\end{proof}

Next, we note that ``linearization'' is an easy process: 
\begin{proposition}\label{proposition:descfreeI}
The free rigid stably symmetric monoidal $\infty$-category over $\cat{D}^b(\bb{Q})$ on an object is given by the category 
    \[
        \mo{Fun}((\mo{Cob})^{op},\cat{D}(\bb{Q}))^{\omega},
    \]
with the Day convolution symmetric monoidal structure.

In particular, it is generated under finite colimits, desuspensions and retracts by images of the Yoneda embedding, whose mapping spectra are $\bb{Q}$-homology spectra of the mapping spaces in $\mo{Cob}$. 
\end{proposition}
More generally, following \cite{HA}, the next result is explicitly spelled out (in the case of spectra, but works equally well over an arbitrary stable base such as $\cat{D}(\bb{Q})$) in \cite[Construction 2.7]{aoki2025higherzariskigeometry}:
\begin{proposition}\label{proposition:descFree}
Let $C$ be a small symmetric monoidal category. The free rational stably symmetric monoidal category with a symmetric monoidal functor from $C$ is $\mo{Fun}(C^{op},\cat{D}(\bb{Q}))^\omega$ equipped with the Day convolution structure. 

It is generated under finite colimits, desuspensions and retracts by images of the Yoneda embedding, whose mapping spectra are $\bb{Q}$-homology spectra of the mapping spaces in $C$. 
\end{proposition}

\begin{notation}
We define the \emph{(rational) affine line} as
    \[
        \A\coloneqq\mo{Fun}((\mo{Cob})^{op},\cat{D}(\bb{Q}))^{\omega},
    \]
and let $y\colon\mo{Cob}\to \A$ denote the $\bb{Q}$-linearized Yoneda embedding. 
\end{notation}
\begin{remark}
To be careful, one should add a $\bb{Q}$-subscript to this notation, but no non-rational 2-ring will appear in Part I of this paper, and we will not use this notation in Part II, save for when we are talking only about rational categories, so there is no risk for confusion.
\end{remark}

\begin{lemma}\label{lemma:endunit}
The graded endomorphism ring of the unit in $\A$ is given by $\pi_*(\bb{1})=\bb{Q}[t][t_1,t_2,t_3,\ldots]$ with $|t|=0$ and $|t_i|=2i$ for $i>0$.
\end{lemma}
\begin{proof}
By \Cref{proposition:descfreeI} and the ($\bb{Q}$-linear) Yoneda embedding, the graded endomorphism ring of the unit is given by 
    \[
        \pi_*(\bb{1})=\bb{Q}(\mo{End}_{\mo{Cob}}(\bb{1})).
    \]
By \Cref{proposition:omnibusCob}, this mapping space is given by    
    \[
        \mo{Free}_{\bb{E}_{\infty}}(BS^1).
    \]
The $\bb{Q}$-homology of this space is in turn given by the free $\bb{E}_{\infty}$-$\bb{Q}$-algebra on $\bb{Q}[BS^1]\simeq\bb{Q}[x]$ with $|x|=2$, so we get that 
    \[
        \mo{End}_{\A}(\bb{1})\simeq \bb{Q}[\bb{Q}[x]],
    \]
which is a free algebra on generators in degree $2n$ for each $n\geq 0$.
\end{proof}

In the next section we will describe more generally the graded mapping spectra between various generators, see \Cref{lemma:end=br}.

\begin{corollary}\label{cor:BalmerSpcBig}
The Balmer spectrum $\mo{Spc}(\A)$ surjects onto the spectrum of homogeneous prime ideals in the graded polynomial ring $\bb{Q}[t_0,t_1,t_2,\ldots].$
\end{corollary}
\begin{proof}
Combine \Cref{lemma:endunit} with the fact that the Balmer spectrum of a tt-category surjects onto the graded spectrum of the unit, at least when the latter is coherent, see \cite[Theorem 7.3]{BalmerSpSpSp}. 
\end{proof}

The fiber of the comparison map over the generic point of $\Spec(\pi_*(\bb{1}))$ appears as the Balmer spectrum of the following 2-ring:
    \begin{notation}
        We let $\A_\eta$ be the quotient of $\A$ by the thick tensor ideal generated by $\mo{cofib}(a\colon\Sigma^{i}\bb{1}\to\bb{1})$ for all nonzero homogeneous elements $a\in \pi_i(\bb{1})$. 
    \end{notation}
    
\begin{remark}
    Equivalently, $\A_\eta$ can be defined as as the basechange of $\A$ along the map of commutative ring spectra $\End_{\A}(\bb{1})\to \End_{\A}(\bb{1})[a^{-1}, a\in \pi_k(\bb{1})\setminus\{0\}, k\in \bb{N}]$. Here, the subscript $\eta$ on $\A_\eta$ is supposed to make one think of a generic point. 
\end{remark}

\subsection{The free symmetric monoidal category on a pointed dualizable object}
The second free construction we will need to analyze is the free symmetric monoidal category on a pointed dualizable object, i.e., on a dualizable object $X$ equipped with a map $\bb{1}\to X$. We will denote this category by $\mo{Cob}^+$.

This category clearly receives a morphism $i\colon\mo{Cob}\to \mo{Cob}^+$ picking out the object $X$, and it is not hard to prove that this morphism is essentially surjective, so one can view the objects of $\mo{Cob}^+$ as oriented $0$-dimensional manifolds as well. 

Heuristically, we may thus think of $\mo{Cob}^+$ as a cobordism category, where two new types of cobordisms are allowed: the half open intervals $(-\infty,0]$ and $[0,\infty)$ as cobordisms from $\emptyset$ to $+$ and $-$ to $\emptyset$ respectively, though this is harder to prove rigorously. A precise proof will appear in forthcoming work of Barkan and Steinebrunner \cite{barkansteinebrunner2} as a special case of their generalized $1$D cobordism hypothesis. We will, however, not need the full strength of this precise identification and only certain properties of $\mo{Cob}^+$ that can actually establish more easily using their previous work \cite{barkansteinebrunnerSeg}.

The precise facts we will need about this category are the following: 
\begin{proposition}\label{proposition:omnibuscob+}
The free symmetric monoidal category $\mo{Cob}^+$ on a pointed dualizable object $X$ and the induced functor $i\colon\mo{Cob}\to \mo{Cob}^+$ have the following properties: 
\begin{enumerate}
\item $i$ is essentially surjective; 
\item $i$ induces equivalences of mapping spaces 
    \[
        \Hom_{\mo{Cob}}(\X{i}{j},\X{r}{s})\to\Hom_{\mo{Cob}^+}(\X{i}{j},\X{r}{s})
    \]
whenever $r-s\leq i-j$. In particular, they are both empty when $r-s < i-j$. 
\end{enumerate}
In particular, $i$ is an equivalence on endomorphisms of the unit. 
\end{proposition}

 We delay the proof to the end of the section, and instead indicate how we will use $\mo{Cob}^+$.

From \Cref{proposition:descFree}, we also have the following special case: 
\begin{proposition}\label{proposition:descfreeII}
The free rigid stably symmetric monoidal $\infty$-category over $\cat{D}^b(\bb{Q})$ on a pointed object is given by the category 
    \[
        \mo{Fun}((\mo{Cob}^+)^{op},\cat{D}(\bb{Q}))^{\omega},
    \]
with the Day convolution symmetric monoidal structure.

In particular, it is generated under finite colimits, desuspensions and retracts by images of the Yoneda embedding, whose mapping spectra are $\bb{Q}$-homology spectra of the mapping spaces in $\mo{Cob}^+$. 
\end{proposition}
\begin{notation}
We let  
    \[
        \Apoi\coloneqq\mo{Fun}((\mo{Cob}^+)^{op},\cat{D}(\bb{Q}))^{\omega}
    \]
and let $y\colon\mo{Cob}^+\to \Apoi$ denote the $\bb{Q}$-linearized Yoneda embedding. 

We also abuse notation and let $i\colon\A\to \Apoi$ denote the $\bb{Q}$-linearization of the functor $i\colon\mo{Cob}\to \mo{Cob}^+$. 
\end{notation}

\begin{construction}\label{construction:iota}
Since $\Apoi$ is free on a pointed object, there is a 2-ring map   
    \[
        \iota\colon\Apoi\to\Apoi
    \]
which takes our free pointed object $f\colon\bb{1}\to X$ to the dual of its fiber $g^{\vee}\colon\bb{1}\to y^{\vee}$ for some choice of fiber $g$ of $f$ (the choice of which is unique up to a contractible space of choices).
\end{construction}
We have:
\begin{proposition}\label{proposition:involution}
The functor $\iota$ is an auto-equivalence.  Moreover, $\iota$ is an involution in the sense that $\iota\circ\iota\simeq id_{\Apoi}$.
\end{proposition}
\begin{proof}
Since $\iota$ is exact and symmetric monoidal, we have that 
    \[
        \iota(g)\simeq \mo{fib}(\iota(f))=\mo{fib}(g^{\vee})\simeq f^{\vee},
    \]
so that 
    \[
        \iota(\iota(f))\simeq \iota(g^{\vee})\simeq \iota(g)^{\vee}\simeq f.
    \]
Since $\Apoi$ is free on the pointed dualizable object $f\colon\bb{1}\to X$, the equivalence $\iota(\iota(f))\simeq f$ extends to an equivalence $\iota\circ \iota\simeq id_{\Apoi}$.
\end{proof}
This functor will be used crucially in the sequel to show that a local category obtained from $\Apoi$ (specifically, $\Apoi_\eta \coloneqq\Apoi\otimes_{\A} \A_\eta$) does not satisfy the exact-nilpotence condition from \Cref{def:ENC}.
\begin{remark}
Taking into account the forthcoming work of Barkan--Steinebrunner \cite{barkansteinebrunner2}, $\mo{Cob}^+$ has a geometric description. 
It seems rather nontrivial to define the functor $\iota$ while referring only to this geometric description, while its description in terms of universal properties is evident. 
\end{remark}

We conclude this section with the proof of \Cref{proposition:omnibuscob+}. 

Recall that for us, $\mo{Cob}^+$ is defined as the free symmetric monoidal category on a pointed object, and not as a geometric object. 

The first thing we will need is the description of the free symmetric monoidal category on a pointed object, with no dualizability condition:
\begin{lemma}[{\cite[Lemma 2.1]{logan}}]\label{lemma:freepoinodbl}
The free symmetric monoidal category on an object $X$ is $\Fin^\simeq$, and the free symmetric monoidal category on an object $X$ equipped with a map $\bb{1}\to X$ is $\Fin^\inj$, and the induced map sending $X$ to $X$ is the canonical inclusion $\Fin^\simeq\to\Fin^\inj$. \end{lemma}

\begin{corollary}
The canonical maps $\Fin^\simeq\to \mo{Cob}$ and $\Fin^\inj\to \mo{Cob}^+$ coming from \Cref{lemma:freepoinodbl} fit in a pushout square of symmetric monoidal categories: 
\[\begin{tikzcd}
	{\Fin^\simeq} & \mo{Cob} \\
	{\Fin^{\inj}} & {\mo{Cob}^+}
	\arrow[from=1-1, to=1-2]
	\arrow[from=1-1, to=2-1]
	\arrow[from=1-2, to=2-2]
	\arrow[from=2-1, to=2-2]
\end{tikzcd}\]
\end{corollary}
\begin{proof}
This is immediate by comparing universal properties. 
\end{proof}

The idea of the proof of \Cref{proposition:omnibuscob+} is then to start by computing the pushout above in commutative monoids in simplicial spaces, and analyze how far the result is from being a Segal space. In turn, the latter analysis will use Barkan and Steinebrunner's formula for Segalification, cf. \cite{barkansteinebrunnerSeg}.

To analyze the pushout in simplicial spaces, we first need the following elementary bit of combinatorics:
\begin{lemma}\label{lemma:injfree}
For every $[n]\in\Delta$, the map 
    \[
        (\Fin^\simeq)\times ((\Fin^\inj)^{[1,n]})^\simeq\simeq (\Fin^\simeq)^{[n]}\times ((\Fin^\inj)^{[1,n]})^\simeq \to ((\Fin^\inj)^{[n]})^\simeq
    \]
defined by $(X,X_1\hookrightarrow X_2\hookrightarrow ... \hookrightarrow X_n)\to (X\hookrightarrow X\coprod X_1 \hookrightarrow X\coprod X_2 \hookrightarrow ... \hookrightarrow X\coprod X_n)$ is an equivalence of $\Fin^\simeq \simeq (\Fin^\simeq)^{[n]}$-modules. 
\end{lemma}
\begin{notation}
Let $N$ denote the Rezk nerve embedding categories into simplicial spaces via $NC\colon [n]\mapsto \Hom([n],C)$, and let $L_{\Seg}$ denote the Segalification functor from simplicial spaces to Segal spaces\footnote{This is almost left adjoint to $N$, up to the distinction between Segal spaces and complete Segal spaces. Since completion does not change the mapping spaces \cite[Corollary 3.15]{hebestreit2025short}, this distinction will be irrelevant for us.}.

We abuse notation by using the same symbol for the induced functor on commutative monoids. 
\end{notation}
\begin{definition}
Let 
    \[
        P^+ \coloneqq N(\mo{Cob})\coprod_{N(\Fin^\simeq)}N(\Fin^\inj)
    \]
denote the pushout in commutative monoids in simplicial spaces. 
\end{definition} 

\begin{remark}
The canonical map $L_\Seg(P^+)\to \mo{Cob}^+$ is a completion, and in particular mapping spaces in $\mo{Cob}^+$ are the same as in $L_\Seg(P^+)$.
\end{remark}
\begin{corollary}\label{corollary:levelwiseff}
The map $N(\mo{Cob})\to P^+$ is levelwise a monomorphism of spaces. 
\end{corollary}
\begin{proof}
The canonical map $\mo{Cob}\to P^+$ is levelwise given by   
    \[
        (\mo{Cob}^{[n]})^\simeq \to (\Fin^\inj)^{[n]}\otimes_{(\Fin^\simeq)^{[n]}}(\mo{Cob}^{[n]})^\simeq
    \]
Hence by \Cref{lemma:injfree}, it is equivalent to the inclusion        
    \[
        (\mo{Cob}^{[n]})^\simeq \to (\Fin^\inj)^{[1,n]} \times (\mo{Cob}^{[n]})^\simeq
    \]
at the component of $\emptyset\to ...\to \emptyset$. 

Since the component of $\emptyset\to ... \to \emptyset$ in $(\Fin^\inj)^{[1,n]}$ is contractible, we find that the canonical map $N(\mo{Cob})\to P^+$ is levelwise an inclusion of components, as claimed.
\end{proof}
We now need to discuss this monomorphism in more detail. For this, we need to recall the notion of necklace from \cite{barkansteinebrunnerSeg}: 
\begin{definition}[{\cite[Definition 1.10]{barkansteinebrunnerSeg}}]\label{definition:necklace}
A \emph{necklace} is a bipointed simplicial space equivalent to a wedge of the form $\Delta^{n_1}\vee ... \vee \Delta^{n_k}$, where all the wedges $\Delta^n\vee \Delta^m$ are taken at the point $n\in\Delta^n, 0\in\Delta^m$,  bipointed at the image of $0\in\Delta^{n_1}, n_k\in \Delta^{n_k}$. 
\end{definition}
The relevance of this notion is that if $x,y\in X_0$ are points in the $0$ simplices of a simplicial space $X$, we can compute the mapping space in $L_\Seg X$ from $x$ to $y$ as a colimit along necklaces $N$ of mapping spaces of bipointed simplicial spaces $N\to (X,x,y)$, cf. \cite[Theorem A]{barkansteinebrunnerSeg}. 

We now need to make the following observation: 
\begin{lemma}
Let $i,j$ be integers and consider $\X{i}{j}\in \mo{Cob}$ as a point in $N(\mo{Cob})_0$, and hence in $(P^+)_0$. Let $N$ be a necklace in the sense of \Cref{definition:necklace}. 

A map $f\colon N\to P^+$ of simplicial spaces factors through\footnote{Necessarily uniquely, by \Cref{corollary:levelwiseff}.} $N(\mo{Cob})$ if and only if it sends the basepoints of $N$ to $\X{i}{j}, \X{r}{s}$ for some integers $i,j,r,s$ with $r-s\leq i-j$. 
\end{lemma}
\begin{proof}
A necklace $N$ is by definition isomorphic to the iterated wedge $\Delta^{n_1}\vee ... \vee \Delta^{n_k}$ for a sequence of natural numbers $n_1,...,n_k$, and a map $N\to P^+$ is thus the data of a point in $(P^+)_{n_1} \times_{(P^+)_0} ... \times_{(P^+)_0}(P^+)_{n_k}$. 

Unwinding the equivalence $(P^+)_n\simeq ((\Fin^\inj)^{[1,n]})^\simeq \times (\mo{Cob}^{[n]})^\simeq$ we find the evaluation at $0\in\Delta^n$ amounts to projection on $\mo{Cob}^{[n]}$ followed by evaluation at $0$, while evaluation at $n\in\Delta^n$ amounts to evaluating both terms at $n$, and using the obvious sum operation $\Fin^\simeq\times \mo{Cob}\to \mo{Cob}$.

So let $(a_1,...,a_k)$ be a point in $(P^+)_{n_1}\times_{(P^+)_0} ... \times_{(P^+)_0}(P^+)_{n_k}$ where $a_q$ has endpoints $(\X{i_q}{j_q},\X{r_q}{s_q})\in \mo{Cob}^\simeq \times \mo{Cob}^\simeq$. We learn from the pullback condition that $r_q=i_{q+1}, s_q=j_{q+1}$ and that $\X{r_q}{s_q} = \X{n}{m}\otimes \X{t}{0}$ for some $t\geq 0$ and some $n,m$ so that there exists a map $\X{i_q}{j_q}\to \X{n}{m}$ in $\mo{Cob}$. In particular, $m=s_q, n\leq r_q$. From the latter, we learn that $i_q-j_q = n-m \leq r_q-s_q$. For them to be equal, $t$ has to be $0$, which means that the last term of the sequence of injections in $(\Fin^\inj)^{[1,n]}$ is empty, so all the terms need to be empty. It follows that all the terms are in the image of $\mo{Cob}$ (see \Cref{corollary:levelwiseff}), as was claimed.
\end{proof}

We obtain the following from \cite[Theorem A]{barkansteinebrunnerSeg}, which is exactly item (b) in \Cref{proposition:omnibuscob+}:
\begin{corollary}
The map $\mo{Cob}= L_\Seg(N(\mo{Cob}))\to L_\Seg(P^+) \to \mo{Cob}^+$ is fully faithful on $\hom(\X{i}{j},\X{r}{s})$ whenever $r-s\leq i-j$. 
\end{corollary}
\begin{proof}
By \cite[Theorem A]{barkansteinebrunnerSeg}, the mapping space in $L_\Seg(P^+)$ (and hence in $\mo{Cob}^+$) from $\X{i}{j}$ to $\X{r}{s}$ can be computed as a colimit over all necklaces $N$ of maps $N\to P^+$ sending the basepoints to $\X{i}{j}, \X{r}{s}$ respectively. By the previous lemma, the conditions on $(i,j,r,s)$ guarantee that this is the same as the colimit of maps $N\to N(\mo{Cob})$ sending the basepoints to the objects with the same name, which is exactly $\Hom_{\mo{Cob}}(\X{i}{j},\X{r}{s})$, as was to be shown. 
\end{proof}
\begin{proof}[Proof of \Cref{proposition:omnibuscob+}]
The previous corollary is exactly item (b), so we are left with item (a). Consider the full subcategory of $\mo{Cob}^+$ spanned by the image of $\mo{Cob}$. This is a symmetric monoidal subcategory, and it contains the pointed object $X$. 

Hence, by the universal property of $\mo{Cob}^+$, the identity of $\mo{Cob}^+$ factors through this subcategory, which proves that they are equal.
\end{proof}

\subsection{A brief digression, and an example}  Using the free symmetric monoidal rigid rational additive 1-category on a pointed object, which we can describe from the above construction, we can present the promised example of a tt-category with all non-zero compact objects being $\otimes$-faithful, but which acts more akin to a nil-extension of a field than an actual tt-field itself.
\begin{remark}\label{rem:hcob}
Consider the free idempotent complete rigid symmetric monoidal $\bb{Q}$-linear additive 1-category on a pointed object, which we will denote by $\bb{Q}[\mo{hCob}^{+}]$ to connect it to \Cref{proposition:omnibuscob+}, and localize this category at the generic point of the unit (a polynomial ring in a variable $t$), to get the category we will call $\bb{Q}[\mo{hCob}^{+}]_{\eta}$.  The indecomposable objects in this category are all summands of objects of the form $\X{i}{j}$ (and are in bijection with the simple objects of $\mathrm{Rep}(GL_t;\bb{Q}(t))$).  If we say that a summand of such an object has grading $i-j$, this gives a well-defined grading on the indecomposable objects in this category, in such a way that the non-zero morphisms are grading non-decreasing.  Let $\mc{I}_2\subseteq \bb{Q}[\mo{hCob}^{+}]_{\eta}$ denote the additive $\otimes$-ideal of morphisms which increase grading by at least 2, and consider the additive quotient 
\[\cat{C}\coloneqq \bb{Q}[\mo{hCob}^{+}]_{\eta}/\mc{I}_2.\]
We claim that every non-zero compact object in $K^b(\cat{C})$ is $\otimes$-faithful.  Since we can further quotient out the ideal $\mc{I}_1$ of morphisms of grading at least 1 (in which case we get back to $\mathrm{Rep}(GL_t;A)$, which is semi-simple by \Cref{thm:Deligne}), this category acts as a sort of ``nil-extension of a field.''

We claim that every non-zero compact object of $K^b(\cat{C})$ is $\otimes$-faithful.  Since this category is Krull-Schmidt, it suffices to show that the dimension of any indecomposable object is invertible.  Since morphisms in $\cat{C}$ of grading zero are split (which follows from \Cref{thm:Deligne}), and morphisms of grading $\geq 2$ vanish, we can kill off contractible chain complexes to see that every indecomposable object $Z_{\bdot}\in K^b(\cat{C})$ is represented by a chain complex of the form
\[\ldots \to Z_i\to Z_{i-1}\to\ldots,\]
with each $Z_i$ either zero or a sum of indecomposable objects all with grading $k-i$ for some fixed $k$ depending only on $Z_{\bdot}$.  The formula given at the end of \cite{deligne1996serie} (see also \Cref{cor:DeligneParityArgument}) shows that the dimension of each $Z_i$ is a polynomial, say $f_{i}(t)$, with positive leading coefficient, such that the parity of the degree of $f_i(t)$ equals the parity of $k-i$.  In particular, we have 
\[\mo{dim}(Z_{\bdot})=\sum_{i\in\bb{Z}}(-1)^{i}f_i(t)\neq 0,\]
where the sum is non-zero since we can look at the $f_i(t)$s with highest degree, and we see that the indices where such $f_i(t)$s appear all have the same parity.  Since the dimension of this object is non-zero, and the endomorphism ring of the unit in this category is a field, the dimension of every non-zero indecomposable object is invertible, as claimed.
\end{remark}

\newpage
\section{The Affine Line $\A$}\label{section:affineline}

\subsection{Semi-Simplicity of $\A_{\eta}$} 
The goal of this subsection is to prove that the category $\A_{\eta}$ defined in \Cref{sec:freeconstructions} is semi-simple, and is even a tt-field in the sense of \cite[Definition 1.1]{BKS2019}. Towards this end, we first have to study the mapping spaces of generators, beginning with the case of the affine line.

In the classical language of tensor categories, the graded mapping spectra between various generators in $\A$ correspond to the endomorphism rings of the generators in Deligne's category $\mathrm{Rep}(GL_t)$ \cite{deligne}, recalled in \Cref{sec:appendix}: 

\begin{lemma}\label{lemma:end=br}
Let $i,j$ be natural numbers. There is an isomorphism of graded rings 
    \[
        \pi_*\End_{\A}(\X{i}{j}) \cong \pi_*(\bb{1})\otimes_{\pi_0(\bb{1})} \End_{\mathrm{Rep}(GL_t)}(\X{i}{j})_{\mathbb Q}.
    \]
In particular, the graded endomorphism ring is concentrated in even degrees.

More generally, for any family of indices $i_k,j_k$, 
\[
    \pi_*\End_{\A}(\bigoplus_k \X{i_k}{j_k}) \cong \pi_*(\bb{1})\otimes_{\pi_0(\bb{1})} \End_{\mathrm{Rep}(GL_t)}(\bigoplus_k \X{i_k}{j_k})
\]
and both sides are concentrated in even degrees.  
\end{lemma}
\begin{proof}
It follows from \Cref{proposition:omnibusCob} and \Cref{proposition:descfreeI} that for any $i_0,j_0,i_1,j_1$, $\Hom_{\A}(\X{i_0}{j_0},X{i_1}{j_1})$ is free on a discrete set as a module over $\End_{\A}(\bb{1})$, and thus this is preserved by passing to $\pi_*$. Specifically, it is free on the set of (isomorphism classes of) simply-connected cobordisms from $\X{i_0}{j_0}$ to $\X{i_1}{j_1}$. Thus its $\pi_*$ is indeed basechanged from its $\pi_0$, and we simply have to identify $\pi_0$. 

The result then follows from the definition of Deligne's $\mathrm{Rep}(GL_t)$ as the free additively symmetric monoidal $1$-category on a dualizable object, cf. \Cref{defn:repglt}, or directly from Deligne's original definition \cite[Definition 10.2]{deligne}. 
\end{proof}

\begin{remark}
    In representation theoretic language, these graded endomorphism rings are isomorphic to the ``walled Brauer algebras $B_{i,j}(t)$ with parameter $t$ on the ring $\pi_*(\bb{1})\cong \bb{Q}[t,t_1,t_2,...]$, see \cite[Section 2]{Brauersemisimplicity}. Indeed, it is clear that the set of these cobordisms is in bijection with the generators of $B_{i,j}(t)$ from \cite{Brauersemisimplicity}, and each connected component arising in the multiplication of two such generators corresponds to a dimension of $X$, i.e., to a multiplication by $t$, so the multiplication of these generators also corresponds to the one on $B_{i,j}(t)$. 
\end{remark}

\begin{lemma}\label{lemma:end=brgeneric} With notation as in \Cref{sec:freeconstructions}, for any $i,j$, the graded endomorphism ring $\pi_*\End_{\A_{\eta}}(\X{i}{j})$ is isomorphic to is isomorphic to the endomorphism ring of $\X{i}{j}$ in Deligne's $\mathrm{Rep}(GL_t, \pi_*(\bb{1}_{\A_\eta}))$, and similarly for finite direct sums of generators. 

Equivalently, it is the walled Brauer algebra $B_{i,j}(t)$ with parameter $t$ in the graded field $\pi_*(\bb{1})\simeq \bb{Q}(t,t_1,t_2\ldots)$.  In particular, the ring $\pi_0(\End_{\A_{\eta}}(\X{i}{j}))$ is isomorphic to the walled Brauer algebra $B_{i,j}(t)$ with parameter $t$ in the field $\pi_0(\bb{1})$.
\end{lemma}
\begin{proof}
This follows from \Cref{lemma:end=br} since $\A_{\eta}$ is obtained from $\A$ by localizing at the prime ideal $(0)$ in the unit of $\A$.
\end{proof}

\begin{proposition}\label{prop:Brsemisimp}
The algebra $\pi_0(\mo{End}_{\A_{\eta}}(\X{i}{j}))$ is semi-simple for all $i,j$. In fact, for any finite family of indices $(i_s,j_s), s\in S$, the algebra $\pi_0(\mo{End}_{\A_{\eta}}(\bigoplus_{s\in S}\X{i_s}{j_s}))$ is semi-simple. 
\end{proposition}
\begin{proof}
Using \Cref{lemma:end=brgeneric}, this follows from \cite[Théorème 10.5]{deligne}, a proof of which is explained in \Cref{thm:Deligne}.

\end{proof}

\begin{lemma}\label{lemma:sillysummand}
Let $C$ be an additive category.
\begin{itemize}
\item Let $x\in C$ such that $\End_C(x)$ is semi-simple. For any summand $y$ of $x$, $\End_C(y)$ is also semi-simple.

\item If $x,z$ are orthogonal, and $\End_C(x),\End_C(z)$ are both semi-simple, then so is $\End_C(x\oplus z)$.

\end{itemize} 

\end{lemma}
\begin{proof}
For the first item, we note that if $e\in \End_C(x)$ is an idempotent corresponding to $y$, we have $\End_C(y)\cong e\End_C(x)e$, so it suffices to focus on the latter. Now, by Artin--Wedderburn, we may assume that $\End_C(x)\cong M_n(D)$ for some division algebra $D$, so that $e$ is the projection onto some sub-vector space $V\subset D^n$. Then, $eM_n(D)e \cong \End_D(V)\cong M_k(D)$ for some $k$ is semi-simple. 

For the second item, if $x,z$ are orthogonal, then $\End_C(x\oplus z)\cong \End_C(x)\times\End_C(z)$ and products of semi-simple rings are semi-simple. 
\end{proof}

\begin{theorem}\label{thm:A1etasemisimple}
The category $\A_{\eta}$ is semi-simple with simple unit.  In particular, this category is a tt-field, and there is a unique point of $\Spc(\A)$ over the generic point of $\mo{Spec}(\pi_*(\bb{1}_{\A}))$.
\end{theorem}
\begin{proof}
By the Yoneda lemma and our explicit description of $\A$ from \Cref{sec:freeconstructions}, $\A$, and hence $\A_{\eta}$, is generated under finite colimits and retracts by the image of the Yoneda embedding, that is, the objects $\X{i}{j}$. 

Consider the subcategory $C$ of $\ho(\A_\eta)$ spanned by sums of the form $y\oplus \Sigma z$, where $y,z$ are summands of objects of the form $\bigoplus_{s\in S}\X{i_s}{j_s}$ for some finite family $(i_s,j_s), s\in S$ of indices. Note that since all the generators have even periodic homotopy rings, $y$ and $\Sigma z$ are orthogonal in $\ho(\A_\eta)$ for any such $y,z$. 

By \Cref{prop:Brsemisimp} and \Cref{lemma:sillysummand}, every object in $C$ has a semi-simple endomorphism ring. By $2$-periodicity, $C$ is closed under shifts as well, and it is clearly closed under retracts. 

This implies directly that its preimage in $\A_{\eta}$ is closed under co/fibers, since any map between objects of a semi-simple category is isomorphic to one of the form $\id\oplus 0\colon x\oplus y\to x\oplus z$.  Thus this preimage is equal to $\A_{\eta}$, which proves the claim in light of \cref{lem:ssimpleimpllocal}.
\end{proof}

\subsection{Points of $\Spc(\A)$}\label{subsection:points}
We note that all the arguments in this section directly generalize to $\A_K:= \A\otimes_{\End(\bb{1}_{\A})} K$ for any map of commutative ring spectra $\End(\bb{1}_{\A})\to K$, where $K$ is a graded field and on $\pi_0, \bb{Q}[t]\to K$ does not kill $(t-n)$ for any integer $n\in\bb{Z}$.

We also observe the following :
\begin{construction}
Since $\End(\bb{1}_{\A})$ is free as a rational commutative ring spectrum on classes in even degrees, it is also \emph{formal}, i.e., equivalent to the commutative ring presented by the cdga $\bb{Q}[t, t_i,\in \bb{N}_{\geq 1}], |t_i|=2i$ with $0$ differentials.  It follows that for any homogeneous prime $\mathfrak p\subset \pi_*\End(\bb{1}_{\A})$, one can realize the map $\pi_*\End(\bb{1}_{\A})\to \pi_*\End(\bb{1}_{\A})/\mathfrak p $ as $\pi_*$ of a map of (formal) commutative ring spectra, and a fortiori also $\pi_*\End(\bb{1}_{\A})\to \mathrm{Frac}(\pi_*\End(\bb{1}_{\A})/\mathfrak p)$ can be realized as a map of formal commutative ring spectra (here, $\mathrm{Frac}$ is the graded fraction field, obtained by inverting all nonzero homogeneous elements). 

We denote the outcome of this construction by $\End(\bb{1}_{\A})\to K(\mathfrak p)$. 
\end{construction}
\begin{remark}
In characteristic $0$, all fields are formally smooth and so by deformation theory any graded field $K$ as above is equivalent, as a commutative ring spectrum, to $L$ or $L[u^{\pm 1}]$ for some ordinary field $L$ and $|u| =2k, k\geq 1$; either way, it is formal. Thus, possibly up to reparametrization of $u$, any map $\End(\bb{1}_{\A})\to K$ to a graded field factors weakly uniquely over $K(\mathfrak p)$ for some homogeneous prime ideal $\mathfrak p$. 

Another construction of $K(\mathfrak p)$ is given by Mathew in \cite[Theorem 1.2]{mathewresidue}, up to the (in this case, minor) issue that $\bb{Q}[t,t_i,i\in \bb{N}_{\geq 1}]$ is not noetherian. He also discusses uniqueness in his situation. 
\end{remark}
\begin{remark}
For the above construction, one can also apply \Cref{thm:consequalshom}, proven later on.  Interpreted for rational $\bb{E}_{\infty}$-rings $R$, this says that the points in the homological spectrum of $\mo{Perf}(R)$ are in bijection with equivalence classes of maps $R\to L$, where $L$ is an even 2-periodic rational $\bb{E}_{\infty}$-ring with $\pi_0(L)$ an algebraically closed field, under the equivalence relation $L_1\sim L_2$ iff $L_1\otimes_{R}L_2\neq 0$.  Since we are working on a free algebra on even degree classes, one can check that these equivalence classes of homological residue fields are determined by homogeneous prime ideals in the graded endomorphism ring of the unit, and a residue field exists for each such prime.
\end{remark}
\begin{corollary}\label{cor:A1pointsdescribe}
Let $S$ be the set of primes in $\bb{Q}[t]$ of the form $(t-n)$.  On the complement of the inverse image of $S$ under the map 
\[
    \Spech(\pi_*\bb{1}_{\A})\to \Spec(\bb{Q}[t])
\]
induced by $\bb{Q}[t]\to \bb{Q}[t,t_i\colon i\in \bb{N}_{\geq 1}]$, the map 
\[
    \Spc(\A) \to \Spech(\pi_*\bb{1}_{\A})
\]
is a bijection. 
\end{corollary}
\begin{proof}
By \Cref{prop:nilconsbasechange} and the discussion preceding the construction, the fiber over $\mathfrak p\in \Spech(\pi_*\bb{1}_\A)$ is the Balmer spectrum of a tt-field, and hence is reduced to a point. 
\end{proof}

\newpage
\section{The Comparison Functor}\label{sec:ComparisonFunctor}
There is a functor $\Apoi\to\A$ taking the free pointed object to the map $\bb{1}\xrightarrow{0}x$, that is, to the zero map to the free object.  This induces an equivalence on endomorphism rings of the unit, and so we may invert all endomorphisms of the unit and get an induced functor of 2-rings $p\colon\Apoi_\eta\to \A_\eta$.  Our goal in this section is to leverage the functor $p$ to get information about $\Apoi_\eta$ from information about $\A_\eta$. 

The key result is:
\begin{proposition}\label{proposition:conservative}
The functor $p\colon\Apoi_\eta\to \A_\eta$ is conservative.
\end{proposition}
As an immediate corollary, we have:
\begin{corollary}
The 2-ring $\Apoi_\eta$ is local. 
\end{corollary}
\begin{proof}
It follows directly from \Cref{thm:A1etasemisimple} that $\A_\eta$ is local, and it is clear that if a tt-category has a conservative symmetric monoidal functor to a local tt-category, it is itself local. 
\end{proof}
Thus, $\Apoi_\eta$ is a potential counterexample, and is in fact one of the universal examples (the universal examples to test the exact nilpotence condition are exactly the $\Apoi/\mathcal{P}, \mathcal{P}\in \Spec(\Apoi)$). 

Before the proof, we need a few lemmas. 

\begin{lemma}
Let $S$ be a non-negatively graded ring spectrum, and $M,N$ two graded $S$-modules. Suppose $M$ is in gradings $\geq n+1$, and $N$ in gradings $\leq n$. Then any map $M\to N$ of graded $S$-modules is null. 

The same is true for bimodules. 
\end{lemma}
\begin{proof}
The bimodule statement has either the same proof, or follows from the module case by using $S\otimes S^{\mathrm{op}}$, so we only treat the module case. 

First, we note that without the $S$-module structure, this is obvious: any map of graded spectra $M\to N$ is $0$. 

Second, we note that since $S$ is non-negatively graded, all tensors $S^{\otimes k}\otimes M$ are in gradings $\geq n+1$. Using the bar resolution for $M$ and the case of graded spectra with no $S$-module structure, we obtain the result. 
\end{proof}

\begin{corollary}\label{corollary:gradingnil}
Let $f\colon S\to R$ be a map of graded ring spectra, where $R,S$ are nonnegatively graded, and suppose furthermore that $S$ is in gradings $\leq n$ for some $n$. If the fiber $I$ is in gradings $\geq 1$, then for $k\geq n+1$, the map $I^{\otimes_S k}\to S$ is null as a map of graded $S$-bimodules, and hence also when forgetting the grading. 
\end{corollary}
\begin{proof}
Since $S$ is nonnegatively graded, $I^{\otimes_S k}$ is in gradings $\geq k\geq n+1$ and so the previous lemma applies directly. 
\end{proof}

For a map of ring spectra $S\to R$ with fiber $I$, the map $I^{\otimes_S k}\to S$ being null is a very strong nilpotence condition, and buys us the following: 
\begin{lemma}\label{lemma:nilNakayama}
Let $f\colon S\to R$ be a map of ring specta with fiber $I$. If for $n$ large enough, the map $I^{\otimes_S n}\to S$ is null as a map of $S$-bimodules, then tensoring with $R$ is conservative on $S$-modules. 
\end{lemma}
\begin{proof}
Let $M$ be an $S$-module. By assumption, for $n$ large enough, the map $I^{\otimes_S n}\otimes_S M\to M$ is null. Thus, it can only be an equivalence if $M\simeq 0$. But if $R\otimes_S M\simeq 0$, then $I\otimes_S M\simeq M$ and so $I^{\otimes_S n}\otimes_SM\simeq M$ by induction, via the canonical map. 
\end{proof}
The final lemma we will need concerns the specific structure of $\Apoi_\eta$, and suggests how we will use the previous three lemmas:

\begin{lemma}\label{lemma:differentnumbersII}
The spectrum $\Hom_{\Apoi_\eta}(\X{i}{j},\X{r}{s})$ is nonzero only when $r-s\geq i-j$.

Furthermore, if $r-s=i-j$, the map 
    \[
        \Hom_{\Apoi_\eta}(\X{i}{j},\X{r}{s})\to \Hom_{\A_\eta}(\X{i}{j},\X{r}{s})
    \]
induced by $p$ is an equivalence. 
\end{lemma}
\begin{proof}
For the first half, we note that this spectrum is a basechange along a certain ring map of the $\bb{Q}$-homology of mapping spaces in a cobordism category. It thus suffices to show that if $r-s<i-j$, this mapping space is empty, which we have in \Cref{proposition:omnibuscob+}.

For the second half of the claim, note that $p$ has a section defined as the unique 2-ring map sending $X\mapsto X$ and so it suffices to prove the analogous claim for the section. However, that section exists already at the level of cobordism categories, and since both mapping spectra are basechanges along the same ring map of $\bb{Q}$-homologies of the mapping spaces in those cobordism categories, it suffices to prove the claim at the level of cobordism categories, which is also covered in \Cref{proposition:omnibuscob+}. 
\end{proof}

We now have all that we need to prove \Cref{proposition:conservative} which, we recall, states that $p$ is conservative:
\begin{proof}[Proof of \Cref{proposition:conservative}]
 The increasing union of the thick subcategories generated by $\X{i}{j}$ for $i+j\leq n$ is equal to $\Apoi_\eta$, hence it suffices to prove conservativity when restricted to any of these. 

By the Schwede-Shipley theorem, this amounts to proving that basechange along the $\bb{E}_1$-algebra map 
    \[
        \End_{\Apoi_\eta}(\bigoplus_{i+j\leq n}\X{i}{j})\to \End_{\A_\eta}(\bigoplus_{i+j\leq n}\X{i}{j})
    \]
is conservative (on perfect modules, though our proof shows it in general). 

 Observe that we can grade naturally the objects $\bigoplus_{i+j\leq n}\X{i}{j}$ in $\Apoi_\eta$ and $\A_\eta$ respectively by putting $\X{i}{j}$ in degree $i-j$, which canonically upgrades their endomorphism rings to graded ring spectra, and the map between them to a map of graded ring spectra. 

 By \Cref{lemma:differentnumbersII}, both these graded rings are concentrated in nonnegative gradings, and furthermore that map is an equivalence on the grading $0$ piece, so that its fiber is concentrated in gradings $\geq 1$. 

 We conclude using \Cref{corollary:gradingnil} and \Cref{lemma:nilNakayama}. 
\end{proof}

\begin{theorem}\label{thm:ENCfails}
The category $\Apoi_\eta$ is local, and the exact-nilpotence condition fails for $\Apoi_\eta$, for the (image of the) universal fiber sequence $Y\to \bb{1}\to X$.

 In particular, the homological spectrum of $\Apoi_{\eta}$ is strictly larger than its Balmer spectrum.
\end{theorem}
\begin{proof}
Since $p\colon\Apoi_\eta\to\A_\eta$ is conservative, with the target a tt-field, we find that $\Apoi_\eta$ is local. 

Now, let $Y\xrightarrow{g}\bb{1}\xrightarrow{f}X$ be the fiber sequence associated to the universal arrow $f\colon\bb{1}\to X$ and suppose there exists some $z\in \Apoi_\eta$ with either $z\otimes f$ or $z\otimes g$ being $\otimes$-nilpotent.  

Recall from \Cref{construction:iota} and \Cref{proposition:involution} the involution $\iota\colon \Apoi\to\Apoi$, which sends $f$ to $g^\vee$. Since it is a 2-ring map and $\Apoi_\eta$ is defined by inverting nonzero endomorphisms of the unit, it also induces an involution on $\Apoi_\eta$, which we abusively still denote $\iota$. 

By replacing $z$ with $\iota(z)$ if necessary, we may assume that $z\otimes g$ is $\otimes$-nilpotent.  This implies then that $p(z)\otimes p(g)\simeq p(z\otimes g)$ is $\otimes$-nilpotent as well. 

 Since $z$ was assumed to be nonzero, \Cref{proposition:conservative} implies that $p(z)$ is also nonzero.  However, by definition of $p$, $p(g)$ is the fiber of $\bb{1}\xrightarrow{0}X$ in $\A_\eta$, and hence it is the projection map $\Omega X\oplus\bb{1}\to \bb{1}$, which is split surjective.  In particular, $p(z)\otimes p(g)$ is split surjective as well, with target $p(z)$, and the only way a split surjective map can be $\otimes$-nilpotent is if the target is zero, a contradiction.
\end{proof}
\begin{remark}\label{rmk:NoSfailsgeneral}
As in \Cref{section:affineline}, and the discussion in \Cref{subsection:points}, we note that a slight variant of the arguments given here also shows that for any map of commutative ring spectra $\End(\bb{1}_{\A})[\frac{1}{t-n}, n\in\bb{Z}] \to K$, where $K$ is a graded field, the exact-nilpotence condition also fails for the basechange $\Apoi_K := \Apoi\otimes_{\End(\bb{1}_{\A})}K$. 

The only point where the argument needs to be changed is when considering the involution $\iota$: instead, it induces an equivalence $\Apoi_K\simeq \Apoi_{K'}$, where $K'$ is the same graded field as $K$, but equipped with a different map $\End(\bb{1}_\A)\to K$ 
\end{remark}

\newpage

\part{Constructible Spectra in Higher Zariski Geometry}
\section{The $\bb{E}_{\infty}$-Constructible Spectrum}\label{sec:SpectraSpectra}
In \cite[Appendix A]{2022arXiv220709929B}, a constructible spectrum is constructed for objects in any suitably nice ``category of algebras,'' by generalizing the notion of an ``algebraically closed field''.  In the context of higher Zariski geometry, there is a map from the constructible spectrum to the Balmer spectrum whose image tracks exactly those primes which can be detected by maps to pointlike rigid 2-rings.  For the convenience of the reader, we recall the definitions/constructions of Nullstellensatzian objects and the constructible spectrum.
\begin{definition}[{\cite[Definition A.1]{2022arXiv220709929B}}]  Let $\mc{A}$ be a presentable category.  Then $\mc{A}$ is said to be \textit{weakly spectral} if it is compactly generated, the terminal object is compact, and any map from the terminal object is an equivalence.
\end{definition}
\begin{definition}[{\cite[Definition 1.1]{2022arXiv220709929B}}]  Let $\mc{A}$ be presentable and take any object $R\in\mc{A}$.  Then we say that $R$ is \textit{Nullstellensatzian} if $R$ is nonzero and every compact object in $\mc{A}_{R/}$ is either zero or admits a map to the initial object.
\end{definition}
These are the key ingredients going into creating a constructible spectrum in some general category.  This is summarized by the following theorem.
\begin{theorem}[{\cite[Theorem A.3/Lemma A.35]{2022arXiv220709929B}}]  Let $\mc{A}$ be a weakly spectral category.  Then there exists a unique functor 
\[\Speccons_{\mc{A}}\colon\mc{A}^{op}\to \mo{Top}^{\mo{cpt},T_1,\mo{cl}}\]
to compact $T_1$-topological spaces with closed maps between them satisfying a number of properties.  In particular, points of $\Speccons_{\mc{A}}(R)$ correspond to equivalence classes of maps $R\to S$ for $S$ Nullstellensatzian, where $S$, $S^{\prime}$ are said to be equivalent if $S\coprod_{R} S^{\prime}\neq 0$.  The topology is determined by the stipulation that the image of $\Speccons_{\mc{A}}(S)\to \Speccons_{\mc{A}}(R)$ is closed for all maps $R\to S$ in $\mc{A}$.
\end{theorem}
With this in hand, we can now discuss the relation to the homological spectrum in the rational case.
\begin{proposition}\label{prop:weaklyspectral}  The category $\ttCat$ of rigid 2-rings is weakly spectral, and so too is the category $\ttCat_{\bb{Q}}$ of rational rigid 2-rings.
\end{proposition}
\begin{proof}
The compactly generated assumption follows e.g., from the fact that $\mo{Cat}^{\mo{perf}}$ is compactly generated, and that the forgetful functor $\ttCat\to \mo{Cat}^{\mo{perf}}$ is conservative, commutes with filtered colimits, and has a left adjoint (which therefore sends compact objects to compact objects, which generate $\ttCat$).  The existence of a map from the terminal object $0$ to a rigid 2-ring $\cat{C}$ implies $\mo{End}_{\cat{C}}(\bb{1})\simeq 0$, which in turn implies that $\cat{C}\simeq 0$.

The final claim follows by \cite[Lemma A.15]{2022arXiv220709929B} and the fact that $\ttCat_{\bb{Q}}\simeq \ttCat_{\cat{D}^b(\bb{Q})/}$.
\end{proof}

\begin{definition}\label{def:consspec} 
For any rigid 2-ring $\cat{C}$, define its \textit{constructible spectrum} to be 
    \[
        \Speccons(\cat{C})\coloneqq\Speccons_{\ttCat}(\cat{C}).
    \]
\end{definition}
Before we proceed, we record the following helper lemma.
\begin{lemma}\label{lem:adjunctionalg}  Let $\cat{C}$ be a rigid 2-ring.  Then there is an induced adjunction
\begin{center}
\begin{tikzcd}
         \mo{Perf}_{\cat{C}}(-)\colon\mo{CAlg}\left(\mo{Ind}\left(\mathcal{C}\right)\right)\arrow[r, shift left=.75ex] & \left(\ttCat\right)_{\mathcal{C}/}\arrow[l, shift left=.75ex]\colon\mo{End}_{-}(\bb{1}),
\end{tikzcd}
\end{center}
with fully faithful left adjoint and where the right adjoint commutes with filtered colimits and finite coproducts.
\end{lemma}
\begin{proof}
Since $\cat{C}$ is rigid, so too are the categories $\mo{Perf}_{\cat{C}}(R):=\Mod_R(\Ind(\cat{C}))^\omega$ for commutative algebras $R$ in $\Ind(\cat{C})$, and the claim about the induced adjunction follows from \cite[Corollary 4.8.5.21]{HA}.  Since filtered colimits commute with taking endomorphisms of the unit, the right adjoint commutes with filtered colimits.

We reduce to checking that for rigid 2-rings $\cat{D}$ and $\cat{E}$ under $\cat{C}$, \[\mo{End}_{\cat{D}\otimes_{\cat{C}}\cat{E}}(\bb{1})\simeq \mo{End}_{\cat{D}}(\bb{1})\otimes \mo{End}_{\cat{E}}(\bb{1}),\]
considered as algebras in $\mo{Ind}(\cat{C})$.  Writing out the pushout diagram and giving names to the relevant functors,
\begin{equation}\label{eq:square}
\begin{tikzcd}
\mo{Ind}(\cat{C})\rar{f^*}\dar{g^*} & \mo{Ind}(\cat{D})\dar{h^*}\\
\mo{Ind}(\cat{E})\rar{k^*} & \mo{Ind}(\cat{D})\otimes_{\mo{Ind}(\cat{C})}\mo{Ind}(\cat{E}),
\end{tikzcd}
\end{equation}
we note that we can rewrite the algebra $\mo{End}_{\cat{D}\otimes_{\cat{C}}\cat{E}}(\bb{1})$ in $\mo{Ind}(\cat{C})$ using the right adjoints, \[\mo{End}_{\cat{D}\otimes_{\cat{C}}\cat{E}}(\bb{1})\simeq f_*h_*(\bb{1}_{\cat{D}\otimes_{\cat{C}}\cat{E}})\simeq g_*k_*(\bb{1}_{\cat{D}\otimes_{\cat{C}}\cat{E}}).\]
As in \Cref{prop:nilconsbasechange}, we can use \cite[Theorem 4.6]{haugsengmonad} to show that the square \ref{eq:square} is vertically right adjointable, which tells us that there is an equivalence $f^*g_*\simeq h_*k^*$. Applying this to the unit, we learn that
\[
f_*h_*(\bb{1}_{\cat{D}\otimes_{\cat{C}}\cat{E}})\simeq f_*h_*k^*(\bb{1}_{\cat{E}})\simeq f_*f^*g_*(\bb{1}_{\cat{E}}).
\]
By the projection formula, we learn that
\[f_*f^*g_*(\bb{1}_{\cat{E}})\simeq (f_*f^*(\bb{1}_{\cat{C}}))\otimes g_*(\bb{1}_{\cat{E}})\simeq f_*(\bb{1}_{\cat{D}})\otimes g_*(\bb{1}_{\cat{E}}).\]
Finally, using again the identification of $\mo{End}_{\cat{D}}(\bb{1})$ with the image $f_*(\bb{1}_{\cat{D}})$ of the unit under the right adjoint to $f^*$, and similarly for $\cat{E}$, this yields the claim.
\end{proof}

With this in hand, we can make the following observation, relating the constructible spectrum of a category $\cat C$ back to the constructible spectrum of its unit in the category of commutative $\cat{C}$-algebras.

\begin{proposition}\label{prop:consequalscons}  Let $\cat{C}\in\ttCat$ be a rigid 2-ring.  Then there is an equivalence of spaces \[\Speccons(\cat{C})\simeq \Speccons_{\mo{CAlg}\left(\mo{Ind}\left(\cat{C}\right)\right)}(\bb{1}_{\cat{C}}).\]
\end{proposition}
\begin{proof}
Using the induced adjunction from \Cref{lem:adjunctionalg},
\begin{center}
\begin{tikzcd}
         \mo{Perf}_{\cat{C}}(-)\colon\mo{CAlg}\left(\mo{Ind}\left(\mathcal{C}\right)\right)\arrow[r, shift left=.75ex] & \left(\ttCat\right)_{\mathcal{C}/}\arrow[l, shift left=.75ex]\colon\mo{End}_{-}(\bb{1}),  
\end{tikzcd}
\end{center}
and the fact that the right adjoint commutes with filtered colimits, the left adjoint preserves compact objects.

We claim that the right adjoint takes Nullstellensatzian rigid 2-$\cat{C}$-algebras to Nullstellensatzian commutative algebras in $\mo{CAlg}\left(\mo{Ind}\left(\mathcal{C}\right)\right)$.  Indeed, given any Nullstellensatzian rigid 2-ring $\cat{D}$ with a map $\cat{C}\to \cat{D}$, consider a nonzero algebra $R$ compact over $\mo{End}_{\cat{D}}(\bb{1})$.  Using the induced pushout square, 
\begin{center}
\begin{tikzcd}
\mo{Perf}_{\cat{C}}(\mo{End}_{\cat{D}}(\bb{1}))\rar\dar &\cat{D}\dar\\
\mo{Perf}_{\cat{C}}(R)\rar & \mo{Perf}_{\cat{D}}(R),
\end{tikzcd}
\end{center}
we find that $\mo{Perf}_{\cat{D}}(R)$ is compact over $\cat{D}$.  Since the top horizontal map is fully faithful, \Cref{lem:FFBasechange} implies the bottom map is too, and in particular $\mo{Perf}_{\cat{D}}(R)$ is nonzero.  Since $\cat{D}$ is Nullstellensatzian, $\cat{D}\to \mo{Perf}_{\cat{D}}(R)$ must split, and applying $\mo{End}_{-}(\bb{1})$ yields a splitting of the map $\mo{End}_{\cat{D}}(\bb{1})\to R$, as desired.

Given any two maps from $\cat{C}$ to Nullstellensatzian rigid 2-rings $\cat{D}$ and $\cat{E}$, these represent the same point in $\Speccons(\cat{C})$ if and only if there's a common refinement to a map to a Nullstellensatzian rigid 2-ring $\cat{F}$.  Applying $\mo{End}_{-}(\bb{1})$ to such a refinement shows that the Nullstellensatzian $\cat{C}$-algebras obtained from these categories determine the same point in $\Speccons_{\mo{CAlg}\left(\mo{Ind}\left(\cat{C}\right)\right)}(\bb{1})$, giving a well-defined map 
\[\Speccons(\cat{C})\to \Speccons_{\mo{CAlg}\left(\mo{Ind}\left(\cat{C}\right)\right)}(\bb{1}),\]
which we claim to be an isomorphism.

First, consider a point in $\Speccons_{\mo{CAlg}\left(\mo{Ind}\left(\mathcal{C}\right)\right)}(\bb{1}_{\cat{C}})$ represented by a map $\bb{1}_{\cat{C}}\to S$ for some Nullstellensatzian $\cat{C}$-algebra $S$.  The category $\mo{Perf}(S)$ is nonzero, so admits a map $\mo{Perf}(S)\to \cat{S}^{\prime}$ for some Nullstellensatzian rigid 2-ring $\cat{S}^{\prime}$ by \cite[Proposition A.31]{2022arXiv220709929B}.  This gives surjectivity of the map.

Next, suppose that we have Nullstellensatzian rigid 2-rings $\cat{S}$ and $\cat{S}^{\prime}$ under $\cat{C}$ such that $\cat{S}\otimes_{\cat{C}}\cat{S}^{\prime}\simeq 0$.  By \Cref{lem:adjunctionalg}, $\mo{End}_{-}(\bb{1})$ commutes with finite coproducts, so that \[\End_{\cat{S}}(\bb{1})\otimes \End_{\cat{S}^{\prime}}(\bb{1})\simeq \mo{End}_{\cat{S}\otimes_{\cat{C}}\cat{S}^{\prime}}(\bb{1})\simeq 0\] in $\mo{CAlg}\left(\mo{Ind}\left(\mathcal{C}\right)\right)$, giving injectivity.
\end{proof}
To relate the constructible spectrum back to higher Zariski geometry, we make the following simple observation.
\begin{lemma}\label{lem:NSatzOnePtBalmer}
If $\cat L$ is a Nullstellensatzian rigid 2-ring, then it has no non-zero proper tt-ideals, and in particular the Balmer spectrum $\Spc(\cat L)\simeq *$ is a single point.
\end{lemma}
\begin{proof}
If $\cat L$ had a proper non-zero tt-ideal $\mc{I}$, we could take some non-zero $x\in \mc{I}$, and note that $\cat L/\langle x\rangle$ is a compact non-zero algebra under $\cat L$, so must split.  Any splitting of $\cat L\to \cat L/\langle x\rangle$ shows that $x\simeq 0$ in $\cat L$, contradicting the choice of $x$.
\end{proof}

\begin{corollary}\label{cor:conscomparisonmap}
For any rigid 2-ring $\cat C$, there is a canonical map 
\[\Speccons(\cat C)\to \Spc(\cat C)\]
natural in $\cat C$.
\end{corollary}

\begin{proof}
This follows from \Cref{lem:NSatzOnePtBalmer}, where the map is defined explicitly by sending a point in the constructible spectrum, represented by a map $\cat C\to \cat L$ for some Nullstellensatzian rigid 2-ring $\cat L$, to the image of the map $\Spc(\cat L)\to \Spc(\cat C)$.
\end{proof}
We can give a precise meaning to the image of the comparison map:
\begin{proposition}\label{prop:ConsDetectFieldPts}  Let $\cat C$ be a rigid 2-ring, and consider any point $\mc{P}\in\Spc(\cat C)$.  Then, the following are equivalent.
\begin{enumerate}
\item  There exists a rigid 2-ring $\cat D$ with $\Spc(\cat D)\simeq *$, together with a map $\cat C\to \cat D$ of rigid 2-rings, such that the image of \[\Spc(\cat D)\to \Spc(\cat C)\] is exactly $\mc{P}$.
\item  The point $\mc{P}$ is in the image of the map \[\Speccons(\cat C)\to \Spc(\cat C).\]
\end{enumerate}
\end{proposition}
\begin{proof}
(a)$\implies$(b)  Consider a rigid 2-ring $\cat D$ under $\cat C$ as in the statement.  Since the Balmer spectrum of $\cat D$ is non-empty, $\cat D$ is in particular nonzero, and so there exists a map $\cat D \to \cat L$ to some Nullstellensatzian rigid 2-ring $\cat L$.  The point of $\Speccons(\cat C)$ represented by the rigid 2-ring map $\cat C\to \cat L$ maps to the point $\mc{P}$.

(b)$\implies$(a)  Consider any point in the constructible spectrum mapping to $\mc{P}$, and take a Nullstellensatzian representative $\cat C\to \cat L$ for said point.  Since $\Spc(\cat L)\simeq *$, and the point represented by $\cat L$ maps to $\mc{P}$, we can take $\cat D = \cat L$ in the statement.
\end{proof}
 The following is essentially \cite[Example A.72]{2022arXiv220709929B}:
\begin{corollary}\label{cor:nonsurj}
Let $R$ be a height $n<\infty$ $\bb{E}_{\infty}$-ring spectrum.  For any point $\mc{P}\in\Spc(\mo{Perf}(R))$ such that $\mc{P}$ has height $n>0$, there is no map of commutative $2$-rings $\mo{Perf}(R)\to \cat C$ to a $2$-ring $\cat C$ with $\Spc(\cat C)\simeq *$ which picks out the point $\mc{P}$.
\end{corollary}
\begin{proof}
It follows by \cite{hahnsupport} (see also \cite[Theorem 1.5]{2022arXiv220709929B}) that if $R$ is an $L_n^f$-local Nullstellensatzian $\bb{E}_{\infty}$-ring, then $R$ is rational, so has height $n=0$.
\end{proof}
The height $n>0$ phenomenon is nothing new, as there already are no $\bb{E}_{\infty}$ maps out of spectra picking out the points of $\Spc(\mo{Perf}(\mathbb S))$ corresponding to the Morava $K$-theories $K(n), n\in [1,\infty)$. It turns out that the map $\Speccons(\cat C)\to \Spc(\cat C)$ can fail to be surjective in positive characteristic, too, as the following example shows.

\begin{example}\label{ex:nosurj}
Consider the free $\bb{E}_{\infty}$-$\bb{F}_2$-algebra $R=\bb{F}_2\{x\}_{\infty}$ on a class $x$ in degree $0$.  It is known (cf.~\cite[Example~5.10]{lawson2020enringspectradyerlashof}) that $\pi_*(R)$ is an infinite dimensional polynomial ring on sequences of Dyer--Lashof operations $Q_I(x)$ applied to $x$.  Consider the localization $S\coloneqq R_{(x)}$ to the prime ideal $(x)$ in $\pi_*(R)$.  One can check that $\mo{Perf}(R)$ has two points in its Balmer spectrum by noting that the weak rings $R/x^2$ and $R[1/x]$ are jointly nil-conservative, so the homological spectrum consists of at least two points\footnote{In fact, using \Cref{thm:nconsequalshom}, one can show that these are the only two points in the homological spectrum.}, which correspond respectively to the zero ideal and the ideal generated by $R/x$ (that $R/x^2$ is a weak ring follows from \cite[Lemma 5.4, Remark 5.5]{burklund2022multiplicativestructuresmoorespectra}).

We claim that there is no symmetric monoidal functor $\mo{Perf}(S)\to \cat D$ to a rigid $2$-ring $\cat D$ with Balmer spectrum a point which picks out the point $\mc{P}$ corresponding to $R/x$.  By \Cref{prop:ConsDetectFieldPts}, it suffices to check that there is no map $S\to L$ to a Nullstellensatzian $\bb{E}_{\infty}$-$\bb{F}_2$-algebra $L$ representing a point in $\Speccons(S)$ mapping to $\mc{P}$.  Nullstellensatzian $\bb{E}_{\infty}$-$\bb{F}_2$-algebras have been studied by Riedel \cite[Proposition 4.16]{riedel2025reducedpointsmathbbeinftyringspositive}, where it was shown that if $L$ is Nullstellensatzian, then it is 1-periodic with $\pi_0(L)$ a field.  In particular, if we had an $\bb{E}_{\infty}$-ring map $S\to L$ picking out the closed point of $\pi_*(S)$, then the image of $x$ must vanish in $\pi_*(L)$.  However, we are considering $\bb{E}_{\infty}$-ring maps and they are compatible with power operations. Thus we would have to have $Q_1(x)\mapsto 0$ as well -- by construction, $Q_1(x)$ is invertible in $S$, so this forces $L=0$, a contradiction.
\end{example}

\newpage
\section{Constructible Spectra in the Rational Case}\label{sec:rational}
\subsection{Constructible spectra and homological spectra}
One cannot expect in general that the comparison map 
\[
    \Speccons(\cat C)\to \Spc(\cat C)
\] to be surjective for a given rigid 2-ring $\cat C$, see \Cref{cor:nonsurj} and \Cref{ex:nosurj}.  Both of the examples there are not rational, and in fact the rational case is nice enough that the comparison map \textit{is} actually always surjective. In fact we can say more, namely we have the following, which is the main theorem of this section:
\begin{theorem}\label{thm:consequalshom}  Let $\cat{C}$ be a rational rigid 2-ring.  Then there is a natural isomorphism of sets $\Speccons(\cat{C})\simeq \Spc^{h}(\cat{C})$  between the constructible spectrum and the homological spectrum of $\cat{C}$.
\end{theorem}
\begin{remark}
We hope that the above result motivates the study of Nullstellensatzian (rational) rigid 2-rings for tt-geometers. Forthcoming work of the third author partially initiates the study of their structure, and in particular of their $K$-theory.
\end{remark}
The following formula for free commutative algebras on pointed objects in the rational case will be crucial to proving \Cref{thm:consequalshom}.
\begin{lemma}[{\cite[Corollary B]{2025arXiv251113536K}}; {\cite[Proposition 4.6]{betts2025motivicweilheightmachine}}]\label{cor:ShaiLior}  Let $\cat{T}$ be a rational  stable presentably symmetric monoidal $\infty$-category, and consider an $\bb{E}_0$-algebra (that is, a pointed object) $\bb{1}\to X$ in $\cat{T}$.  Then the free $\bb{E}_{\infty}$-algebra on this $\bb{E}_0$-algebra is given by the filtered colimit
\[\FrEinfz(\bb{1}\to X)=\varinjlim_n\left((X^{\otimes n})_{h\Sigma_n}\right).\]
\end{lemma}
This in turn, allows us to make the following observation.
\begin{proposition}\label{prop:CAlgDetection}  Let $\cat{T}$ be a rational presentably symmetric monoidal stable $\infty$-category such that the unit $\bb{1}_{\cat{T}}$ is compact (e.g., if $\cat{T}=\mo{Ind}\left(\cat{C}\right)$ is the big category attached to a rational rigid 2-ring $\cat{C}$).  A map $f\colon\bb{1}\to X$ in $\cat{T}$ is $\otimes$-nilpotent if and only if the free $\bb{E}_{\infty}$-algebra on $f$ vanishes, which is to say, if and only if \[\FrEinfz(\bb{1}\to X)\simeq 0\colon\bb{1}\to 0.\]  In particular, any pointed object $f\colon\bb{1}\to X$ for which the pointing map is not $\otimes$-nilpotent admits a nontrivial $\bb{E}_0$-algebra map to the underlying $\bb{E}_0$-algebra of an $\bb{E}_{\infty}$-algebra.
\end{proposition}
\begin{proof}
Clearly, if $f$ is $\otimes$-nilpotent, \[\FrEinfz(\bb{1}\to X)\simeq 0\colon\bb{1}\to 0,\] since the multiplication on any $\bb{E}_{\infty}$-algebra $R$ splits the map $\bb{1}\to X$ over $R$, and if this map were $\otimes$-nilpotent, then taking a high enough $\otimes$-power for it to be zero, we can factor $\bb{1}\to R$ as $\bb{1}\to 0\to R$, so $R\simeq 0$.

Conversely, suppose that \[\FrEinfz(f\colon\bb{1}\to X)\simeq 0\colon\bb{1}\to 0.\]  By \Cref{cor:ShaiLior}, we can identify $\FrEinfz(f\colon\bb{1}\to X)$ with the filtered colimit of $f^{\otimes n}_{h\Sigma_n}$ to rewrite this as
\[\varinjlim\left(f^{\otimes n}_{h\Sigma_n}\colon\bb{1}\to X^{\otimes n}_{h\Sigma_n}\right)\simeq 0\colon\bb{1}\to 0.\]  Since the unit object was assumed to be compact, this implies there exists some $n$ such that $f^{\otimes n}_{h\Sigma_n}\simeq 0$, and we claim also that $f^{\otimes n}\simeq 0$.  Note that there is a natural commutative diagram
\begin{center}
\begin{tikzcd}
\bb{1}^{h\Sigma_n}\rar{(f^{\otimes n})^{h\Sigma_n}}\dar & (X^{\otimes n})^{h\Sigma_n}\dar\\
\bb{1}\rar{f^{\otimes n}}\dar & X^{\otimes n}\dar\\
\bb{1}_{h\Sigma_n}\rar{f^{\otimes n}_{h\Sigma_n}} & (X^{\otimes n})_{h\Sigma_n}.
\end{tikzcd}
\end{center}
Since we are rational, the rightmost vertical composite from the homotopy fixed points to the homotopy orbits of the action of the finite group $\Sigma_n$ is an equivalence, with inverse given by a multiple of the norm map $\mo{Nm}$ (specifically $\frac{1}{n!}\mo{Nm}$).  Similarly, since the action of $\Sigma_n$ on $\bb{1}^{\otimes n}\simeq \bb{1}$ is trivial, both leftmost vertical maps are equivalences, where in fact the first leftmost vertical map can be identified with the identity map, and the second one with multiplication by $n!$.  In particular, we have that $f^{\otimes n}$ factors through $(f^{\otimes n})^{h\Sigma_n}$ which is equivalent to $(f^{\otimes n})_{h\Sigma_n}$, which in turn is nullhomotopic, so that $f^{\otimes n}$ is nullhomotopic. 
\end{proof}
The next step is to finally relate the homological spectrum to the constructible spectrum by means of computing the homological spectrum of Nullstellensatzian objects.
\begin{lemma}\label{lem:NSianttField}
Let $\cat{C}$ be a rational rigid 2-ring, and let $A\in\mo{CAlg}\left(\mo{Ind}\left(\mathcal{C}\right)\right)$ be a Nullstellensatzian commutative algebra object.  Then the homological spectrum of the category of perfect $A$-modules in $\cat{C}$, $\Spc^{h}(\mo{Perf}_{\cat{C}}(A))$, consists of a single point.
\end{lemma}
\begin{proof}

Consider two points 

$\mf{m}_1$ and $\mf{m}_2$ in $\Spc^h(\mo{Perf}_{\cat{C}}(A))$.  Taking the corresponding ``residue fields'' from \Cref{rem:weakringprimes} we find corresponding weak rings $E_{\mf{m}_i}$ over $A$.

Since the pointed object in $A\text{-}\mo{Mod}(\mo{Ind}(\cat{C}))$ given by $A\to E_{\mf{m}_i}$ is a weak ring, the pointing map cannot be $\otimes$-nilpotent.  Applying \Cref{prop:CAlgDetection}, we can find nonzero $\bb{E}_{\infty}$-$A$-algebras $L_i$ in $\mo{Ind}(\cat{C})$ together with maps of pointed objects \[A\to E_{\mf{m}_i}\to L_i,\]

Since $L_i$ are nonzero as $A$-algebras, they are also nonzero in $\mo{CAlg}\left(\mo{Ind}\left(\mathcal{C}\right)\right)$, so there exists Nullstellensatzian objects $S_i$ in $\mo{CAlg}\left(\mo{Ind}\left(\mathcal{C}\right)\right)$ with maps $L_i\to S_i$.

Since $A$ was chosen to be Nullstellensatzian, its constructible spectrum consists of a single point by \cite[Proposition A.31]{2022arXiv220709929B}, and so we find that $S_1\otimes_A S_2\neq 0$. This implies in turn that $L_1\otimes_A L_2\neq 0$, and thus that $E_{\mf{m}_1}\otimes_{A} E_{\mf{m}_2}\neq 0$.  Applying \Cref{prop:weakRingtenszero}, this in turn implies that $\mf{m}_1= \mf{m}_2$. Since they were arbitrary points, this proves the claim. 
\end{proof}
\begin{corollary}\label{cor:NSianFieldIsh}
Let $\cat{C}$ be a rational rigid 2-ring. For any Nullstellensatzian algebra $A\in \mo{CAlg}\left(\mo{Ind}\left(\cat{C}\right)\right)$, every non-zero object of the category $\mo{Perf}_{\cat{C}}(A)$ is $\otimes$-faithful; equivalently, its coevaluation map splits.
\end{corollary}
\begin{proof}
Note first that given any non-zero $X\in \mo{Perf}_{\cat{C}}(A)$, the morphism \[\mo{coev}_X\colon A \to X\otimes X^{\vee}\] is split after tensoring with $X$, so cannot be $\otimes$-nilpotent.  Using \Cref{prop:CAlgDetection}, the free $\bb{E}_{\infty}$-$A$-algebra on $\mo{coev}_X$ is a nonzero compact $A$-algebra, so admits a section back to $A$ since $A$ is Nullstellensatzian.  For a given choice of section, the composite \[A\to X\otimes X^{\vee}\to \mo{Free}_{\bb{E}_{\infty}/\bb{E}_0}(\mo{coev}_X)\to A\] splits coevaluation of $X$, and so $X$ must be $\otimes$-faithful.
\end{proof}
\begin{corollary}\label{Cor:NSianVeryFieldIsh}  
Let $\cat{C}$ be a rational rigid 2-ring. For any Nullstellensatzian algebra $A\in \mo{CAlg}\left(\mo{Ind}\left(\cat{C}\right)\right)$, every rigid 2-ring map \[F\colon\mo{Perf}_{\cat{C}}(A)\to \cat D\] with $\cat D\neq 0$ is conservative for maps between compact objects, meaning that any non-zero map $f\colon Y\to X$ between compact objects in $\mo{Perf}_{\cat{C}}(A)$ has $F(f)\neq 0$ in $\cat D$.
\end{corollary}
\begin{proof}
Without loss of generality, $A=\bb{1}$ (and $\bb{1}$ is Nullstellensatzian).

Using rigidity, up to replacing $f$ by its mate $Y\otimes X^{\vee} \to \bb{1}$, we may assume without loss of generality that $f$ has target $\bb{1}$. If $f$ is not zero, then the map \[r\coloneqq\mo{cofib}(f)\colon \bb{1}\to Z\]
is not split injective.  Since $\FrEinfz(r)$ is a compact commutative algebra, and $\bb{1}$ is Nullstellensatzian, it follows that $\FrEinfz(r)\simeq 0$.  Applying \Cref{cor:ShaiLior}, this implies that $r$ is $\otimes$-nilpotent, and hence so is $F(r)$. Since $\cat D$ is nonzero and $F(r) $ has source $\bb{1}_\cat D$, it follows that $F(r)$ is not split injective, and hence $F(f)$ is non zero, as was to be shown.
\end{proof}
\begin{remark}
Although it seems reasonable to expect, forthcoming work of the third author will show that Nullstellensatzian rigid 2-rings are \textit{not} tt-fields in the sense of \cite{BKS2019}! In fact, they are not even pure-semisimple in the sense of \cite[Theorem 5.7]{BKS2019} and drastically fail condition (iii) in \textit{loc.~cit.}

\end{remark}

Without further ado, we can prove the main theorem of the section.
\begin{proof}[Proof of \Cref{thm:consequalshom}]  By \Cref{prop:consequalscons}, instead of working with the constructible spectrum of $\cat{C}$, $\Speccons(\cat{C})$, we can work equally well with the constructible spectrum of $\bb{1}_{\cat{C}}$ defined internally to commutative algebras in the big category $\mo{Ind}(\cat{C})$,\[\Speccons_{\mo{CAlg}\left(\mo{Ind}\left(\mathcal{C}\right)\right)}(\bb{1}_{\cat{C}}).\]

Any Nullstellensatzian algebra $S$ in $\mo{CAlg}\left(\mo{Ind}\left(\mathcal{C}\right)\right)$ determines a functor \[\cat{C}\xrightarrow{-\otimes S}\mo{Perf}_{\cat{C}}(S).\]  By \Cref{lem:NSianttField}, we have that $\Spc^{h}(\mo{Perf}_{\cat{C}}(S))$ is a singleton, so we can define a map 
\[
\Speccons_{\mo{CAlg}\left(\mo{Ind}\left(\mathcal{C}\right)\right)}(\bb{1}_{\cat{C}})\to \Spc^{h}(\cat{C}),
\] by taking a Nullstellensatzian commutative algebra $S$ representing a point in the constructible spectrum to the image of the induced map \[*\simeq \Spc^{h}(\mo{Perf}_{\cat{C}}(S))\to \Spc^{h}(\cat{C}).\]  If $S_1$ and $S_2$ are two Nullstellensatzian commutative algebras representing the same point in the constructible spectrum, then again by \cite[Lemma A.35]{2022arXiv220709929B}, there exists another Nullstellensatzian commutative algebra $S_3$ with maps $S_1\to S_3\gets S_2$.  Functoriality of $\Spc^{h}(-)$ shows that point of the homological spectrum attached to $S_1$ agrees with the one attached to $S_2$, and so our total comparison map is indeed well-defined.

To see that the map is surjective, pick any homological prime $\mf{m}\in\Spc^{h}(\cat{C})$, and consider once again the residue field $E_{\mf{m}}$ at $\mf{m}$.  Applying \Cref{prop:CAlgDetection}, there exists a nonzero $\bb{E}_{\infty}$-algebra $L$ with a map of pointed objects $E_{\mf{m}}\to L$.  Choose some Nullstellensatzian $\bb{E}_{\infty}$-algebra $S$ in $\mo{CAlg}\left(\mo{Ind}\left(\mathcal{C}\right)\right)$ with a map $L\to S$.  Then, for all homological primes $\mf{p}\neq \mf{m}\in \Spec^{h}(\cat{C})$, since we have that $E_{\mf{p}}\otimes E_{\mf{m}}\simeq 0$ by \cite[Proposition 5.3]{balmernilpotence}, it follows that $S\otimes E_{\mf{p}}\simeq 0$ for all primes $\mf{p}\neq \mf{m}$.  Since the kernel of $S\otimes -$ contains $E_{\mf{p}}$ for all $\mf{p}\neq \mf{m}$, the only possible prime that the map $\Spc^{h}(\mo{Perf}_{\cat{C}}(S))\to \Spc^{h}(\cat{C})$ can possibly hit is $\mf{m}$ itself.  Since $\mf{m}$ was arbitrary, the map is surjective.

Finally, we must show injectivity of the comparison map.  Choose an arbitrary set-theoretic section \[\Spc^h(\cat{C})\to \Speccons_{\mo{CAlg}\left(\mo{Ind}\left(\mathcal{C}\right)\right)}(\bb{1}_{\cat{C}}),\] and lift this arbitrarily to a collection $S_{\mf{m}}$ of Nullstellensatzian commutative algebras in $\mo{CAlg}\left(\mo{Ind}\left(\mathcal{C}\right)\right)$ with $S_{\mf{m}}$ representing the point above $\mf{m}$ in the chosen section.  If the map is not injective, then there exists some point not in the image of this section, which corresponds to a Nullstellensatzian commutative algebra $S$ such that $S\otimes S_{\mf{m}}\simeq 0$ for all $\mf{m}\in\Spc^h(\cat{C})$.  Now, \cite[Theorem 1.9]{BCHS2024} tells us that the family $\{\cat{C}\to \mo{Perf}_{\cat{C}}(S_{\mf{m}})\}_{\mf{m}\in\Spc^h(\cat{C})}$ is jointly nil-conservative.  By \Cref{prop:nilconsbasechange}, the basechange 
\[\{\mo{Perf}_{\cat{C}}(S)\to \mo{Perf}_{\cat{C}}(S)\otimes_{\cat{C}}\mo{Perf}_{\cat{C}}(S_{\mf{m}})\}_{\mf{m}}\]
is also jointly nil-conservative.  However, we have that
\[\mo{Perf}_{\cat{C}}(S)\otimes_{\cat{C}}\mo{Perf}_{\cat{C}}(S_{\mf{m}})\simeq \mo{Perf}_{\cat{C}}(S\otimes S_{\mf{m}})\simeq 0,\]
for all $\mf{m}\in\Spc^h(\cat{C})$.  Since the map from $\mo{Perf}_{\cat{C}}(S)\to 0$ is nil-conservative, we find that $S$ itself must be zero, contradicting the fact that $S$ is Nullstellensatzian.
\end{proof}

\subsection{Consequences for residue fields} Using \Cref{thm:consequalshom}, we can provide criteria for certain rational rigid 2-rings to admit enough tt-fields.
\begin{corollary}\label{cor:EnoughTTFields}
Let $\cat{C}$ be a rational rigid 2-ring such that for all Nullstellensatzian commutative $\cat{C}$-algebras $L\in \mo{CAlg}\left(\mo{Ind}\left(\cat{C}\right)\right)$, $\mo{Perf}_{\cat{C}}(L)$ is a tt-field in the sense of \cite{BKS2019}.  Then for any commutative algebra $R\in \mo{CAlg}\left(\mo{Ind}\left(\cat{C}\right)\right)$, the category $\mo{Perf}_{\cat{C}}(R)$ has enough tt-fields. 
\end{corollary}
\begin{proof}
For any point in the homological spectrum of $\mo{Perf}_{\cat{C}}(R)$, choose a map $R\to L$ to a Nullstellensatzian commutative $\cat{C}$-algebra $L$ representing that point, which exists by \Cref{thm:consequalshom}.  The induced functor
\[\mo{Perf}_{\cat{C}}(R)\to \mo{Perf}_{\cat{C}}(L)\]
acts as a tt-field at the given point.
\end{proof}
\begin{corollary}\label{cor:NaiveEqualsGenuineEnoughTT}
Suppose that $\cat{C}$ is a rational rigid 2-ring such that all Nullstellensatzian commutative algebras in $\mo{Ind}(\cat{C})$ have module categories which are tt-fields.  Then for all commutative algebras $R\in \mo{CAlg}\left(\mo{Ind}\left(\cat{C}\right)\right)$, and for any object $X\in\mo{Ind}(\left(\mo{Perf}_{\cat{C}}(R)\right)$, the genuine and naive homological support of $X$ agree, 
\[\supph(X)=\Supp^{n}(X).\]
\end{corollary}
\begin{proof}
This follows by \Cref{cor:EnoughTTFields} and \cite[Proposition~1.16]{BHSZ2024pp}.
\end{proof}
\begin{corollary}\label{cor:Qttresidues}
If $R$ is a rational $\bb{E}_{\infty}$-ring, then $\mo{Perf}(R)$ has enough tt-fields, and points in the homological spectrum of $\mo{Perf}(R)$ are all witnessed by maps from $R$ into rational 2-periodic fields.  In particular, for any module over a rational $\bb{E}_{\infty}$-ring $R$, its naive homological support agrees with its genuine homological support.
\end{corollary}
\begin{proof}
This follows from \Cref{cor:EnoughTTFields}, using the description (\cite[Theorem~A]{2022arXiv220709929B}) of Nullstellensatzian rational $\bb{E}_{\infty}$-rings as even 2-periodic fields with $\pi_0$ algebraically closed.  The final claim follows by \Cref{cor:NaiveEqualsGenuineEnoughTT}.
\end{proof}

\subsection{The c-topology}  In the context of the general theory of constructible spectra, a certain canonical topology was introduced in \cite{2022arXiv220709929B}, which we now recall.
\begin{definition}[{\cite[Definition A.44]{2022arXiv220709929B}}]  A map $\cat{C}\to\cat{D}$ of rational rigid 2-rings will be said to be a \textit{c-cover} if it induces a surjection on constructible spectra, and the topology generated by c-covers is called the \textit{c-topology}.
\end{definition}
\begin{remark}
The \textit{c-topology} has another more classical name, appearing in \cite[pg. 115, (e)]{maclane2012sheaves} under the name of the \textit{dense topology} (or later as the \textit{double negation topology}).
\end{remark}
\begin{remark}
Combining \Cref{thm:consequalshom} and \Cref{thm:surjspech} we find that c-covers of rational rigid 2-rings are exactly nil-conservative functors.
\end{remark}
In our specific case, we are able to use \Cref{thm:consequalshom} to deduce some very nice properties about covers for the c-topology, at least locally.
\begin{proposition}\label{Prop:NicecCov}  Let $\cat C$ be a rational rigid 2-ring, $I$ a set, and $L_i$ an $I$-indexed collection of Nullstellensatzian algebras $L_i\in\mo{CAlg}\left(\mo{Ind}\left(\cat{C}\right)\right).$  Any c-cover $\prod_{I}L_i\to \prod_J L^{\prime}_j$ by a product of Nullstellensatzians refines to a c-cover \[\prod_{I}L_i\to \prod_{J}  L^{\prime}_j\to \prod_{I} L^{\prime\prime}_i,\] with the composite induced by maps of Nullstellensatzian algebras $L_i\to L^{\prime\prime}_i$ indexed by the original set $I$.
\end{proposition}
\begin{proof}
First, note that since $\prod_I L_i\to \prod_J L^{\prime}_j$ is a c-cover, it is nil-conservative, so that \[L_i\otimes_{\prod_{I} L_i}\prod_J L^{\prime}_j\neq 0.\]  Define $L_i^{\prime\prime}$ as some Nullstellensatzian commutative algebra in $\cat C$ under this tensor product $L_i\otimes_{\prod_{I} L_i}\prod_J L^{\prime}_j$.

We find that these maps $L_i\to L_i^{\prime\prime}$ assemble to give a factorization 
\[\prod_I L_i\to \prod_J L_j^{\prime} \to \prod_{I}L_i^{\prime\prime},\]
and we are left to show that the composite is a c-cover.

For this, recall from \Cref{Cor:NSianVeryFieldIsh} that any functor out of a Nullstellensatzian rational rigid 2-ring is conservative on morphisms between compacts, and in particular this is true for basechange along $L_i\to L_i^{\prime\prime}$.  As this property is stable under taking products of rigid 2-rings (which the categories of perfect modules over the products of algebras embed fully faithfully into), basechange along 
\[\prod_I L_i\to \prod_{I} L_i^{\prime\prime}\]
is conservative on morphisms between compact objects.  Finally, using that a 2-ring map which is conservative on morphisms between compacts is in particular nil-conservative (e.g. by writing any weak ring as a filtered colimit of compact $\bb{E}_0$-algebras and using that $\bb{1}\to 0$ is compact), we find that $\prod_I L_i \to \prod_{I} L_i^{\prime\prime}$ is a c-cover, as desired.
\end{proof}
Before deducing more consequences of \Cref{thm:consequalshom}, we make a small digression to compute the Balmer spectrum of a product of Nullstellensatzian objects.
\begin{lemma}\label{lem:NSianProductSpc}  Let $\cat{C}$ be a rational rigid 2-ring, $I$ a set, and $L_i$ an $I$-indexed collection of Nullstellensatzian algebras $L_i\in\mo{CAlg}\left(\mo{Ind}\left(\cat{C}\right)\right).$  Then the Balmer spectrum $\Spc(\mo{Perf}_{\cat{C}}(\prod_{I}L_i))$ is isomorphic to the Stone--\v{C}ech compactification $\beta I$ of the set $I$.
\end{lemma}
\begin{proof}
Note that there is a fully faithful inclusion \[\mo{Perf}_{\cat{C}}(\prod_{I}L_i)\hookrightarrow \prod_{I}\mo{Perf}_{\cat{C}}(L_i),\]
which is therefore nil-conservative and hence induces a surjective map on Balmer spectra by \cite[Theorem 1.4]{BCHS2024}.  By naturality of the comparison map from \Cref{prop:comparisonUnit}, and the fact that the endomorphism ring of the unit in the two categories agree, it suffices to show that \[\beta I\cong \Spc(\prod_{I}\mo{Perf}_{\cat{C}}(L_i))\cong \Spech(\pi_*(\mo{End}_{\prod_{I}\mo{Perf}_{\cat{C}}(L_i)}(\bb{1}))).\]
Since each $L_i$ is a Nullstellensatzian commutative algebra in $\cat{C}$, its endomorphism ring is also Nullstellensatzian, and in particular by \cite[Theorem 6.3]{2022arXiv220709929B} is an even 2-periodic algebraically closed field.  Since taking endomorphism rings of the unit commutes with products, we have that \[\mo{End}_{\prod_{I}\mo{Perf}_{\cat{C}}(L_i)}(\bb{1})\simeq \prod_{I}\mo{End}_{\mo{Perf}_{\cat{C}}(L_i)}(\bb{1})\]
is an $I$-indexed product of even 2-periodic algebraically closed fields, and in particular has graded homogeneous spectrum exactly the Stone--\v{C}ech compactification $\beta I$ of the set $I$.

The Balmer spectrum of $\prod_{I}\mo{Perf}_{\cat{C}}(L_i)$ clearly surjects onto $\beta I$, since for any ultrafilter $\mc{U}$ on $I$, the ultraproduct $\prod_{\mc{U}}\mo{Perf}_{\cat{C}}(L_i)$ is nonzero, living over the point $\mc{U}\in\beta I$.  Note that this localization also has Balmer spectrum consisting of a single point by \Cref{lem:equivtensorfaithful} and \Cref{cor:NSianFieldIsh} since every nonzero object is $\otimes$-faithful, and the localization corresponds to the prime \[\mc{P}_\mc{U}=\{(x_i)_{i\in I}\in \prod_{I}\mo{Perf}_{\cat{C}}(L_i)\colon\{i\colon x_i\simeq 0\}\in\mc{U}\}.\]  

We claim that these are the only prime ideals.  Indeed, note first that for any $(x_i)_{i\in I}\in \prod_{I}\mo{Perf}_{\cat{C}}(L_i)$, the ideal generated by $x_i$ agrees with the ideal generated by the object $(\delta_{x_i\neq 0})_{i\in I}$ with a $\bb{1}$ in the position of every nonzero $x_i$ and zero elsewhere.  To see this, note that this second object (which is in fact idempotent) tensors with $(x_i)_{i\in I}$ to itself, and also splits off of $(x_i)_{i\in I}\otimes (x_i^{\vee})_{i\in I}$, again by \Cref{cor:NSianFieldIsh}.  Therefore any ideal $\mc{I}\subseteq \prod_{I}\mo{Perf}_{\cat{C}}(L_i)$ is determined by the set
\[\mc{F}(\mc{I})\coloneqq\{S\subseteq I\colon (\delta_{i\in S})_{i\in I}\notin \mc{I}\}.\]
In general, one can check that for any proper ideal $\mc{I}$, the set $\mc{F}(\mc{I})$ is a filter on the set $I$, and this gives a bijective correspondence between filters on $I$ and proper ideals in $\prod_{I}\mo{Perf}_{\cat{C}}(L_i).$  

Under this correspondence, an ideal $\mc{I}$ is prime if and only if, given any two sets $S_1,S_2\subseteq I$, $S_1\cup S_2\in \mc{F}(\mc{I})$ implies that $S_1\in \mc{F}(\mc{I})$ or $S_2\in\mc{F}(\mc{I})$, which is exactly the condition that $\mc{F}(\mc{I})$ is an ultrafilter, in which case the prime ideal $\mc{I}$ is nothing but $\mc{P}_{\mc{F}(\mc{I})}$, which must therefore be the only primes, as claimed.
\end{proof}
\begin{corollary}\label{cor:BalmerSpcSheafyProd}  Let $\cat{C}$ be a rational rigid 2-ring.  When restricted to the full sub-category $\mo{NS}^{\Pi}\left(\mo{CAlg}\left(\mo{Ind}\left(\cat{C}\right)\right)\right)$ of products of Nullstellensatzian commutative algebras in $\cat{C}$, the Balmer spectrum functor $R\mapsto \Spc(\mo{Perf}_{\cat{C}}(R))$, considered as a presheaf valued in $\mo{Set}^{op}$, is a sheaf for the c-topology.
\end{corollary}
\begin{proof}
To begin, note that if we have a cover of the form $\prod_{I}L_i\to \prod_{I}L_i^{\prime}$ for $L_i$, $L_i^{\prime}$ Nullstellensatzian commutative algebras and maps $L_i \to L_i^{\prime}$, then given distinct ultrafilters $\mc{U}_1$ and $\mc{U}_2$ on $I$, we have
\[\prod_{\mc{U}_1}L_i^{\prime}\otimes_{\prod_{I}L_i}\prod_{\mc{U}_2}L_i^{\prime}\simeq 0.\]
This implies that the two induced maps on Balmer spectra under 
\begin{center}
\begin{tikzcd}
\prod_{I}L_i^{\prime} \ar[r,shift left=.75ex]
  \ar[r,shift right=.75ex,swap]
&
\prod_{I}L_i^{\prime}\otimes_{\prod_{I}L_i}\prod_{I}L_i^{\prime}
\end{tikzcd}
\end{center}
actually agree, which forces $\Spc(\mo{Perf}_{\cat{C}}(-))$ to satisfy the sheaf condition on these covers.  

We must still check, given a cover of the form 
\[\prod_{I}L_i\to \prod_{J}L_j^{\prime},\]
that $\Spc(\mo{Perf}_{\cat{C}}(\prod_{I}L_i))$ is the quotient of $\Spc(\mo{Perf}_{\cat{C}}(\prod_{J}L_j^{\prime}))$ under the two maps from $\Spc(\mo{Perf}_{\cat{C}}(\prod_{J}L_j^{\prime}\otimes_{\prod_{I}L_i}\prod_{J}L_j^{\prime}))$.
Using \Cref{Prop:NicecCov}, fix some refinement of this cover of the form $\prod_{I}L_i\to \prod_{J}L_j^{\prime}\to \prod_{I}L_i^{\prime\prime}$.  

Fix some ultrafilter $\mc{U}$ on $I$, and denote by $\mc{U}^{\prime}$ the ultrafilter on $J$ such that $\prod_{J}L_j^{\prime}\to \prod_{\mc{U}}L_i^{\prime\prime}$ factors over $\prod_{\mc{U}^{\prime}}L_j^{\prime}$, that is, the image of $\mc{U}$ in $\Spc(\mo{Perf}_{\cat{C}}(\prod_{J}L_j^{\prime}))$.  Take any other ultrafilter $\mc{V}$ on $J$ which maps to $\mc{U}$ in $\Spc(\mo{Perf}_{\cat{C}}(\prod_{I}L_i))$, we claim that there is a point mapping to both $\mc{V}$ and $\mc{U}^{\prime}$ under the two maps we wish to co-equalize.  Using once again (the proof of) \Cref{Prop:NicecCov}, we find that the map 
\[\prod_{\mc{U}}L_i\to \prod_{\mc{U}}L_i^{\prime\prime}\]
is nil-conservative, and by construction, this factors over the map
\begin{equation}\label{eq:nilcons}
\prod_{\mc{U}}L_i\to \prod_{\mc{U}^{\prime}}L_j^{\prime},
\end{equation}
which must therefore be nil-conservative as well.  Finally, we note that to say $\mc{V}$ maps to $\mc{U}$, is to say that the map
\[\prod_{I}L_i\to \prod_{\mc{V}}L_j^{\prime}\]
factors over $\prod_{\mc{U}}L_i$, and by nil-conservativity of \ref{eq:nilcons}, we have that
\[\prod_{\mc{V}}L_j^{\prime}\otimes_{\prod_{\mc{U}}L_i}\prod_{\mc{U}^{\prime}}L_j^{\prime}\neq 0,\]
and this gives a point of $\Spc(\mo{Perf}_{\cat{C}}(\prod_{J}L_j^{\prime}\otimes_{\prod_{I}L_i}\prod_{J}L_j^{\prime}))$ which co-equalizes $\mc{U}^{\prime}$ and $\mc{V}$, as desired.
\end{proof}
This allows us to construct a sheafification for the Balmer spectrum functor in the c-topology, which takes values in a category which is not presentable, but that we can nevertheless work with. For this, we first recall some basic facts about sheaves in not necessarily presentable $1$-categories. 
\begin{proposition}\label{prop:sheavesbasis}
Let $C$ be a Grothendieck site, and $i:B\subset C$ a subbasis for the topology: that is, a full subcategory such that each $c\in C$ admits a cover by an object in $B$.

In this case, restriction along the inclusion $i$ induces an equivalence on categories of sheaves with values in any complete $1$-category, with inverse given by right Kan extension. 
\end{proposition}
\begin{proof}
For the target category being $\mo{Set}$, this is \cite[Exposé III, Théorème 4.1]{SGA4}. For a general complete target category, use the Yoneda lemma to reduce to $\mo{Set}$. 

The claim about the inverse follows \cite[Exposé III, Proposition 2.3, 3)]{SGA4}, with the same reduction to the case of sets. 
\end{proof}
\begin{remark}
For $\infty$-categories of coefficients, one would have to restrict to hypersheaves. 
\end{remark}
\begin{corollary}\label{cor:sheafifification}
Let $C$ be a Grothendieck site and $B$ a subbasis on $C$. Let $D$ be a complete $1$-category and $F:C\op\to D$ be a presheaf. 

If the restriction of $F$ to $B$ is a sheaf, then the sheaf $\tilde F$ on $C$ corresponding to $F_{\mid B}$ under the equivalence from the previous Proposition is the sheafification of $F$; in particular $F$ admits a sheafification. 
\end{corollary}
\begin{proof}
Let $G$ be a sheaf. By the right Kan extension property of \Cref{prop:sheavesbasis} applied to $G$, restriction to $B$ induces an isomorphism 
\[
\hom(F,G) \cong \hom(F_{\mid B}, G_{\mid B})
\]
Now since both are sheaves, \Cref{prop:sheavesbasis} implies that this is isomorphic (again via restriction to $B$) to $\hom(\tilde F,G)$, as was to be shown. 
\end{proof}
In particular, concretely, for any object $x\in C$, the value of the sheafification of $F$ at $x$ is given by finding a hypercover $y_1 \rightrightarrows y_0 \to x$, with $y_0,y_1 \in B$, and taking $\mo{eq}(F(y_0) \rightrightarrows F(y_1))$. This implies in particular that these results remain valid for large sites. 
\begin{corollary}\label{cor:localsheafification}
Let $C$ be a Grothendieck site and $B$ a subbasis on $C$. Let $D$ be a complete $1$-category and $F:C\op\to D$ be a presheaf. 

If the restriction of $F$ to $B$ is a sheaf, and $F\to G$ is a map to a sheaf which is an isomorphism locally on $B$, then it is a sheafification. 
\end{corollary}

One can look at the comparison map from the Balmer spectrum to the spectrum of graded homogeneous prime ideals in the graded endomorphism ring of the unit, rationally.  Although these functors will often differ, as a corollary of what we have seen thus far, we find that when we sheafify them with respect to the c-topology, the sheafifications will in particular exist, and they will become equal.
\begin{corollary}\label{cor:BalmerSpccSheafification}
Let $\cat{C}$ be a rational rigid 2-ring.  The sheafification of the Balmer spectrum functor 
\[\Spc(\mo{Perf}_{\cat{C}}(-))\colon \mo{CAlg}\left(\mo{Ind}\left(\cat{C}\right)\right)\to \mo{Set}^{op}\]
with respect to the c-topology exists, and agrees with its value on products of Nullstellensatzian objects.  Furthermore, this sheafification is equal to the c-sheafification of the functor sending $R$ to the spectrum of graded homogeneous prime ideals in the underlying $\bb{E}_{\infty}$-ring of $R$:
\[R\mapsto \mo{Spec}^{hom}(\pi_*(\mo{Hom}_{\cat{C}}(\bb{1},R))).\]
\end{corollary}
\begin{proof}
When restricted to the sub-category of products of Nullstellensatzian objects, these two functors agree by \Cref{lem:NSianProductSpc}.  Now, since this restriction is a c-sheaf by \Cref{cor:BalmerSpcSheafyProd}, and this full sub-category forms a sub-basis for the c-topology, the sheafifications of both functors exist and agree by \Cref{cor:sheafifification}.
\end{proof}

In nice cases, we can say more about this constructible topology.
\begin{theorem}\label{thm:NoSandENCNSian}
Let $\cat{C}$ be a rational rigid 2-ring, and consider the following statements.
\begin{enumerate}
\item  There exists some $n$ such that for all Nullstellensatzian commutative algebras $L\in\mo{CAlg}\left(\mo{Ind}\left(\cat{C}\right)\right)$, the category $\mo{Perf}_{\cat{C}}(L)$ satisfies the exact-nilpotence condition to order $n$.
\item  Ultraproducts of Nullstellensatzian objects in $\mo{CAlg}\left(\mo{Ind}\left(\cat{C}\right)\right)$ have constructible spectrum consisting of a single point, which is to say that ultraproducts are pointlike.  In particular, in this case the category $\mo{CAlg}\left(\mo{Ind}\left(\cat{C}\right)\right)$ is spectral in the sense of \cite[Definition A.63]{2022arXiv220709929B}.
\item  For all $R\in \mo{CAlg}\left(\mo{Ind}\left(\cat{C}\right)\right)$, the sheafification of the Balmer spectrum functor $R\mapsto \Spc(\mo{Perf}_{\cat{C}}(R))$ landing in $\left(\mo{Set}\right)^{op}$ with respect to the c-topology agrees with $\Speccons(\mo{Perf}_{\cat{C}}(R))$. 

\item  For all $R\in \mo{CAlg}\left(\mo{Ind}\left(\cat{C}\right)\right)$, the constructible topology on $\Speccons(\mo{Perf}_{\cat{C}}(R))$ is Hausdorff.
\end{enumerate}
Then, (a)$\implies$(b)$\iff$(c)$\implies$(d).
\end{theorem}
\begin{proof}
(a)$\implies$(b) Consider an ultrafilter $\mc{U}$ on a set $I$, and an $I$-indexed family $L_i$ of Nullstellensatzian algebras in $\cat{C}$.  We wish to show that the constructible spectrum of the ultraproduct $\prod_{\mc{U}}L_i$ is a single point.  By \Cref{thm:consequalshom}, it suffices to show that the homological spectrum of $\mo{Perf}_{\cat{C}}(\prod_{\mc{U}}L_i)$ is a single point.  Using full faithfulness (and hence nil-conservativity) of the inclusion \[\mo{Perf}_{\cat{C}}(\prod_{\mc{U}}L_i)\to \prod_{\mc{U}}\mo{Perf}_{\cat{C}}(L_i),\] it suffices to show that the homological spectrum of $\prod_{\mc{U}}\mo{Perf}_{\cat{C}}(L_i)$ is a single point.  Now, by assumption, the categories $\mo{Perf}_{\cat{C}}(L_i)$ satisfy ENC($n$), and so too does their ultraproduct, whose homological spectrum therefore agrees with its Balmer spectrum.
\Cref{lem:NSianProductSpc} further shows that the Balmer spectrum of the ultraproduct is a single point.

(b)$\implies$(c)  By \cite[Corollary A.46]{2022arXiv220709929B}, the target of the natural comparison map
\[\Spc(\mo{Perf}_{\cat{C}}(R))\to \Speccons(\mo{Perf}_{\cat{C}}(R))\]
of presheaves valued in $\left(\mo{Set}\right)^{op}$ is such that the target is actually a sheaf. By \Cref{cor:localsheafification}, in order to check that the target is a sheafification of the source, it suffices to check that this map is an isomorphism locally for the c-topology.

Taking a c-cover \[R\to \prod_{q\in \Speccons(R)}(L_q)\] given by a product of Nullstellensatzian representatives $L_q$ for each point in the constructible spectrum of $R$, we may assume that $R=\prod_{I}L_i$ is a product of Nullstellensatzian $\cat{C}$-algebras.  Using that ultraproducts are pointlike and $\mo{CAlg}\left(\mo{Ind}\left(\cat{C}\right)\right)$ is op-disjunctive by \cite[Lemma A.58]{2022arXiv220709929B}, \cite[Lemma A.61]{2022arXiv220709929B} implies that the constructible spectrum of a product $\Speccons(\prod_{I} L_q)$ is given by the Stone--\v{C}ech compactification $\beta I$ of the set $I$.  Similarly, \Cref{lem:NSianProductSpc} shows that the Balmer spectrum of $\mo{Perf}_{\cat{C}}(\prod_{I}L_i)$ is also the Stone--\v{C}ech compactification of $I$, with the local category at the prime $\mc{U}\in \beta I$ given by the ultraproduct $\mo{Perf}_{\cat{C}}(\prod_{\mc{U}}L_i)$, which tells us that the comparison map \[\Speccons(\prod_{I}L_i)\to \Spc(\mo{Perf}_{\cat{C}}(\prod_{I}L_i))\]
is an equivalence, and the claim is shown.

(c)$\implies$(b)  Using \Cref{cor:BalmerSpccSheafification}, the sheafification of the Balmer spectrum functor with respect to the c-topology agrees with the Balmer spectrum functor on products of Nullstellensatzian commutative algebras.  To say that this sheafification is the constructible spectrum implies that, for such a product $\prod_{I}L_i$, the comparison map induces an equivalence
\[\Speccons(\mo{Perf}_{\cat{C}}(\prod_{I}L_i))\simeq \Spc(\mo{Perf}_{\cat{C}}(\prod_{I}L_i)). \]
That is to say, the nerves of steel condition holds for the category $\mo{Perf}_{\cat{C}}(\prod_{I}L_i)$.  Translating to the local rings over a point in the Balmer spectrum, we find that, for any ultrafilter $\mc{U}$ on a set $I$, and an $I$-indexed set of Nullstellensatzians, we have
\[\Speccons(\mo{Perf}_{\cat{C}}(\prod_{\mc{U}}L_i))\simeq \Spc(\mo{Perf}_{\cat{C}}(\prod_{\mc{U}}L_i))\simeq *,\]
which is to say, ultraproducts in $\mo{CAlg}\left(\mo{Ind}\left(\cat{C}\right)\right)$ are pointlike.

(b)$\implies$(d)  This follows from \cite[Proposition~A.62]{2022arXiv220709929B}.
\end{proof}
\begin{corollary}  Let $\cat{C}$ be a rational rigid 2-ring.  Then the following are equivalent:
\begin{enumerate}
\item  The weakly spectral category $\mo{CAlg}\left(\mo{Ind}\left(\cat{C}\right)\right)$ is spectral.
\item  Given an ultrafilter $\mc{U}$ on a set $I$, and an $I$-indexed collection $L_i\in \mo{CAlg}\left(\mo{Ind}\left(\cat{C}\right)\right)$ of Nullstellensatzian objects, the nerves of steel condition holds for the category $\mo{Perf}_{\cat{C}}(\prod_{\mc{U}}L_i)$.
\item  Given a set $I$, and an $I$-indexed collection $L_i\in \mo{CAlg}\left(\mo{Ind}\left(\cat{C}\right)\right)$ of Nullstellensatzian objects, the nerves of steel condition holds for the category $\mo{Perf}_{\cat{C}}(\prod_{I}L_i)$.
\end{enumerate}
\end{corollary}
\begin{proof}
(a)$\iff$(b)  Note that $\mo{CAlg}\left(\mo{Ind}\left(\cat{C}\right)\right)$ is spectral if and only if ultraproducts are pointlike, which is to say, given an ultrafilter $\mc{U}$ on a set $I$, and an $I$-indexed collection $L_i$ of Nullstellensatzian objects, the spectrum $\Speccons_{\cat{C}}(\prod_{\mc{U}}L_i)$ is a single point.  Once again, by \Cref{cor:NSianFieldIsh}, every non-zero object in $\mo{Perf}_{\cat{C}}(\prod_{\mc{U}}L_i)$ is $\otimes$-faithful, and in particular the Balmer spectrum has only a single point.  In particular, the nerves of steel condition holds for this category if and only if $\Speccons_{\cat{C}}(\prod_{\mc{U}}L_i)$ is a single point, if and only if $\mo{CAlg}\left(\mo{Ind}\left(\cat{C}\right)\right)$ is spectral.

(b)$\iff$(c)  Again, it follows from \Cref{cor:NSianFieldIsh} that the Balmer spectrum of $\mo{Perf}_{\cat{C}}(\prod_{I}L_i)$ is given by the Stone--\v{C}ech compactification $\beta I$ of $I$.  In particular, by the fiber-wise criterion for the nerves of steel condition, this category satisfies NoS if and only if all of its localizations to closed points do.  Since the points of $\beta I$ are given by ultrafilters $\mc{U}$ on $I$, with corresponding local rigid 2-ring $\mo{Perf}_{\cat{C}}(\prod_{\mc{U}}L_i)$ at the prime represented by $\mc{U}$, the claim follows.
\end{proof}

If we are allowed to work with all rational rigid 2-rings, the conditions of \Cref{thm:NoSandENCNSian} strengthen to all be equivalent, and in fact one can add a few more:
\begin{proposition}\label{prop:CHausmaybe}
The following are equivalent
\begin{enumerate}
\item  For all rational rigid 2-rings $\cat{C}$, the space $\Speccons(\cat{C})$ is Hausdorff.
\item  The space $\Speccons(\Apoi_{\eta})$ is Hausdorff.
\item  There exists some $n\gg 0$ such that the algebras $\FrEinfz(\cofib(f^{\vee,\otimes n}))$ and $\FrEinfz(\cofib(g^{\otimes n}))$ in $\Apoi_{\eta}$ are jointly nil-conservative, with $f$ the free map from the unit to the free pointed object, and $g$ its fiber.
\item  There exists $n\gg 0$ such that for all Nullstellensatzian rational rigid 2-rings $\cat{L}$ which contain an object $Z$ such that $\A\xrightarrow{Z}\cat{L}$ factors over $\A_{\eta}$, $\cat{L}$ satisfies ENC(n).
\item  There exists some $n\gg 0$ such that all Nullstellensatzian rational rigid 2-rings satisfy ENC(n).
\item  The category $\ttCat_{\bb{Q}}$ is spectral.
\end{enumerate}
\end{proposition}
\begin{proof}
(a)$\implies$(b) is clear.

(b)$\implies$(c)  The points in $\Apoi_{\eta}$ determined by $p\colon\Apoi_{\eta}\to \A_{\eta}$ and $p\circ \iota$ have residue fields picked out by the $\bb{E}_{\infty}$-algebras $\FrEinfz(f)$ and $\FrEinfz(g^{\vee})$, respectively.  If $\Speccons(\Apoi_{\eta})$ is Hausdorff, then there exists disjoint open neighborhoods of these two points.  These open subsets are represented by functors $\Apoi_{\eta}\to \cat{C}$ and $\Apoi_{\eta}\to \cat{D}$, such that $\FrEinfz(f)=0$ in $\cat{C}$, $\FrEinfz(g^{\vee})=0$ in $\cat{D}$, and $\Apoi_{\eta}\to \cat{C}\times \cat{D}$ is nil-conservative.  Since $\FrEinfz(f)\simeq 0$ in $\cat{C}$, \Cref{prop:CAlgDetection} implies that $f^{\otimes m}\simeq 0$ in $\cat{C}$ for some $m\gg 0$.  This means in particular that $\FrEinfz(\cofib(f^{\vee,\otimes m}))$ admits a splitting back to the unit after base-changing it into $\cat{C}$, which in turn implies that $\Apoi_\eta\to \cat{C}$ factors over \[\Apoi_{\eta}\to \mo{Perf}_{\Apoi_{\eta}}(\FrEinfz(\cofib(f^{\vee,\otimes m})))\to \cat{C}.\]
Similarly, there exists some $k$ such that $\Apoi_{\eta}\to\cat{D}$ factors as
\[\Apoi_{\eta}\to\mo{Perf}_{\Apoi_{\eta}}(\FrEinfz(\cofib(g^{\otimes k})))\to \cat{D}.\]
Up to increasing $k$ or $m$ as needed, we may assume without loss of generality that $k=m=n$, for some $n\gg 0$.  Now, since $\Apoi_{\eta}\to \cat{C}$ and $\Apoi_{\eta}\to \cat{D}$ are jointly nil-conservative, so too are the functors \[\Apoi_{\eta}\to\mo{Perf}_{\Apoi_{\eta}}(\FrEinfz(\cofib(g^{\otimes n})))\] and 
\[\Apoi_{\eta}\to\mo{Perf}_{\Apoi_{\eta}}(\FrEinfz(\cofib(f^{\vee,\otimes n}))),\]
as claimed.

(c)$\implies$(d)  Let $\cat{L}$ be a rational Nullstellensatzian rigid 2-ring containing such an object $Z$, and consider any fiber sequence 
\[Y\xrightarrow{g}\bb{1}\xrightarrow{f}X\]
in $\cat{L}$.  Up to replacing $f\colon \bb{1}\to X$ by $f\oplus 0\oplus 0\colon \bb{1}\to X\oplus \Sigma X\oplus Z$ if needed, we may assume without loss of generality that the functor $\Apoi\xrightarrow{f}\cat{L}$ factors over $\Apoi_{\eta}$.  

By (c), there exists some $n$ such that the algebras $\FrEinfz(\cofib(f^{\vee,\otimes n}))$ and $\FrEinfz(\cofib(g^{\otimes n}))$ are jointly nil-conservative on $\Apoi_{\eta}$, and so too then will be their basechange to $\cat{L}$.  Since these are free $\bb{E}_{\infty}$-algebras in $\cat{L}$ on \textit{compact} $\bb{E}_0$-algebras, and $\cat{L}$ is Nullstellensatzian, the map from the unit splits as soon as one of them is nonzero- and one must be nonzero because they are jointly nil-conservative.  

By symmetry of cases, assume that, in $\cat{L}$, $\FrEinfz(\cofib(g^{\otimes n}))\neq 0$, such that this algebra splits back to $\bb{1}_{\cat{L}}$.  Then the map $\cofib(g^{\otimes n})\colon \bb{1}\to W$ splits, which implies that $g^{\otimes n}\simeq 0$, and $\cat{L}$ satisfies ENC($n$).

(d)$\implies$(e)  Using \Cref{Cor:NSianVeryFieldIsh} (applied e.g. to $\cat{C}=\cat{L}$ and $A=\bb{1}$), we find that any functor out of a Nullstellensatzian rational rigid 2-ring $\cat{L}$ is conservative for morphisms between compact objects.  Taken together with \Cref{cor:NSianFieldIsh}, which tells us that we are free to always take $Z=\bb{1}$ when checking the exact-nilpotence condition in Nullstellensatzian rational rigid 2-ring, we find that any failure of ENC($n$) in such a category $\cat{L}$ will provide a similar failure for every Nullstellensatzian rational rigid 2-ring $\cat{L}^{\prime}$ admitting a map from $\cat{L}$.

It remains to show that any Nullstellensatzian $\cat{L}$ admits a map to another Nullstellensatzian rigid 2-ring $\cat{L}^{\prime}$ satisfying the hypothesis in (d).  For this, use that $\cat{D}^{b}(\bb{Q})\to \A_{\eta}$ is nil-conservative, and so too then is $\cat{L}\to \cat{L}\otimes \A_{\eta}$, and in particular the target is nonzero.  For any choice of map $\cat{L}\otimes \A_{\eta}\to \cat{L}^{\prime}$ to a Nullstellensatzian rigid 2-ring $\cat{L}^{\prime}$ yields a map $\cat{L}\to\cat{L}^{\prime}$ with $\cat{L}^{\prime}$ satisfying the hypothesis in (d).

(e)$\implies$(f)  Since $\ttCat_{\bb{Q}}$ can be shown to be op-disjunctive, it remains to show that ultraproducts are pointlike.  Using \Cref{thm:consequalshom}, this is equivalent to showing that for any ultrafilter $\mc{U}$ on a set $I$, and $I$-indexed collection $\cat{L}_i$ of Nullstellensatzian rigid 2-rings, the ultraproduct $\prod_{\mc{U}}\cat{L}_i$ has homological spectrum consisting of a single point.  Since its Balmer spectrum is a single point (as once again every object is $\otimes$-faithful), this is equivalent to the Nerves of steel condition holding for this product, which is itself, by locality, equivalent to the exact-nilpotence condition in this category.  Since the ultraproduct of local categories satisfying ENC($n$) is still local, and still satisfies ENC($n$), the claim follows.

(f)$\implies$(a)  This follows by \cite[Theorem A.65]{2022arXiv220709929B}.
\end{proof}
\begin{remark}  Since it seems likely that Nullstellensatzian rational rigid 2-rings \textit{should} be called tt-fields, in the case that ultraproducts in rational rigid 2-rings are not pointlike, it would mean that one does not expect ultraproducts of a good notion of ``tt-fields'' to remain tt-fields.  This would be a point of divergence from classical algebraic geometry, where of course, ultraproducts of fields remain fields.
\end{remark}

\newpage
\section{The $\bb{E}_n$-Constructible Spectrum}\label{sec:enconsspec}  
\subsection{$\bb{E}_n$-constructible spectra} We have seen that the $\bb{E}_{\infty}$-constructible spectrum of the unit in a rigid 2-ring may not necessarily even surject onto the Balmer spectrum.  In contrast, it turns out that fixing any finite $n\geq 1$ and working with the $\bb{E}_n$-constructible spectrum, this problem can essentially be fixed because of Burklund's work \cite{burklund2022multiplicativestructuresmoorespectra}.  The natural generality to work in for the purposes of this section is with rigid $\bb{E}_n$-2-rings, $n\geq 3$. We make the following definition. 
\begin{definition}
Let $\mo{Ind}(\cat C)$ be the Ind category of a rigid $\bb{E}_m$-2-ring.  For any $1\leq n\leq m$, and any $A\in\Algn(\mo{Ind}(\cat C))$, let 
\[\Speccons_{\bb{E}_n}(A)\coloneqq \Speccons_{\Algn(\mo{Ind}(\cat C))}(A)\]
be the constructible spectrum of $A$ in the weakly spectral\footnote{See \cite[Remark A.10]{2022arXiv220709929B}.} category of $\bb{E}_n$-algebra objects in $\mo{Ind}(\cat C)$.
\end{definition}

The main goal of this subsection is to prove the following theorem.

\begin{theorem}\label{thm:nconsequalshom}
Let $\mo{Ind}(\cat C)$ be the Ind category of a rigid $\bb{E}_m$-2-ring.  Then for all $1\leq n < m$, there is a natural equivalence of sets
\[\Spc^h(\cat C)\simeq \Speccons_{\bb{E}_n}(\bb{1}_{\cat C}).\]
\end{theorem}
As a consequence of this theorem, we can give a new topology to the homological spectrum: the constructible topology.  In its usual topology, the homological spectrum is not necessarily even $T_0$, but in the constructible topology it becomes $T_1$!  
\begin{remark}
The topology defined via the constructible topology likely agrees with the one definable from \cite{BirdWilliamson2025} through the identification of the homological spectrum with their ``closed Zeigler spectrum'' $\mo{KZg}_{\mo{Cl}}^{\otimes}(-)$, which naturally carries the topology of a $T_1$-space.
\end{remark}
Before giving the proof, we require a few generalities on homological spectra.

\begin{proposition}\label{prop:EnHomFields}  Let $\mo{Ind}(\cat C)$ be the Ind category of a rigid $\bb{E}_m$-2-ring, and take any $1\leq n\leq m$.  Let $\mf{m}\in\Spc^h(\cat C)$ be a homological prime, and let $E_{\mf{m}}$ be the corresponding homological residue field.  Write $E_{\mf{m}}=\mo{cofib}(v\colon I_{\mf{m}}\to \bb{1})$.  
	
Then the object 
\[
    E_{\mf{m}}^{n}\coloneqq \mo{cofib}(v^{n+1}\colon I_{\mf{m}}^{\otimes n+1}\to \bb{1})
\]
is an $\bb{E}_n$-algebra with homological support $\supph(E_{\mf{m}}^{n})=\{\mf{m}\}$.
\end{proposition}
\begin{proof}
Since $E_{\mf{m}}$ is a weak ring, it admits a right unital multiplication.  Fixing a choice of right unital multiplication, we find that $E_{\mf{m}}^n$ inherits an $\bb{E}_n$-algebra structure by \cite[Theorem~1.5]{burklund2022multiplicativestructuresmoorespectra}.  Since $E_{\mf{m}}\otimes v$ is nullhomotopic, we have that $E_{\mf{m}}\otimes E_{\mf{m}}^n\neq 0$, so $\mf{m}\in\supph(E_{\mf{m}}^n)$.

On the other hand, if we had any $\mf{m}^{\prime}\in\Spc^h(\cat C)$ different from $\mf{m}$, $E_{\mf{m}^{\prime}}\otimes E_{\mf{m}}\simeq 0$ implies that $E_{\mf{m}^{\prime}}\otimes v$ is an equivalence, and then so too is $E_{\mf{m}^{\prime}}\otimes v^{n+1}$, and hence $E_{\mf{m}^{\prime}}\otimes E_{\mf{m}}^{n}\simeq 0$ as well.
\end{proof}
We make use of the following simple observation.
\begin{lemma}\label{lem:MappingSupport}
If $A,B$ are weak rings in $\mo{Ind}(\cat C)$, and there is a map \[A\to B\] of pointed objects, then 
\[\supph(B)\subseteq \supph(A).\]
In particular, if $B\neq 0$ and $\supph(A)$ is a single point, then $\supph(B)=\supph(A)$.
\end{lemma}
\begin{proof}
Suppose that we are given $\mf{m}\notin \supph(A)$, so that $E_{\mf{m}}\otimes A\simeq 0$.  Note that $E_{\mf{m}}\otimes B$ is a tensor product of weak rings, so is itself a weak ring.  Using that the map \[\bb{1}\to E_{\mf{m}}\otimes B\] factors over \[\bb{1}\to E_{\mf{m}}\otimes A\simeq 0\to E_{\mf{m}}\otimes B\] 
the zero map, we see that we must have $E_{\mf{m}}\otimes B\simeq 0$ as well.
\end{proof}

\begin{proposition}\label{prop:pushoutsupportprop}  Let $A, B$ be $\bb{E}_n$-algebras in the Ind category of a rigid $\bb{E}_m$-2-ring $\mo{Ind}(\cat C)$, with $1\leq n < m$.  Then, we have that 
	\[\supph(A\coprod B)=\supph(A)\cap \supph(B)=\supph(A\otimes B),\]
and in particular, $A\coprod B\neq 0$ if and only if $A\otimes B\neq 0$.	
\end{proposition}
\begin{proof}
Since $n<m$, $A\otimes B$ can be given a canonical $\bb{E}_n$-algebra structure, and the maps $A\to A\otimes B, B\to A\otimes B$ can be made canonically $\bb{E}_n$. 

Thus, there exists an $\bb{E}_n$-algebra map \[A\coprod B\to A\otimes B\] by the universal property, so \Cref{lem:MappingSupport} ensures that
\[\supph(A\otimes B)\subseteq \supph(A\coprod B).\]
Similarly, the $\bb{E}_n$-algebra maps
\[A\to A\coprod B \quad \text{ and } \quad B\to A \coprod B\]
show that 
\[\supph(A\coprod B)\subseteq \supph(A)\cap \supph(B).\]
Since $\supph(A\otimes B)= \supph(A)\cap\supph(B)$ by \Cref{prop:tensorproductprop}, we get the desired equality. 

The final claim follows from \Cref{prop:suppdetection}.
\end{proof}

Now we can finally begin to introduce the constructible spectrum to the picture, by way of describing homological support of Nullstellensatzian $\bb{E}_n$-algebras:
\begin{proposition}\label{prop:NSianPointlikeHomSupport}
Let $L$ be a Nullstellensatzian $\bb{E}_n$-algebra in $\mo{Ind}(\cat C)$, with $n<m$.  Then the homological support $\supph(L)$ of $L$ consists of a single point in the homological spectrum.
\end{proposition}
\begin{proof}
Indeed, suppose that $\mf{m}$ and $\mf{m}^{\prime}$ are two distinct points in $\supph(L)$. By \Cref{prop:pushoutsupportprop}, the following two $\bb{E}_n$-algebras under $L$
\[E_{\mf{m}}^n\coprod L, E_{\mf{m}^{\prime}}^n\coprod L\]
are nonzero, but
\[(E_{\mf{m}}^n\coprod L)\coprod (E_{\mf{m}^{\prime}}^n\coprod L)\simeq 0.\]
This in turn implies that
\[(E_{\mf{m}}^n\coprod L)\coprod_{L} (E_{\mf{m}^{\prime}}^n\coprod L)\simeq 0\]
as well.  

This tells us that the images of the two maps on constructible spectra of induced by these algebras are disjoint.  However, we assumed that $L$ was Nullstellensatzian, so that $\Speccons_{\bb{E}_n}(L)$ is a single point by \cite[Proposition A.31(2)]{2022arXiv220709929B}. Therefore, the constructible spectrum of one of these algebras must actually be empty, and that algebra must then be zero, a contradiction.
\end{proof}
Finally, we can prove the theorem.
\begin{proof}[Proof of \Cref{thm:nconsequalshom}]
We claim that the map
\[\Speccons_{\bb{E}_n}(\bb{1}_{\cat C})\to \Spc^h(\cat C)\]
taking any point of the source, represented by a Nullstellensatzian $\bb{E}_n$-algebra $L$, to the point $\supph(L)$, is a bijection.  

For injectivity, note that if $L_1$ and $L_2$ are two Nullstellensatzian $\bb{E}_n$-algebras with $\supph(L_1)=\supph(L_2)$, then \Cref{prop:pushoutsupportprop} implies 
\[\supph(L_1\coprod L_2)=\supph(L_1)\cap \supph(L_2)\neq \emptyset,\]
and in particular $L_1\coprod L_2\neq 0$, so that $L_1, L_2$ represent the same point in the constructible spectrum.

For surjectivity, fix an arbitrary $\mf{m}\in\Spc^h(\cat C)$, take the $\bb{E}_n$-ring $E_{\mf{m}}^n$ from \Cref{prop:EnHomFields}, and choose some Nullstellensatzian $\bb{E}_n$-algebra $L$ under it.  Since $L$ is non-zero, and \[\supph(E_{\mf{m}}^n)=\{\mf{m}\}\] is a singleton, \Cref{lem:MappingSupport} shows that \[\supph(L)=\{\mf{m}\}\] as well.  Since the prime $\mf{m}$ was arbitrary, surjectivity follows. 
\end{proof}

\begin{remark}
Given an $\bb{E}_n$-algebra $R$, one can also consider $R$ as an $\bb{E}_k$-algebra for any $1\leq k\leq n$, and define a constructible spectrum that way.  This will in general be different from the homological spectrum as sets.  For example, consider a field $K$, and take the polynomial ring on $K$ in one variable, $K[t]$.  The map 
\[
    K[t]\to \mo{End}_{K}\left(\bigoplus_{\bb{N}}K\right)
\]
sending $t$ to the endomorphism $e_i\mapsto e_{2i}$ induces a map the other way on $\bb{E}_1$-constructible spectra.  We claim that the image of this map misses the usual (commutative) spectrum of $K[t]$.

First, note that since $t$ maps to an element which has a left inverse, we cannot find a common factorization of the map to this algebra with the residue field where $t=0$.

In the same vein, we can choose two left inverses, e.g. $f\colon e_{2i}\mapsto e_{i}$, $e_{2i+1}\mapsto 0$, and $g\colon e_{2i}\mapsto e_i$, $e_{2i+1}\mapsto e_{i}$, such that $g-f$ has a right inverse, and so there is no algebra map out of the target where $f=g$.  This tells us that there is no map to a non-zero algebra where $t$ obtains a right inverse.  In particular, we cannot have a common refinement of this associative ring map together with a residue field of $K[t]$ where $t\neq 0$, so the constructible spectrum of $K[t]$ considered as an associative algebra is strictly larger than its constructible spectrum in commutative rings (which agrees with its homological spectrum).
\end{remark}
\subsection{\Cref{thm:nconsequalshom} is sharp}  We will now include a specific example to show that the requirement of $n<m$ in \Cref{prop:pushoutsupportprop} is sharp, and we will use this to show the same is true in \Cref{thm:nconsequalshom}.  For the purpose of this section, fix some $n\geq 4$, and consider the free $\bb{E}_n$-$\bb{F}_2$-algebra $R\coloneqq \bb{F}_2\{x,y\}_n$ on two generators $x,y$ in degree 0.  The category $\mo{Perf}(R)$ is naturally a rigid $\bb{E}_{n-1}$-2-ring, which we now proceed to analyze.
\begin{notation}
Fix some basis $E$ of the free $\bb{F}_2$-Lie algebra on two generators $x,y$, such that 
\begin{enumerate}
\item  The generators $x,y$ are in $E$.
\item  The bracket between the generators, $[x,y]$, is in $E$.
\item  Every basis vector in $E$ can be written as an expression involving only brackets, $x$s and $y$s.
\end{enumerate}
\end{notation}
The exact choice of basis itself isn't so important, as long as it satisfies the properties above. 
\begin{proposition}\label{prop:F2Enbasis2elmts}
The homotopy ring of $R$ is isomorphic to a graded polynomial ring:
\[\pi_*(R)\simeq \bb{F}_2[Q_I(e)\colon e\in E]\]
for admissible sequences $I=(a_1,\ldots,a_{n-1})$ of non-negative integers, and basis vectors $e\in E$.  The grading on a basis vector $e$ is determined by the rules $|x|=|y|=0$, and $|[-,-]|=n-1$.
\end{proposition}
\begin{proof}
This follows from \cite[Theorem~5.5]{lawson2020enringspectradyerlashof}.
\end{proof}
\begin{construction}\label{cons:En-1Cofibers}
Consider the $\bb{E}_{n-1}$-cofiber of $x\colon\bb{F}_2\to R$, denoted by $R{\sslash}^{n-1}x$.  By \cite[Lemma~2.2]{riedel2025reducedpointsmathbbeinftyringspositive}, we have that 
\[R{\sslash}^{n-1}x\simeq R\otimes_{\bb{F}_2\{x\}_n}\bb{F}_2,\]
and in particular, this algebra has homotopy groups given by the quotient of $R$ by the regular sequence $\{Q_{I}(x)\}$ over all admissible sequences $I$.
\end{construction}
\begin{proposition}\label{prop:pushoutEn-1optimal}
There exists $\bb{E}_{n-1}$-$R$-algebras $A,B$ such that $A\otimes_{R}B\neq 0$, but $A\coprod_{R}B\simeq 0$ in the category of $\bb{E}_{n-1}$-$R$-algebras.  In particular, the requirement that $n<m$ in the statement of \Cref{prop:pushoutsupportprop} is necessary.
\end{proposition}
\begin{proof}
Take $A\coloneqq R{\sslash}^{n-1}x[([x,y]^{-1})]$ as a localization of the algebra from \Cref{cons:En-1Cofibers}, and let $B\coloneqq R{\sslash}^{n-1}y$.  Then, by the explicit description from \Cref{prop:F2Enbasis2elmts}, we learn that 
\[A\otimes_R B\simeq A/(Q_{I}(y))_{I\text{ admissible}}\neq 0.\]

On the other hand, taking the pushout in $\bb{E}_{n-1}$-$R$-algebras, and once again using \cite[Lemma~2.2]{riedel2025reducedpointsmathbbeinftyringspositive}, we learn that
\begin{align*}
A\coprod_{R}B \simeq R{\sslash}^{n-1}(x,y)[([x,y])^{-1}] &\simeq R[([x,y])^{-1}]\otimes_{\bb{F}_2\{x,y\}_n}\bb{F}_2\\
&\simeq \bb{F}_2[0^{-1}]\simeq 0,
\end{align*}
as desired.
\end{proof}
Our goal now is to extend this into showing that \Cref{thm:nconsequalshom} fails for $\mo{Perf}(R)$ if we were to look at the $\bb{E}_{n-1}$ constructible spectrum of this rigid $\bb{E}_{n-1}$-2-ring.  To do this, we will have to analyze a little bit about certain residue fields for $R$.

We will assume for now that $n\geq 6$ to make the details of the example easier, though with more care one should be able to make it work for $n=4,5$ as well.
\begin{proposition}\label{prop:SomeHomResidueField}
The $\bb{E}_{n-2}$-$R$-algebra
\[k\coloneqq \mo{Frac}(R{\sslash}^{n-1}x\otimes_{R} R{\sslash}^{n-1}y)\]
defined by inverting all non-zero homogeneous polynomials in the $\bb{E}_{n-2}$-$R$-algebra given by the tensor product of the $\bb{E}_{n-1}$-quotients by $x$ and $y$, respectively, is a residue field for some point in the homological spectrum of $\mo{Perf}(R)$.
\end{proposition}
\begin{proof}
Since $n\geq 6$, this algebra $k$ is at least $\bb{E}_4$, so that $\mo{Perf}(k)$ is a rigid $\bb{E}_{n-3}$-2-ring (with $n-3\geq 3$), generated by the unit, which has $\pi_*(k)$ is a graded field, is in fact a tt-field.  Since $\pi_*(k)$ is indecomposable as a $\pi_*(R)$-module, the result follows by \cite[Theorem~3.1]{BalmerCameron2021}.
\end{proof}
Our goal now is to construct an $\bb{E}_{n-1}$-algebra under $R{\sslash}^{n-1}x[([x,y])^{-1}]$ which has homological support precisely the point picked out by this residue field $k$.  By symmetry, this will also construct such an $\bb{E}_{n-1}$-algebra under $R{\sslash}^{n-1}y[([x,y])^{-1}]$, which will necessarily have coproduct with the former algebra being zero.  This will show that there are at least two distinct points in the $\bb{E}_{n-1}$ constructible spectrum of $R$ whose homological support is the point $k$, which will provide the desired failure of \Cref{thm:nconsequalshom}.

Towards this end, we need to find a way to tell when an algebra has homological support exactly $k$, which we shall do by finding explicit jointly nil-conservative covers of $R$.  To begin, we note the following.
\begin{proposition}\label{prop:quotientbyx}
Let $S$ be any $\bb{E}_n$-$\bb{F}_2$-algebra, and consider a class $x\in \pi_0(S)$.  Then, the family of maps
\[\{S\to S[Q_I(x)^{-1}]\colon I \textnormal{ admissible}\}\cup\{S\to S{\sslash}^{n-1}x\}\]
is jointly nil-conservative.
\end{proposition}
\begin{proof}
Using \Cref{thm:nconsequalshom}, it suffices to test nil-conservativity with respect to $\bb{E}_{n-2}$-$S$-algebras $A$, which helps us to work in a case where we can actually talk about categories of $A$-modules as rigid $\bb{E}_{n-2}$-2-rings (recall we are assuming $n\geq 6$ for the moment).

Suppose that $A[Q_I(x)^{-1}]\simeq 0$ for all admissible sequences $I$.  Then we must have that $Q_I(x)$ acts nilpotently on $A$ for all such $I$.  In particular, the fiber of
\[A\to A/Q_I(x)\]
is $\otimes$-nilpotent for all $I$.

Using once again that 
\[S{\sslash}^{n-1}x\simeq S\otimes_{\bb{F}_2\{x\}_n}\bb{F}_2\simeq \bigotimes_I S/Q_I(x),\]
we find that we can write the $\bb{E}_{n-2}$-algebra map
\[A\to A\otimes_{S} S{\sslash}^{n-1}x\simeq \bigotimes_{I}A/Q_I(x)\]
as a filtered colimit of $\bb{E}_0$-$A$-algebra maps with $\otimes$-nilpotent fiber.  In particular, $A\to A\otimes_{S} S{\sslash}^{n-1}x$ is nil-conservative, and so if $A$ was non-zero, the target is non-zero as well.
\end{proof}
Applying this to our case with $S=R$ twice, taking both $x$ and $y$, and finding the common refinement of the resulting nil-conservative covers, we find,
\begin{corollary}\label{cor:cover1}
The family of maps
\[\{R\to R[Q_{I}(x)^{-1}]\}\cup\{R\to R[Q_{I}(y)^{-1}]\}\cup\{R\to R{\sslash}^{n-1}x \otimes_{R} R{\sslash}^{n-1} y\}\]
is jointly nil-conservative.
\end{corollary}
Finally, we will require one last little claim.
\begin{proposition}\label{prop:finalcover}
The family of weak rings under $R$ given by 
\[\{R\to R[Q_{I}(x)^{-1}]\}\cup\{R\to R[Q_{I}(y)^{-1}]\}\cup\{R\to (R{\sslash}^{n-1}x \otimes_{R} R{\sslash}^{n-1} y)/p(\bold{e})^4\}\cup \{R\to k\},\]
where $p(\bold{e})$ ranges over homogeneous polynomials in the variables $Q_{I}(e)$ with $e\in E\backslash\{x,y\}$, is jointly nil-conservative.
\end{proposition}
\begin{proof}
Take a non-zero $\bb{E}_{n-2}$-algebra $A$ under $R$.  Using \Cref{cor:cover1}, we may assume without loss of generality that 
\[A\otimes_{R} (R{\sslash}^{n-1}x \otimes_{R} R{\sslash}^{n-1}y)\neq 0.\]
Since the residue field $k$ is a localization of the algebra we are tensoring with at its non-zero homogeneous elements, we find that
\[A\otimes_{R} k\simeq 0\]
if and only if there exists some homogeneous polynomial 
\[p(\bold{e})\in \pi_*(R{\sslash}^{n-1}x \otimes_{R} R{\sslash}^{n-1}y)\]
which acts nilpotently on $A\otimes_{R} (R{\sslash}^{n-1}x \otimes_{R} R{\sslash}^{n-1}y)$.  In this case, we must then have that the fiber of 
\[A\otimes_{R}(R{\sslash}^{n-1}x \otimes_{R} R{\sslash}^{n-1} y)\to A\otimes_{R}(R{\sslash}^{n-1}x \otimes_{R} R{\sslash}^{n-1} y)/p(\bold{e})^4\]
is $\otimes$-nilpotent.  

In particular, since $A\otimes_{R}(R{\sslash}^{n-1}x \otimes_{R} R{\sslash}^{n-1} y)$ is a nonzero $\bb{E}_{n-2}$-algebra, and the target is a weak ring by \cite[Lemma~5.4]{burklund2022multiplicativestructuresmoorespectra}, the target must be non-zero as well, and the claim is shown.
\end{proof}

We are now in the position to provide our counterexample.
\begin{theorem}\label{thm:nconsequalshomoptimal}
Let $R\coloneqq \bb{F}_2\{x,y\}_n$ be the free $\bb{E}_n$-$\bb{F}_2$-algebra on two degree zero generators $x$ and $y$, for any $n\geq 6$.  Consider the $\bb{E}_{n-1}$-$R$-algebras defined as
\[C\coloneqq R{\sslash}^{n-1}(x,Q_I(y)^2)[p(\bold{e})^{-1}]\qquad \text{ and } \qquad D\coloneqq R{\sslash}^{n-1}(y,Q_I(x)^2)[p(\bold{e})^{-1}]\]
as $I$ ranges across admissible sequences and $p(\bold{e})$ ranges across homogeneous polynomials in the variables $Q_{I}(e)\in E\backslash\{x,y\}$.  These algebras are both non-zero with homological support exactly $\{k\}$, yet $C\coprod_{R}D\simeq 0$.

In particular, for any choice of Nullstellensatzian $\bb{E}_{n-1}$-$R$-algebras $L_1$ under $C$ and $L_2$ under $D$, $L_i$ has homological support $\{k\}$, but $L_1$ and $L_2$ determine different points in $\Speccons_{\bb{E}_{n-1}}(R)$.  This is all to say that the requirement $n<m$ in \Cref{thm:nconsequalshom} is necessary.
\end{theorem}
\begin{proof}
Since the bracket $[x,z]=0$ vanishes when $z=w^2$ is a square of a class, $z=Q_i(w)$ is a Dyer--Lashof operation applied to a class with $i<n-1$, and satisfies $[x,Q_{n-1}w]=[w,[w,x]]$ (see \cite[Theorem 5.2]{lawson2020enringspectradyerlashof}), it follows that the image of 
\[\bb{F}_2\{x,Q_I(y)^2\}_n\to R\]
misses the multiplicative set generated by homogeneous polynomials $p(\bold{e})$ in $Q_I(e)$ for $e\in E\backslash\{x,y\}$, and in particular $C\neq 0$.  

Since $C$ admits an algebra map from $R{\sslash}^{n-1}x [([x,y])^{-1}]$, which has pushout with $R{\sslash}^{n-1}y$ (which admits an algebra map to $D$) equal to zero, it suffices to show that $\supph(C)=\{k\}$.  Note that, by construction,
\[C[Q_I(x)^{-1}]=0, \quad C[Q_I(y)^{-1}]=0, \quad \text{ and }\quad C/p(\bold{e})\simeq 0.\]
Using \Cref{prop:finalcover}, we necessarily must have that $C\otimes k\neq 0$, and since $C\otimes L=0$ for all other weak rings $L$ in the nil-conservative cover from \Cref{prop:finalcover}, $\supph(C)\subseteq \supph(k)=\{k\}$, as desired.
\end{proof}

\newpage
\section{Geometric Points for Homological Spectra}\label{sec:geompts}
In general, we have seen one cannot expect there to be a rigid 2-ring map from a rigid 2-ring $\cat C$ to a rigid 2-ring $\cat D$ whose Balmer spectrum is a point, which picks out any given point in $\Spc(\cat C)$. 

\Cref{thm:nconsequalshom} almost provides this at the level of constructible spectra, at least if one only asks for rigid $\bb{E}_n$-2-ring maps, but suffers from the following defect: if $\cat C$ is a rigid $\bb{E}_n$-2-ring, it makes points in its homological spectrum correspond to Nullstellensatzian $\bb{E}_n$-algebras in $\Ind(\cat C)$. However, if $L$ is such a Nullstellensatzian $\bb{E}_n$-algebra, it does not have any reason to be Nullstellensatzian in $\mo{Alg}_{\bb{E}_{n-1}}(\Ind(\mo{Perf}_{\cat C}(L)))$, and so we cannot guarantee via \Cref{thm:nconsequalshom} that its homological spectrum is a point. 

As it turns out, using the methods that went into the proof of \Cref{thm:nconsequalshom}, we can fix this ``defect" by proving that, at least for $k$ large enough, the constructible spectra of $E^k_\mathfrak m$ are points, thus obtaining a sufficient supply of ``highly structured geometric points".

\subsection{Geometric points of the homological spectrum for rigid $\bb{E}_n$-2-rings}
The main theorem of this section is:
\begin{theorem}\label{thm:spechEn} Let $\cat C$ be a rigid $\bb{E}_m$-2-ring for some $m\geq 4$ (possibly $m=\infty$), and let \[\mf{m}\in \Spc^h(\cat C)\] be a homological prime.  Then, for any $3\leq n<m$, there exists a rigid $\bb{E}_n$-2-ring $\cat K$, equipped with an $\bb{E}_n$-2-ring map $\cat C\to \cat K$, such that \[\Spc^h(\cat K)=*\] is a single point, and such that the map \[\Spc^h(\cat K)\to \Spc^h(\cat C)\]
has image exactly $\{\mf{m}\}$.
\end{theorem}

To prove this theorem, we will need the following, which relies on inputs that we will delay until after proving the main theorem of this section. The term ``$v$-compatible" comes from \cite{burklund2022multiplicativestructuresmoorespectra} and will be recalled later. 
\begin{theorem}\label{thm:technicalSqZero}
Let $\cat C$ be an $\bb{E}_m$-monoidal stable $\infty$-category, for some $m\geq 3$, and suppose we are given a map $v\colon\mc{I}\to\bb{1}$ in $\cat C$ such that the cofiber $\bb{1}/v$ admits a right unital multiplication.

Then, for any, $3\leq n \leq m$, $1\leq k \leq n-1$, $q\geq n+1$, and any $w\geq q+k$ the unique $v$-compatible $\bb{E}_{n}$-algebra structures on $\bb{1}/v^{q}$ and $\bb{1}/v^{w}$ are such that the $\bb{E}_k$-algebra map given by inclusion on the left factor
\[\bb{1}/v^{q}\to \bb{1}/v^{q}\otimes \bb{1}/v^{w}\]
is a splitting of a split square-zero $\bb{E}_k$-algebra extension.
\end{theorem}
We will postpone the proof of the above, opting to first prove the main theorem of this section assuming the result.  Before this though, we note the following quick lemma.
\begin{lemma}\label{lm:sqzspech}
Let $\cat C$ be a rigid $\bb{E}_m$-ring, and let $f: R\to S$ be a map of $\bb{E}_n$-algebras in $\Ind(\cat C)$, $3\leq n\leq m$. If $f$ is a square zero extension, $\Spec^h(f)\colon\Spec^h(\mo{Perf}_{\cat C}(S))\to \Spec^h(\mo{Perf}_{\cat C}(R))$ is surjective. If $f$ is a split square zero extension, it is a bijection. 
\end{lemma}
\begin{proof}
In the case of a split square zero extension, $\Spec^h(f)$ admits a retraction and hence is injective. Thus, the second claim follows from the first. 

For the first claim, we apply \Cref{thm:surjspech}: it suffices to prove that $S\otimes_R -$ is nil-conservative. For this we note that it is conservative: if $S\otimes_R M = 0$, then since $\mo{fib}(f)$ is an $S$-module in $R$-modules, $\mo{fib}(f)\otimes_R M = 0$ and so $M=R\otimes_R M = 0$ as an extension of two $0$ modules. 
\end{proof}
\begin{remark}
The keen-eyed reader will note that in the $n=3$ case of \Cref{lm:sqzspech}, we actually only had rigid $\bb{E}_2$-rings which we applied \Cref{thm:surjspech} to, meaning our homotopy category was merely a ``braided tt-category,'' as opposed to a symmetric monoidal one.  Nevertheless, one can prove the analogue of \Cref{thm:surjspech} with essentially the same arguments as the symmetric monoidal case.  One can alternatively proceed without ever leaving the symmetric monoidal world, were it so desired, but at the cost of requiring $m\geq 5$ in the following theorem, with minimal modifications to the proof.
\end{remark}
With these preliminaries, we can prove the main theorem of this section:
\begin{proof}[Proof of \Cref{thm:spechEn}]
Fix a residue field $E_{\mf{m}}$ at the point $\mf{m}$.  A natural candidate for the category $\cat{K}$ we want to consider will be the category of modules over $E_{\mf{m}}^q$ for some $q\geq n+1$, which, as a category of modules over an $\bb{E}_{n+1}$-algebra object, inherits an $\bb{E}_n$-monoidal structure.

Towards this end, let 
\[\cat{K}_q\coloneqq \mo{Perf}_{\cat{C}}(E_{\mf{m}}^q)\]
for $q\geq n+1$.  It is clear that the image of 
\[\Spc^h(\cat{K}_q)\to \Spc^h(\cat C)\]
is exactly $\mf{m}$, by construction.  Our goal now is to show that for some $q\gg 0$, $\Spc^h(\cat{K}_q)$ is a single point.  In fact, we claim that taking $q\geq \mo{max}\{n+1,7\}$ is enough.

First, we note that using \Cref{thm:technicalSqZero}, the $\bb{E}_3$-algebra map 
\[E_{\mf{m}}^{q-3}\to E_{\mf{m}}^{q-3}\otimes E_{\mf{m}}^q\]
is a split square-zero extension of the source, and hence induces an equivalence on homological spectra by \Cref{lm:sqzspech}.  In particular, the multiplication map (which serves as a retraction for this map), given by the composite
\[E_{\mf{m}}^{q-3}\otimes E_{\mf{m}}^q\to E_{\mf{m}}^{q-3}\otimes E_{\mf{m}}^{q-3}\to E_{\mf{m}}^{q-3}\]
also induces an equivalence on homological spectra, and is thus nil-conservative.

Using this, we note also that $E_{\mf{m}}^q\to E_{\mf{m}}^{q-3}$ is nil-conservative- since any non-zero weak ring $A$ under $E_{\mf{m}}^q$ has homological support $\{\mf{m}\}$ in $\cat C$, which implies that $A\otimes E_{\mf{m}}^{q-3}$ a non-zero weak ring in $E_{\mf{m}}^q\otimes E_{\mf{m}}^{q-3}$-modules, and 
\[A\otimes_{E_{\mf{m}}^q}E_{\mf{m}}^{q-3}\simeq (A\otimes E_{\mf{m}}^{q-3})\otimes_{E_{\mf{m}}^q\otimes E_{\mf{m}}^{q-3}}E_{\mf{m}}^{q-3} \neq 0.\]

Now, fix some given $q\geq \mo{max}\{n+1,7\}$, and suppose that we had two non-zero weak rings $A,B$ in $\cat{K}_q$ such that $A\otimes_{E_{\mf{m}}^q} B\simeq 0$.  Without loss of generality, we may replace $B$ by $B\otimes_{E_{\mf{m}}^q}E_{\mf{m}}^{q-3}$ to assume that the weak ring structure on $B$ is induced from a weak $E_{\mf{m}}^{q-3}$-ring structure on an object which we abusively also denote by $B$.  Now, again since $A$ and $B$, considered as objects of $\cat{C}$, have homological support $\{\mf{m}\}$, $A\otimes B\neq 0$ in $\cat{C}$.  In particular, $A\otimes B$ is a non-zero weak ring in the category of $E_{\mf{m}}^q\otimes E_{\mf{m}}^{q-3}$-modules.  We have 
\begin{align*}
(A\otimes B)\otimes_{E_{\mf{m}}^q\otimes E_{\mf{m}}^{q-3}}E_{\mf{m}}^{q-3} & \simeq ((A\otimes_{E_{\mf{m}}^q}E_{\mf{m}}^{q-3})\otimes B)\otimes_{E_{\mf{m}}^{q-3}\otimes E_{\mf{m}}^{q-3}}E_{\mf{m}}^{q-3}\\
&\simeq (A\otimes_{E_{\mf{m}}^q}E_{\mf{m}}^{q-3})\otimes_{E_{\mf{m}}^{q-3}}B\\
&\simeq A\otimes_{E_{\mf{m}}^q}B\simeq 0,
\end{align*}
which, as $A\otimes B$ was non-zero, contradicts nil-conservativity of the multiplication map.
\end{proof}

\subsection{Recollections on deformations}
Before proving the technical lemmas used in the previous subsection, we recall a bit of background from \cite{patchkoria2025adams}, \cite{BURKLUND2025110270} and \cite{burklund2022multiplicativestructuresmoorespectra}.

\begin{proposition}[{\cite[5.34,5.37,5.47,5.60]{patchkoria2025adams}}]\label{prop:deformations}  Let $\cat E$ be an epimorphism class in an essentially small stable $\infty$-category $\cat C$.  Then there exists a stable $\infty$-category $\mc{D}(\cat C,\cat E)$, a (non-exact) functor 
\[\nu\colon\cat C\to \mc{D}(\cat C,\cat E)\]
and a natural transformation 
\[\tau\colon \Sigma\nu(\Omega-)\to \nu(-)\]
such that 
\begin{enumerate}
\item The functor $\nu$ is fully faithful.
\item The image of $\nu$ generates $\mc{D}(\cat C,\cat E)$.
\item The functor $\nu$ preserves fiber sequences $x\to y\to z$ such that the map $y\to z$ is in the epimorphism class $\cat E$.
\item  There is an exact automorphism $(-)[1]$ on $\mc{D}(\cat C,\cat E)$ determined by the property that $\nu(X)[1]\simeq \nu(\Sigma X)$.
\item  Inverting the natural transformation $\tau$ yields an (exact) functor \[(-)[\tau^{-1}]\colon\mc{D}(\cat C,\cat E)\to \cat C\] such that $(\nu(-))[\tau^{-1}]\simeq \mo{id}_{\cat C}$.
\item  If $\mc{I}$ is an $\cat{E}$-injective object of $\cat C$, then for all objects $X\in \cat C$, 
\[ [\Sigma^{-s}\nu(X),\nu(\mc{I})]\simeq 0\]
for all $s>0$.
\end{enumerate}
\end{proposition}
\begin{remark}
In contrast to \cite{patchkoria2025adams}, where they deal primarily with prestable deformations which they call $\mc{D}(\cat C, \cat E)$, we will never have reason to leave the stable world.  In particular, we implicitly work with the stable envelope (that is, the Spanier-Whitehead $\infty$-category, see \cite[Construction~C.1.1.1, Proposition~C.1.1.7, Proposition~C.1.2.2]{SAG}) of the prestable categories constructed by Patchkoria-Pstr\c{a}gowski.
\end{remark}

\begin{proposition}[{\cite[Proposition~A.11]{BURKLUND2025110270}}]\label{prop:functorialityDeformation}
Given two stable $\infty$-categories $\cat C$, $\cat D$, with epimorphism classes $\cat E$, $\cat F$, in $\cat C$ and $\cat D$ respectively, as well as an exact functor 
\[\cat C\to \cat D\]
which sends $\cat E$ to $\cat F$, we obtain an exact functor 
\[\mc{D}(\cat C,\cat E)\to \mc{D}(\cat D,\cat F)\]
compatible with $\nu$ and $(-)[\tau^{-1}]$.
\end{proposition}

\begin{definition}[{\cite[Definition~4.2]{burklund2022multiplicativestructuresmoorespectra}}]\label{def:CompatibleEpiClass}
An epimorphism class $\cat E$ on a stably $\bb{E}_n$-monoidal $\infty$-category $\cat C$ is said to be \textit{$\otimes$-compatible} if $\cat E$ is preserved under tensoring with arbitrary objects of $\cat C$.
\end{definition}

\begin{proposition}[{\cite[Proposition~4.3]{burklund2022multiplicativestructuresmoorespectra}}]\label{prop:EkCompatibleEpiClass}
If $\cat C$ is a stably $\bb{E}_n$-monoidal $\infty$-category with $\otimes$-compatible epimorphism class $\cat E$, then the deformation $\mc{D}(\cat C,\cat E)$ inherits the structure of a stably $\bb{E}_n$-monoidal category in such a way that $\nu$ and $(-)[\tau^{-1}]$ are $\bb{E}_n$-monoidal functors.
\end{proposition}

\begin{proposition}\label{prop:functorialityEkCompatibleEpiClass}
If $\cat C$ and $\cat D$ are two stably $\bb{E}_n$-monoidal $\infty$-categories equipped with $\otimes$-compatible epimorphism classes $\cat E$ and $\cat F$, together with an $\bb{E}_n$-monoidal functor 
\[\cat C \to \cat D\]
mapping $\cat E$ to $\cat F$, then the induced functor 
\[\mc{D}(\cat C,\cat E)\to \mc{D}(\cat D,\cat F)\]
can naturally be promoted to an $\bb{E}_n$-monoidal functor.
\end{proposition}
\begin{proof}
The $\bb{E}_n$-monoidal functor $\cat C \to \cat D$ yields an $\bb{E}_n$-monoidal functor of prestable $\infty$-categories
\[\mo{PSh}_{\Sigma}(\cat C,\mc{S})\to\mo{PSh}_{\Sigma}(\cat D,\mc{S})\]
between categories of spherical presheaves via left Kan extension.  The prestable deformation category $\mc{D}_{\geq 0}(\cat C,\cat E)$ of $\cat C$ with respect to $\cat E$ is given by the category of perfect sheaves on $\cat C$ with respect to the $\cat E$-epimorphism topology (see \cite[Definition~5.32]{patchkoria2025adams}).  In particular, this arises as a localization of $\mo{PSh}_{\Sigma}(\cat C,\mc{S})$, where the $\otimes$-compatibility of $\cat E$ is used to give the localization the structure of an $\bb{E}_n$-localization.

The diagram
\begin{center} 
\begin{tikzcd}
\mo{PSh}_{\Sigma}(\cat C,\mc{S})\rar\dar &\mo{PSh}_{\Sigma}(\cat D,\mc{S})\dar\\
\mc{D}_{\geq 0}(\cat C,\cat E)\rar & \mc{D}_{\geq 0}(\cat D,\cat F)
\end{tikzcd}
\end{center}
is induced from a monoidal localization of an $\bb{E}_n$-monoidal map of $\bb{E}_n$-monoidal categories, and so the bottom map is itself $\bb{E}_n$-monoidal.  The result now follows by passing to stabilizations.
\end{proof}

We now specialize to the main case of interest.  Let $\cat C$ be a stably $\bb{E}_n$-monoidal $\infty$-category, let $v\colon\mc{I}\to \bb{1}$ be a map from some object to the unit such that the cofiber $\bb{1}/v$ admits a weak ring structure.  In this case, there is a $\otimes$-compatible epimorphism class on $\cat C$, denoted $\cat{E}(v)$, consisting of those morphisms which are split surjective after tensoring with $\bb{1}/v$.  In this case, the fiber sequence
\[\bb{1}\to \bb{1}/v\to \Sigma \mc{I}\]
has the second map in the epimorphism class.  In particular, this gives rise to a fiber sequence 
\[\nu(\bb{1})\to \nu(\bb{1}/v)\to \nu(\mc{I})[1].\]
Following the conventions set in \cite{burklund2022multiplicativestructuresmoorespectra}, we name the fiber of the first map here 
\[\tilde{v}\colon \Sigma^{-1}\nu(\mc{I})[1]\to \nu(\bb{1}).\]

Before restating the main result of \cite{burklund2022multiplicativestructuresmoorespectra}, we recall one final definition.
\begin{definition}[{\cite[Definition~5.1]{burklund2022multiplicativestructuresmoorespectra}}]\label{def:vcompatibleLift}
If we are given a weak ring $\bb{1}/v$ in a stably $\bb{E}_n$-monoidal $\infty$-category $\cat C$, then for $1\leq k\leq n$, an $\bb{E}_k$-algebra structure on $\bb{1}/v^{q}$ is said to be \textit{$v$-compatible} if it arises as the $\tau$-inversion of an $\bb{E}_k$-algebra structure on $\nu(\bb{1})/\tilde{v}^q$ in $\mc{D}(\cat C, \cat{E}(v))$.
\end{definition}
\begin{theorem}[{\cite[Theorem~5.2]{burklund2022multiplicativestructuresmoorespectra}}]\label{thm:RobertThm}
Suppose we are given a weak ring $\bb{1}/v$ in a stably $\bb{E}_n$-monoidal $\infty$-category $\cat C$.  Then for any $1\leq k\leq n$, and any $q\geq k+1$, there is a unique $v$-compatible $\bb{E}_k$-algebra structure on $\bb{1}/v^q$.
\end{theorem}

This theorem is proved through defining an obstruction theory for $\bb{E}_n$-algebra structures on cofibers.  We will need to make use of the explicit obstruction theory, summarized by the following statement.
\begin{theorem}[{\cite[Proposition~2.4, Remark~2.5, Corollary~2.7]{burklund2022multiplicativestructuresmoorespectra}}]\label{thm:obstruction-theory}  Let $\cat C$ be a stably $\bb{E}_m$-monoidal $\infty$-category, and consider a cofiber sequence 
\[\mc{I}\xrightarrow{v}\bb{1}\to \bb{1}/v\]
involving the unit.  For any fixed $1\leq k\leq m$, there is a sequence of inductively defined classes 
\[ \theta_{r,\alpha}\in \pi_0\mo{Hom}_{\cat C}(\Sigma^{-2-k-c_{\alpha}}(\Sigma^{k+1}\mc{I})^{\otimes r},\bb{1}/v)\]
for $r\geq 2$ and $0\leq c_{\alpha}\leq (r-1)(k-1)$ which provide obstructions to the existence of an $\bb{E}_k$-algebra structure on $\bb{1}/v$.

Moreover, given that an $\bb{E}_k$-algebra structure can be defined on $\bb{1}/v$, there exists a sequence of inductively defined classes 
\[\gamma_{r,\alpha}\in \pi_0\mo{Hom}_{\cat C}(\Sigma^{-1-k-c_{\alpha}}(\Sigma^{k+1}\mc{I})^{\otimes r},\bb{1}/v)\]
for $r\geq 1$ and $0\leq c_{\alpha}\leq (r-1)(k-1)$ which provide obstructions to the uniqueness of the $\bb{E}_k$-algebra structure on $\bb{1}/v$.
\end{theorem}

\subsection{Technical lemmas}
We now proceed with the ingredients that went into the proof of \Cref{thm:technicalSqZero}.  For the remainder of this section, fix the following data
\begin{enumerate}
\item  A small, idempotent complete $\bb{E}_m$-monoidal stable $\infty$-category $\cat C$ for some $m\geq 3$.
\item  A map $v\colon\mc{I}\to\bb{1}$ in $\cat C$ such that $\bb{1}/v$ admits a right unital multiplication.  
\item  Integers $n,k,q,w$ with $3\leq n \leq m$, $1\leq k\leq n-1$, $q\geq n+1$, and $w\geq q+k$.
\end{enumerate}
Using the epimorphism class $\mc{E}$ of $\bb{1}/v$-split epimorphisms in $\cat C$, we will work in the deformation category $\mc{D}(\cat C,\mc{E})$ with respect to this epimorphism class.
\begin{notation}  For the remainder of this section, fix the following notation.
\begin{enumerate}
\item  We will write \[\widetilde{\mc{I}}\coloneqq \Sigma^{-1}\nu(\Sigma \mc{I})\] for the object written on the right in $\mc{D}(\cat C,\mc{E})$.
\item  Similarly, we denote by 
\[\tilde{v}\colon \widetilde{\mc{I}}\to \bb{1}\]
the fiber of the map $\bb{1}\simeq \nu(\bb{1})\to \nu(\bb{1}/v)$, (which has the stated source since the map it is a fiber of is a $\bb{1}/v$-split monic).
\end{enumerate}
\end{notation}
To begin, let's prove the following easy helper lemma.
\begin{lemma}\label{lem:DeformationRequiredBounds}  The object $\bb{1}/\tilde{v}^q\otimes \bb{1}/\tilde{v}^w$ has 
\[\pi_0\mo{Hom}(\Sigma^{-s}\nu(X),\bb{1}/\tilde{v}^q\otimes \bb{1}/\tilde{v}^w)\simeq 0\]
for all $s\geq q+w-1$ and $X\in\cat C$.
\end{lemma}
\begin{proof}
Since $\bb{1}/\tilde{v}^q$ admits a weak ring structure and $w\geq q$, we have that 
\[\bb{1}/\tilde{v}^q\otimes \bb{1}/\tilde{v}^w\simeq \bb{1}/\tilde{v}^q\oplus \Sigma \widetilde{\mc{I}}^{\otimes w}\otimes \bb{1}/\tilde{v}^q,\]
and we reduce the claim to proving it for the individual summands.  For the first summand, the claim follows in fact for $s\geq q$ by \cite[Lemma~4.8]{burklund2022multiplicativestructuresmoorespectra}.

Moving onto the second summand, we proceed exactly as in the proof of \cite[Lemma~4.8]{burklund2022multiplicativestructuresmoorespectra}, working by induction on $q$.  When $q=1$, 
\[\Sigma\widetilde{\mc{I}}^{\otimes w}\otimes \nu(\bb{1}/v)\simeq \Sigma^{1-w}\nu(\Sigma^w\mc{I}^{\otimes w}\otimes \bb{1}/v)\]
is a $1-w$-shift of an $\mc{E}$-injective object, and the claim follows by \Cref{prop:deformations}(f).  In general, using the fiber sequence
\[\Sigma \widetilde{\mc{I}}^{\otimes w}\otimes \widetilde{\mc{I}}^{\otimes q-1}\otimes \bb{1}/\tilde{v}\to \Sigma\widetilde{\mc{I}}^{\otimes w}\otimes \bb{1}/\tilde{v}^q\to \Sigma\widetilde{\mc{I}}^{\otimes w}\otimes \bb{1}/\tilde{v}^{q-1},\]
together with induction and the fact that the leftmost object is a $2-w-q$-shift of an $\mc{E}$-injective object, the claim follows.
\end{proof}
Now, we can prove.
\begin{proposition}\label{prop:SqZeroThmInput}  There is a unique $\bb{E}_k$-$\bb{1}/\tilde{v}^q$-algebra structure on the object $\bb{1}/\tilde{v}^q\otimes \bb{1}/\tilde{v}^w$ in $\mc{D}(\cat C,\mc{E})$.  In particular, the $\bb{E}_k$-algebra map given by inclusion on the left factor 
\[\bb{1}/\tilde{v}^q\to \bb{1}/\tilde{v}^q\otimes \bb{1}/\tilde{v}^w\]
admits the structure of a splitting for a choice of split square-zero $\bb{E}_k$-$\bb{1}/\tilde{v}^q$-algebra structure on the target.
\end{proposition}
\begin{proof}
We note that the fiber sequence
\[\bb{1}/\tilde{v}^q\otimes \widetilde{\mc{I}}^{\otimes w}\to \bb{1}/\tilde{v}^q\to \bb{1}/\tilde{v}^q\otimes \bb{1}/\tilde{v}^w,\]
allows us to present $\bb{1}/\tilde{v}^q\otimes \bb{1}/\tilde{v}^w$ as a quotient of the unit (in $\bb{1}/\tilde{v}^q$-modules) by the ideal ``$\bb{1}/\tilde{v}^q\otimes \widetilde{\mc{I}}^{\otimes w}$'').

Since there is an $\bb{E}_k$-$\bb{1}/\tilde{v}^q$-algebra structure on $\bb{1}/\tilde{v}^q\otimes \bb{1}/\tilde{v}^w$, \Cref{thm:obstruction-theory} applies to give a sequence of inductively defined obstructions to this $\bb{E}_k$-algebra structure to be unique.  These obstructions take values in 
\begin{equation}\label{eqn:ouchie}
\pi_0\mo{Hom}_{\bb{1}/\tilde{v}^q-\mo{Mod}(\mc{D}(\cat C,\mc{E}))}(\Sigma^{-1-k-c_{\alpha}}(\Sigma^{k+1}\bb{1}/\tilde{v}^q\otimes \widetilde{\mc{I}}^{\otimes w})^{\otimes_{\bb{1}/\tilde{v}^q} r},\bb{1}/\tilde{v}^q\otimes \bb{1}/\tilde{v}^w),
\end{equation}
for $r\geq 1$ and $0\leq c_{\alpha}\leq (k-1)(r-1)$.

When $r=1$, just as in the proof of \cite[Theorem~5.2]{burklund2022multiplicativestructuresmoorespectra}, the obstruction is the composite 
\[\bb{1}/\tilde{v}^q\otimes \widetilde{\mc{I}}^{\otimes w}\to \bb{1}/\tilde{v}^q\to \bb{1}/\tilde{v}^q\otimes \bb{1}/\tilde{v}^w,\]
which we choose a null-homotopy of in order to make this into a fiber sequence.

Using tensor-hom, we may rewrite the group in \cref{eqn:ouchie} as
\[\pi_0\mo{Hom}_{\mc{D}(\cat C,\mc{E})}(\Sigma^{-1-k-c_{\alpha}}(\Sigma^{k+1}\widetilde{\mc{I}}^{\otimes w})^{\otimes r},\bb{1}/\tilde{v}^q\otimes \bb{1}/\tilde{v}^w).\]
Expanding out the definition of the source, we find that we can write this as 
\[\Sigma^{-1-k-c_{\alpha}}(\Sigma^{k+1}\widetilde{\mc{I}}^{\otimes w})^{\otimes r}\simeq\Sigma^{-1-k-c_{\alpha}+r(k+1)-rw}\nu(\Sigma^{rw}\mc{I}^{\otimes rw}),\]
where, as $c_{\alpha}\geq 0$, and $k+1\leq q+k\leq w$, allows us to write
\[-1-k+r(k+1)-c_{\alpha}-rw\leq (r-1)(k+1)-rw\leq k+1-2w\leq 1-w-q.\]
In particular, we see that the group \cref{eqn:ouchie} vanishes by \Cref{lem:DeformationRequiredBounds}.  Since all of the obstructions vanish, this implies the $\bb{E}_k$-$\bb{1}/\tilde{v}^q$-algebra structure on $\bb{1}/\tilde{v}^q\otimes \bb{1}/\tilde{v}^w$ is unique, as desired.

The final claim follows since we can write \[\bb{1}/\tilde{v}^q\otimes \bb{1}/\tilde{v}^w\simeq \bb{1}/\tilde{v}^q\oplus \Sigma \bb{1}/\tilde{v}^q\otimes \widetilde{\mc{I}}^{\otimes w},\]
and give this the structure of a split $\bb{E}_k$-$\bb{1}/\tilde{v}^q$-algebra, which must then agree with any other choice of $\bb{E}_k$-$\bb{1}/\tilde{v}^q$-algebra structure by uniqueness.
\end{proof}

\begin{proof}[Proof of \Cref{thm:technicalSqZero}]  Using \Cref{prop:SqZeroThmInput}, we find that $\bb{1}/\tilde{v}^q\otimes \bb{1}/\tilde{v}^w$ is a split square-zero extension of $\bb{1}/\tilde{v}^q$ in $\mc{D}(\cat C,\mc{E})$, and the result follows from inverting $\tau$ to pass back to $\cat C$.
\end{proof}

\newpage
\appendix
\section{The Semi-Simplicity of $\mo{Rep}(GL_t)$}\label{sec:appendix}
The goal of this appendix is to give a brief sketch of Deligne's proof of the semi-simplicity of $\mo{Rep}(GL_t)$ from \cite{deligne}, for the reader's convenience. To keep things simple, we will actually only prove a special case of Deligne's result (which is sufficient for our main results), namely that when $t$ is specialized to an element which is not an \emph{algebraic} integer. The generalization to $t$ not an integer involves a more precise analysis, which would take us too far afield. 

For our purposes, and by way of \cite[Proposition 10.3]{deligne}, we may define $\mathrm{Rep}(GL_t)$ as follows (though this is not Deligne's original definition): 
\begin{definition}\label{defn:repglt}
Deligne's category $\mathrm{Rep}(GL_t)$ is the the free idempotent complete additively symmetric monoidal category on a dualizable object $x$. 

The endomorphism ring of the unit is $\bb{Z}[t]$, and for every commutative ring $A$ equipped with a choice of element $\delta\in A$, corresponding to a map $\bb{Z}[t]\to A$, we let $\mathrm{Rep}(GL_t;A)$ denote the basechange of $\mathrm{Rep}(GL_t)$ along $\bb{Z}[t]\to A$. 
\end{definition}
By \Cref{thm:cobhyp}, it follows that $\mathrm{Rep}(GL_t)$ is the idempotent completion of $\bb{Z}[\ho(\mo{Cob})]$, which explains Deligne's original definition \cite[Definition 10.2]{deligne}. 

In our language, Deligne's result (in the restricted generality described above) is the following:\
\begin{theorem}[{Deligne, \cite[Théorème 10.5]{deligne}}]\label{thm:Deligne}
Let $K$ be a field of characteristic $0$ and $t\in K$ be an element which is not an algebraic integer. The category $C_K :=\mathrm{Rep}(GL_t;K)$ is abelian semi-simple. 
\end{theorem}
Recall that $\ho(\mo{Cob})$ consists of objects $\X{i}{j} \coloneqq X^{\otimes i}\otimes X^{\vee,\otimes j}$ where $X$ is the universal dualizable object, and therefore $C_K$ is generated under finite direct sums and retracts by the Yoneda images of these generators. 

We organize what we need about these generators in the following lemma:

\begin{lemma}\label{lm:generatorsCobK}
    Let $i,j,r,s\in \NN$. 
    \begin{itemize}
        \item If $i-j\neq r-s$, then $\Hom_{\mo{Cob}}(\X{i}{j},\X{r}{s})=\emptyset$ and so $\Hom_{C_K}(\X{i}{j},\X{r}{s}) = 0$; 
        \item If $i-j = r-s$ and $i\leq r$, then there are maps $\X{i}{j}\to \X{r}{s}, \X{r}{s}\to \X{i}{j}$ whose composite is the identity of $\X{i}{j}$ multiplied with a number of circles in $\mo{Cob}$, and so, the identity multiplied by some power of $t$ in $C_K$. In particular, in $C_K$, $\X{i}{j}$ is a retract of $\X{r}{s}$. 
    \end{itemize}
\end{lemma}
\begin{proof}
    The first is visible in the cobordism category: a cobordism out of $\X{i}{j}$ consists of circles which do not change $i-j$, maps from $+$ to $+$ or $-$ to $-$ which also do not change $i-j$, and finally evaluations or coevaluations which also do not change $i-j$.   

    For the second one, it suffices to prove that these maps exist for $\bb{1}$ and $\X{k}{k}$, but then one can simply pick evaluation and coevaluation, where the composite is $k$ circles. 

    Since $t$ is invertible in $K$, the retraction claim follows. 
\end{proof}

We now organize how semi-simplicity behaves with orthogonality and retracts in the following elementary lemma: 

\begin{lemma}\label{lm:ssreduction}
    Let $C$ be an additive $1$-category. 
    \begin{enumerate}
        \item Suppose $x\in C$ is such that $\End_C(x)$ is semi-simple. Then the same holds for $\oplus_n x$ for any $n\geq 0$; 
        \item Suppose $x,y\in C$ are such that $\End_C(x)$ is semi-simple, and $y$ is a retract of $x$. Then $\End_C(y)$ is semi-simple. 
        \item Suppose $x,y\in C$ are such that $\End_C(x),\End_C(y)$ are semi-simple and that $\hom_C(x,y)=\hom_C(y,x) = 0$. Then $\End_C(x\oplus y)$ is semi-simple. 
    \end{enumerate}
\end{lemma}
\begin{proof}
    1. $\End_C(\oplus_n x)\cong M_n(\End_C(x))$ and semi-simple rings are closed under matrix rings. 

    2. Without loss of generality, $C$ is idempotent-complete and generated under finite direct sums and retracts by $x$. Thus it is equivalent to $\mathrm{Proj}^{\mathrm{ f.g.}}_{\End_C(x)}$ and $\End_C(x)$ is a product of matrix rings over division algebras by Artin--Wedderburn. In such a category, the endomorphism ring of any object is a product of matrix rings, and hence is semi-simple, which proves the claim. 

    3. The conditions guarantee that $\End_C(x\oplus y)\cong \End_C(x)\times\End_C(y)$ and semi-simple rings are closed under finite products. 
\end{proof}

By induction, the above two lemmas reduce us to interrogating when $\End(\X{i}{j})$ is semi-simple for a single $(i,j)$. We need some base case of the induction to prove that certain things are semi-simple. The following is the key lemma: 
\begin{lemma}\label{lm:perfss}
    Suppose $x,y\in C$ are in a rigid additive idempotent-complete category whose unit has a field $L$ of endomorphisms, and $\End_C(x)$ is semi-simple. Assume all hom's in $C$ are finite dimensional over $L$.

    Finally, assume the pairing $\hom_C(x,y)\otimes \hom_C(y,x)\to \End_C(x)$ is a perfect pairing and witnesses $\hom_C(x,y)$ and $\hom_C(y,x)$ as $\End_C(x)$ dual to one another. 

    Then $y$ splits as $\hom_C(x,y)\otimes_{\End_C(x)} x \oplus z$ where $\hom_C(x,z)= 0, \hom_C(z,x) = 0$. 
\end{lemma}
\begin{proof}
By combining finite dimensionality and semi-simplicity of $\End_C(x)$ with idempotent-completeness of $C$, the tensor product $\hom_C(x,y)\otimes_{\End_C(x)}x$ exists in $C$, and it comes with a canonical map $\hom_C(x,y)\otimes_{\End_C(x)}x\to y$. 

Similarly, we have a map $y\to \hom_{\End_C(x)}(\hom_C(y,x), x)$. If we can prove that the composite 
\[
    \hom_C(x,y)\otimes_{\End_C(x)}x\to y\to \hom_{\End_C(x)}(\hom_C(y,x), x)
\]
is an isomorphism, we will be done since this will be the desired retraction and the properties of $z$ are then automatic.  

Both source and target are in the thick additive category generated by $x$, so this map being an isomorphism can be checked by applying $\hom_C(x,-)$, where the map becomes $\hom_C(x,y)\to \hom_{\End_C(x)}(\hom_C(y,x),\End_C(x))$, which is an isomorphism by assumption. 
\end{proof}

We are now equipped to prove \Cref{thm:Deligne}:
\begin{proof}[Proof of \Cref{thm:Deligne}]
By \Cref{lm:generatorsCobK} and \Cref{lm:ssreduction}, we reduce by induction to proving that $\End_{C_K}(\X{i}{j})$ is semi-simple. We prove that statement by induction as well,

The induction hypothesis is both that $\End_{C_K}(\X{r}{s})$ is semi-simple for smaller $r,s$, and also that the trace pairing $\End_{C_K}(\X{r}{s})\otimes_K\End_{C_K}(\X{r}{s})\to K$ given by multiplication followed by the trace map\footnote{The trace map coming from dualizability in $C_K$, not the usual trace map of a finite dimensional $K$-algebra. } $\End_{C_k}(\X{r}{s})\to K$ is perfect. 

If either $i=0$ or $j=0$, $\End_{C_K}(\X{i}{j})$ is a group algebra $K[\Sigma_j]$ or $K[\Sigma_i]$, and is therefore semi-simple since $K$ is characteristic $0$. 

Furthermore, the trace pairing is perfect in this case: indeed,  by direct calculation, for any $\sigma\in \Sigma_i$, $\tr(\sigma) = t^{n(\sigma)}$ where $n(\sigma)$ is the number of cycles in $\sigma$, see \cite[Lemma 4.7]{EndTHH}, examplified below in a drawing.

\[
\begin{tikzpicture}
\draw[black, very thick] (-4,0.5) rectangle (4,5);
\path (-3.25,2.75)	node	[black]	{$\emptyset$}
	(-1.75,4.5)	node	[black]	{$+$}
	(-1.75,4)	node	[black]	{$-$}
	(-1.75,3.5)	node	[black]	{$+$}
	(-1.75,3)	node	[black]	{$-$}
	(-1.75,2.5)	node	[black]	{$+$}
	(-1.75,2)	node	[black]	{$-$}
	(-1.75,1.5)	node	[black]	{$+$}
	(-1.75,1)	node	[black]	{$-$}
	(3.25,2.75)	node	[black]	{$\emptyset$}
	(1.75,4.5)	node	[black]	{$+$}
	(1.75,4)	node	[black]	{$-$}
	(1.75,3.5)	node	[black]	{$+$}
	(1.75,3)	node	[black]	{$-$}
	(1.75,2.5)	node	[black]	{$+$}
	(1.75,2)	node	[black]	{$-$}
	(1.75,1.5)	node	[black]	{$+$}
	(1.75,1)	node	[black]	{$-$};
\draw[thick, ->] (-1.9,4) arc (270:90:0.25);
\draw[thick, ->] (-1.9,3) arc (270:90:0.25);
\draw[thick, ->] (-1.9,2) arc (270:90:0.25);
\draw[thick, ->] (-1.9,1) arc (270:90:0.25);	
\draw[thick, <-] (1.9,4) arc (90:270:-0.25);
\draw[thick, <-] (1.9,3) arc (90:270:-0.25);
\draw[thick, <-] (1.9,2) arc (90:270:-0.25);
\draw[thick, <-] (1.9,1) arc (90:270:-0.25);	
\draw[thick, ->] (-1.60,4.5) .. controls (-1,4.5) and (1,3.5) .. (1.60,3.5);
\draw[thick, ->] (-1.60,3.5) .. controls (-1,3.5) and (1,2.5) .. (1.60,2.5);
\draw[thick, ->] (-1.60,2.5) .. controls (-1,2.5) and (1,4.5) .. (1.60,4.5);
\draw[thick, ->] (-1.60,1.5) -- (1.60,1.5);
\draw[thick, <-] (-1.60,4) -- (1.60,4);
\draw[thick, <-] (-1.60,3) -- (1.60,3);
\draw[thick, <-] (-1.60,2) -- (1.60,2);
\draw[thick, <-] (-1.60,1) -- (1.60,1);
\path (0,0)	node	[black]	{The trace of the endomorphism in $K[\Sigma_4]$ given by the 3-cycle $\sigma=(123)$.};
\end{tikzpicture}
\]
Up to replacing the pairing $(\sigma,\tau)\mapsto \tr(\sigma\tau)$ by $(\sigma,\tau)\mapsto \tr(\sigma^{-1}\tau)$ (which is degenerate exactly if the other one is) the corresponding matrix therefore has $t^i$'s on the diagonal and $t^k, k <i$ elsewhere. Therefore, expanding out the determinant of this matrix we find a monic polynomial with integer coefficients of degree $ii!$ in the variable $t$, which is nonzero since $t$ is not an algebraic integer.

We now want to check the hypotheses of \Cref{lm:perfss} for $x=\X{i-1}{j-1}, y = \X{i}{j}$. Since the trace pairing for $\End_{C_K}(x)$ is assumed to be perfect by induction, the pairing $\Hom_{C_K}(x,y)\otimes_K\Hom_{C_K}(y,x)\to \End_{C_K}(x)$ is perfect if and only if its composition with the trace pairing is perfect.

But now the trace pairing $\Hom_{C_K}(\X{i-1}{j-1},\X{i}{j})\otimes_K \Hom_{C_K}(\X{i}{j},\X{i-1}{j-1})\to K$ is equal, under the obvious identifications, to the trace pairing on $\End_{C_K}(\X{i-1}{j})$, which is perfect by induction. 

So we can apply the lemma and find that 
\[
    \X{i}{j}= \Hom_{C_K}(\X{i-1}{j-1},\X{i}{j})\otimes_{\End_{C_K}(\X{i-1}{j-1})}\X{i-1}{j-1}\oplus Z
\]
for some $Z$ orthogonal to $\X{i-1}{j-1}$. 

In particular, $\End_{C_K}(Z)$ is the quotient of $\End_{C_K}(\X{i}{j})$ by the ideal of morphisms factoring through $\X{i-1}{j-1}$, which is clearly isomorphic to $K[\Sigma_i \times \Sigma_j]$ and is in particular semi-simple. 

On this term, the trace pairing is $(\sigma, \tau)\mapsto t^{n(\sigma) n(\tau)}$, which is perfect for the same reason as above: up to a change of basis, the corresponding matrix has $t^{ij}$ on the diagonal, and $t^k, k<ij$ away from the diagonal so that again, perfectness is guaranteed by $t$ not being an algebraic integer. 
\end{proof}
\begin{corollary}\label{cor:DeligneParityArgument}
Suppose $K=\bb{Q}(t)$ is a function field in one variable $t$.  For any simple object $Z$ in $C_K$ arising as a summand of an object of the form $\X{i}{j}$, with $i$ minimal with this property.  Then the dimension of $Z$ is a polynomial with degree $i+j$ with positive leading coefficient.  In particular, the parity of the degree of the dimension of $Z$ is equal to the parity of $i-j$.
\end{corollary}
\begin{proof}
The object $\X{i}{j}$ has dimension $t^{i+j}$, and all of its summands of the form $\X{r}{s}$ with $r<i$ have dimension a polynomial of degree strictly less than $i+j$.  In particular, killing off all summands of $\X{i}{j}$ which were summands of $\X{r}{s}$ with $r<s$, the resulting object has dimension a monic polynomial of degree $i+j$.  The proof of \Cref{thm:Deligne} shows that the endomorphism algebra of $\X{i}{j}$ modulo all of these summands agrees with the group algebra $K[\Sigma_{i}\times \Sigma_j]$.  In particular, simple summands correspond to minimal idempotents in this group algebra, and in turn to irreducible representations of $\Sigma_i\times \Sigma_j$.  Taking such an idempotent $\sum_{\sigma\in \Sigma_i\times \Sigma_j}a_{\sigma}\sigma$ (with $\sigma\in\bb{Q}$), one sees that the trace of this endomorphism, which is the dimension of the simple summand it picks out, is a polynomial of degree $i+j$, with coefficient of $t^{i+j}$ equal to the coefficient of the identity element in this expansion for our idempotent.  If the idempotent corresponded to an irreducible representation of dimension $d$, then this leading coefficient is exactly $d/|\Sigma_i\times \Sigma_j| = d/(i! j!)$, which is in particular positive, as claimed.
\end{proof}
\newpage
	\printbibliography
\end{document}